\documentclass{amsart}

\usepackage{amsmath,amssymb,amsthm, mathrsfs, mathtools, bm, amsfonts}
\usepackage{mathabx}
\usepackage{amscd,mathtools}
\usepackage[utf8]{inputenc}
\usepackage[english]{babel}
\usepackage{cite}
\usepackage{color}
\usepackage[colorlinks,citecolor=blue,linkcolor=red]{hyperref}
\usepackage{hyperref}
\usepackage{float}
\usepackage{tikz}
\usepackage{lmodern}
\usetikzlibrary{decorations.pathmorphing}

\usepackage{multicol}
\tikzset{snake it/.style={decorate, decoration=snake}}
\usetikzlibrary{automata}
\usetikzlibrary{cd}
\usepackage[margin=3cm]{geometry}

\usepackage{comment}
\usepackage{mathtools}

\usepackage{csquotes}

\usepackage{amsthm}

\usepackage{tikz}
\usetikzlibrary{calc}
\usetikzlibrary{arrows}
\usetikzlibrary{decorations.pathreplacing}
\usetikzlibrary{intersections}

\usepackage{tikz}
\usepackage{caption, subcaption}
\usepackage{tikz-cd}
\usetikzlibrary{calc,decorations.pathreplacing}
\usepackage{tikz}
\usetikzlibrary{calc}
\usetikzlibrary{arrows}
\usetikzlibrary{decorations.pathreplacing}
\usetikzlibrary{intersections}

 \newtheorem{theorem}{Theorem}[section]
  \newtheorem{proposition}[theorem]{Proposition}
    
      \newtheorem{construction}[theorem]{Construction}
         \newtheorem{assumption}[theorem]{Assumption}

  \newtheorem{corollary}[theorem]{Corollary}
  \newtheorem{lemma}[theorem]{Lemma}

  \newtheorem{question}[theorem]{Question}
  \newtheorem{axiom}[theorem]{Axiom}

  \theoremstyle{definition}
  \newtheorem{definition}[theorem]{Definition}
  \newtheorem{claim}[theorem]{Claim}
  
  \newtheorem*{claim*}{Claim}

    \newtheorem{notation}[theorem]{Notation}

  \newtheorem*{question*}{Question}
  \newtheorem*{answer*}{Answer}
  \newtheorem*{application*}{Application}

  \theoremstyle{remark}
  \newtheorem{remark}[theorem]{Remark}
  \newtheorem*{remark*}{Remark}

\usepackage{amssymb}

\DeclarePairedDelimiterX{\Norm}[1]{\lVert}{\rVert}{#1}
 \newcommand{\from}{\colon\thinspace} 
\theoremstyle{definition}

  \newcommand{\sQ}{{\sf Q}}

  \renewcommand{\aa}{{\sf a}}   
     
  \newcommand{\cc}{{\sf c}}   
  \newcommand{\dd}{{\sf d}}

  \newcommand{\mm}{{\sf m}}   
  \newcommand{\nn}{{\sf n}}

  \newcommand{\qq}{{\sf q}}

  \newcommand{\gothic}{\mathfrak}
  \newcommand{\go}{{\gothic o}}

  \newcommand{\calB}{\mathcal{B}}
  \newcommand{\calC}{\mathcal{C}}

  \newcommand{\calH}{\mathcal{H}}
  
  \newcommand{\calJ}{\mathcal{J}}

  \newcommand{\calN}{\mathcal{N}}

  \newcommand{\calU}{\mathcal{U}}

  \newcommand{\calX}{\mathcal{X}}

  \newcommand{\CAT}{\mathrm{CAT}(0)}

    \newcommand{\Mod}{\mathcal{MCG}}
      \newcommand{\Teich}{\mathrm{Teich}}

\newcommand{\diam}{\mathrm{diam}}
\newcommand{\hull}{\mathrm{hull}}
\newcommand{\med}{\mathfrak{m}}
\newcommand{\gate}{\mathfrak{g}}

\newcommand{\nest}{\sqsubset}

\newtheorem{thmi}{Theorem}
\newtheorem{cori}[thmi]{Corollary}

\newtheorem{questioni}{Question}

  \newcommand{\ST}{\mathbin{\Big|}} 
 
\begin{document}

\title[LQC and the sublinear Morse boundary of MCG and Teich]{The geometry of genericity in mapping class groups and Teichm\"uller spaces via CAT(0) cube complexes
}


 \author   {Matthew Gentry Durham}
 \address{Department of Mathematics, University of California, Riverside, CA }
 \email{mdurham@ucr.edu}
 
  \author   {Abdul Zalloum}
 \address{Department of Mathematics, Queen's University, Kingston, ON }
 \email{az32@queensu.ca}


\begin{abstract}
Random walks on spaces with hyperbolic properties tend to sublinearly track geodesic rays which point in certain hyperbolic-like directions.  Qing-Rafi-Tiozzo recently introduced the sublinearly Morse boundary and proved that this boundary is a quasi-isometry invariant which captures a notion of generic direction in a broad context.

In this article, we develop the geometric foundations of sublinear Morseness in the mapping class group and Teichm\"uller space. We prove that their sublinearly Morse boundaries are visibility spaces and admit continuous equivariant injections into the boundary of the curve graph. Moreover, we completely characterize sublinear Morseness in terms of the hierarchical structures of these spaces.

Our techniques include developing tools for modeling the hulls of median rays in hierarchically hyperbolic spaces via CAT(0) cube complexes.  Part of this analysis involves establishing direct connections between the geometry of the curve graph and the combinatorics of hyperplanes in the approximating cube complexes.

\end{abstract}

\maketitle



\section{Introduction}

The \emph{sublinearly Morse boundary} was introduced by Qing-Rafi-Tiozzo \cite{QRT19, QRT20} to capture the asymptotic behavior of random walks on finitely generated groups with hyperbolic-like properties. Building on work of Kaimanovich \cite{kaimanovich2000poisson}, Tiozzo \cite{Tiozzo_tracking},  Sisto \cite{Sisto_tracking}, Mathieu-Sisto \cite{MS_acyl}, Sisto-Taylor \cite{ST_random}, among others, they prove that the $\log^p$-Morse boundary (for some $p=p(S)>0$) of the mapping class group of a finite-type surface $S$ is a model for the Poisson boundary for its random walks, and that random walks $\log^p$-track $\log^p$-Morse geodesics (see also \cite{NQ22}).  In this sense, $\log^p$-Morse geodesics are \emph{generic directions} in the mapping class group.  This boundary has subsequently been studied in the context of CAT(0) cube complexes \cite{Murray-Qing-Zalloum, IMZ21}, where the picture is more transparent than the general context and its details relevant to us.

In this article, we establish various connections between the sublinearly Morse boundary of the mapping class group and the hierarchical geometry of the curve graph.  More generally, we develop several foundational properties of the sublinearly Morse boundaries of \emph{hierarchically hyperbolic spaces} (abbreviated as HHSes), of which mapping class groups and Teichm\"uller spaces are examples.  

Our arguments exploit the cubical geometry of HHSes, building on work of Behrstock-Hagen-Sisto \cite{HHS_quasi}. They showed that the median hull of a finite set of points in an HHS is approximated by a CAT(0) cube complex via a median quasi-isometry.  We develop a limiting version of their machinery (Theorem \ref{thm:cube model, informal}) which is suitable for studying the median hulls of median quasi-geodesic rays, allowing us to analyze the asymptotic geometry of the curve graph via the combinatorics of hyperplanes in cube complexes (Theorem \ref{thm:curves and cubes intro}). These quasicubical results are separate from the $\kappa$-Morseness discussion and are therefore of independent interest.

For a sublinear function $\kappa$, a quasi-geodesic ray $q$ in an HHS $\calX$ has a \emph{$\kappa$-persistent shadow} in its top-level hyperbolic space $\calC(S)$ if $d_S(q(s), q(t)) \succ \frac{t-s}{\kappa(t)}$ for all $s,t \in [0,\infty)$ with $t \geq s$. The following theorem summarizes our main results about sublinearly Morse rays in HHSes:

\begin{thmi}\label{thm:main informal}
Let $\calX$ be a proper hierarchically hyperbolic space with unbounded products, $\calC(S)$ be the top-level hyperbolic space, $\kappa$ be a sublinear function and let $\partial_\kappa \calX$ denote the $\kappa$-Morse boundary. We have the following.

\begin{enumerate}
\item The subsurface projection map $\pi_S: \calX \rightarrow \calC(S)$ induces an $\text{Aut}(\calX)$-equivariant continuous injection, $i_{\kappa}:\partial_\kappa\calX \rightarrow \partial \calC(S).$
\item There exists $p = p(\calX)>0$ so that for any quasi-geodesic ray $q$, the following holds:
\begin{enumerate}
\item $q$ is $\kappa$-Morse if and only if $q$ is $\kappa$-contracting.
\item If $q$ is $\kappa$-Morse, then $q$ has a $\kappa^p$-persistent shadow.

\item If $q$ is a median ray with a $\kappa$-persistent shadow, then $q$ is $\kappa$-Morse. 
\end{enumerate}

\item For any proper geodesic metric space $X$, the $\kappa$-boundary $\partial_\kappa X$ is a visibility space, i.e. points in $\partial_\kappa X$ are connected by bi-infinite $\kappa$-Morse geodesics.
\end{enumerate}

\end{thmi}

In particular, having a $\kappa$-persistent shadow completely characterizes sublinear Morseness for median rays.  We also prove a related characterization in terms of bounding the growth of subsurface projections (Theorem \ref{thm:HHS subproj char intro}).  Moreover, we establish both of these characterizations for Teichm\"uller geodesic rays (see Theorem \ref{thm:Teich geo intro}), which are not median in general.  We discuss some potential implications to the vertical foliations of $\kappa$-Morse Teichm\"uller rays in Subsection \ref{subsec:UE intro} below. 

Theorem \ref{thm:main informal} is new for mapping class groups \cite{MM99,MM00, HHS_I} and Teichm\"uller spaces of finite-type surfaces \cite{Raf07, Dur16}, extra-large type Artin groups \cite{HMS21}, the genus two handlebody group \cite{miller2020stable}, surface group extensions of lattice Veech groups \cite{DDLS1,DDLS2} and multicurve stabilizers \cite{russell2021extensions}, and the fundamental groups of 3–manifolds without Nil or Sol components \cite{HHS_II}, as well as various combinations and quotients of these objects \cite{BHMScombo, BR_combo}.

We remark that Theorem \ref{thm:main informal} is a direct generalization of the Morse case (where $\kappa = 1$) \cite{Cordes, ACGH, ABD}.  However, our techniques are substantially different and it is far from clear that the original approaches in HHSes avoiding cubical techniques would work, at least not without substantial new technical complications.

Although the statements of items (1), (2b), and (2c) only involve the hierarchichal structure of an HHS $\calX,$ the proofs heavily rely on the cubical approximations of median hulls in $\calX$ (Theorem \ref{thm:cube model, informal}). Moreover, the characterization in (2a) of Theorem \ref{thm:main informal} proceeds entirely via quasicubical arguments, and thus holds in the broader class of \emph{locally quasi-cubical} coarse median spaces.

The rest of the introduction explains this context, including some of the tools we develop, and gives various refinements of the above stated results.

\subsection{Capturing hyperbolicity in HHSes via CAT(0) cube complexes}

A geodesic metric space $X$ is \emph{coarse median} if, roughly, triangles have coarsely well-defined barycenters.  The canonical examples are hyperbolic spaces, mapping class groups \cite{BM_centroid} and HHSes \cite{HHS_II}, and CAT(0) cube complexes; see Definition \ref{defn:coarse median} and Bowditch \cite{Bowditch13, Bowditch_medianbook}.

For two coarse median spaces $X,Y$, a map $f:X \rightarrow Y$ is said to be \emph{$\lambda$-median} if $f$ preserves the median up to an additive error of $\lambda,$ for some $\lambda \geq 0.$

A \emph{$\lambda$-median path} in a coarse median space $X$ is a $(\lambda, \lambda)$-quasi-isometric embedding $\gamma:[a,b] \rightarrow X$ which is $\lambda$-median. More generally, any finite set $F \subset X$ admits a \emph{median hull}, denoted $\hull_{X}(F)$, which encodes all the median paths between the points in $F$ (Definition \ref{def:median_hulls}) and is median convex (Definition \ref{def:median convex}).

 A coarse median space $X$ is said to be \textbf{locally quasi-cubical (LQC)} if for each finite set of points $F$, there exist some $\lambda=\lambda(|F|)$, a CAT(0) cube complex $Q$ of a uniform dimension, and a $(\lambda, \lambda)$-quasi-isometry $f:Q \rightarrow \hull_X(F)$ which is $\lambda$-median.
 
 Local quasi-cubicality is a higher-rank generalization of Gromov's characterization of hyperbolicity \cite{Gro87}, i.e., hyperbolic spaces are exactly the locally quasi-arboreal spaces.

The LQC property was originally established for HHSes in \cite{HHS_quasi} (as ``cubical approximations''), extended to a more general class of coarse median spaces in \cite{Bowditch2019CONVEXITY}, and stabilized in \cite{DMS20} (as ``cubical models'').  It was studied as an abstract property for hulls of pairs of points in \cite{HHP} (as ``quasi-cubical intervals''), and we are using the terminology from the forthcoming \cite{PSZ}, where those authors will systematically study the class of LQC spaces.

One of our main tools, Theorem \ref{thm:limiting model}, allows us to study the hulls of any finite number of median rays via cube complexes:

\begin{thmi}\label{thm:cube model, informal}
Let $X$ be a proper LQC space, $F \subset X$ nonempty and finite, and $Y$ a finite set of median rays.  Then $\hull_{X}(F \cup Y)$ admits a $\lambda$-median $(\lambda,\lambda)$-quasi-isometry to a uniformly finite dimensional CAT(0) cube complex $Q$, with $\lambda = \lambda(X, |F|, |Y|)>0$.
\end{thmi}

Among our contributions to this quasi-cubical machinery are the techniques we develop for comparing the approximating CAT(0) cube complexes (cubical models for short) for finite sets $Y' \subsetneq Y$ and $F' \subsetneq F$.  This allow us to, for instance, redefine certain median (and hierarchical) gate maps as cubical gate maps (Lemma \ref{lem: gates are median} and Theorem \ref{thm:gates in LQC}), which are crucial for our various characterizations of $\kappa$-Morseness in HHSes.

More importantly, we obtain a cubical interpretation for the well-established philosophy that hyperbolicity in an HHS is encoded in its top-level hyperbolic space, which we will now discuss.

In \cite{GenHyp,Genevois2020}, Genevois introduced \textbf{well-separated hyperplanes} (Definition \ref{def: well-separated hyperplanes}) and showed that, similarly to how the top-level hyperbolic space records the hyperbolicity of an HHS, these hyperplanes capture the hyperbolic aspects of a CAT(0) cube complexe. For instance, he showed that Morse geodesic rays are characterized by crossing such hyperplanes at a uniform rate, and that well-separated hyperplanes can be used to build hyperbolic models for the CAT(0) cube complex under consideration.  This philosophy was further reinforced by Murray-Qing-Zalloum \cite{Murray-Qing-Zalloum} and Incerti-Medici--Zalloum \cite{IMZ21}, in work relevant to our paper.

In this article, we show that hyperbolicity of the top-level curve graph is witnessed by the crossing of well-separated hyperplanes in the approximating CAT(0) cube complex. In particular, we establish a direct quantitative connection between the speed of the shadow of a median ray in an HHS in its top-level hyperbolic space and the speed at which its image in the corresponding cubical model of its median hull crosses uniformly well-separated hyperplanes. More precisely, let $h$ be a median ray (equivalently, hierarchy ray) in an HHS $\calX$ with unbounded products.  Let $H = \hull_{\calX}(h)$ be its median hull and $f:Q \to H$ the cubical model provided by Theorem \ref{thm:cube model, informal} above, and $g:H \to Q$ its coarse inverse.

\begin{thmi} \label{thm:curves and cubes intro}
For any proper HHS $\calX$ with unbounded products, there exist $L= L(\calX)>0$ so that for any median ray $h$, the following holds. For any $a,b \in h$, let $\calH(a,b)$ be a maximal collection of pairwise-$L$-well-separated hyperplanes in $Q$ separating $g(a), g(b)$.  Then
$$|\calH(a,b)| \succ d_S(\pi_S(a), \pi_S(b)).$$

\end{thmi}

One consequence of this theorem is that if $\pi_S(h)$ has infinite diameter in $\calC(S)$, then $g(h)$ crosses an infinite sequence of $L$-well-separated hyperplanes $\{H_i\}_i$ in $Q$.  Moreover, the rate at which $g(h)$ crosses the $H_i$ is controlled by the rate of progress of $\pi_S(h)$ in $\calC(S)$.  See Corollary \ref{cor:curves and cubes} for a precise statement and Subsection \ref{subsec:sketch} for a detailed sketch of the proof. 

In fact, by increasing the constant $L$, this statement holds for the cubical model of the median hull of any finite number of median rays and interior points via Theorem \ref{thm:cube model, informal}.  This flexibility is crucial for this paper and future applications.

We note that Petyt \cite{Petyt21} recently proved the surprising result that mapping class groups and other \emph{colorable} HHGs \cite{DMS20, hagen2021non} are median quasi-isometric to CAT(0) cube complexes.  While this very nice fact eliminates the need for Theorem \ref{thm:cube model, informal} for colorable HHGs in a few instances (e.g., item (2a) of Theorem \ref{thm:main informal}), one would still need something akin to our arguments to derive the other main results.  Moreover, the mapping class group is not \emph{equivariantly} quasi-cubical \cite{KL_Hadamard, Bridson_notCCC}, and Petyt's cube complex need not admit a \emph{factor system} \cite{HHS_I}.  As a consequence, one cannot obtain the equivariance conclusion in item (1) of Theorem \ref{thm:main informal} by avoiding Theorem \ref{thm:cube model, informal}, nor can one obtain the injection by passing through the work of \cite{IMZ21}, which requires the ambient cube complex to admit a factor system.

\subsection{Sublinear Morseness in LQC spaces}

Recall that a function $\kappa:[0,\infty) \to [1, \infty)$ is \emph{sublinear} if  $\lim_{t\to \infty} \kappa(t)/t = 0$.  The $\kappa$-\emph{neighborhood} of a geodesic ray $\alpha:[0,\infty) \to \calX$ is, roughly, a cone whose radius grows in $\kappa$ with its distance to $\alpha(0)$ (Definition \ref{Def:Neighborhood}). We say $\alpha$ is \emph{weakly} $\kappa$-\emph{Morse} if quasi-geodesics with endpoints on $\alpha$ stay in its $\kappa$-neighborhood (Definition \ref{def:kappa weakly Morse}).  Two rays $\alpha, \alpha'$ are $\kappa$-\emph{asymptotic}, $\alpha \sim_{\kappa} \alpha'$, if each is contained in the other's $\kappa$-neighborhood; see Figure \ref{fig:kappa_nbhd}.  Note that $\alpha$ is \emph{Morse} \cite[Definition 1.3]{Cordes} when $\kappa = 1$.

In fact, the above (rough) definition is not the main one used in \cite[Definition 3.2]{QRT20}, where they were seeking to topologize the set of such rays along the lines of \cite{CashenMackay}.  For that purpose, they used an a priori stronger notion from \cite{QRT19}, which they call \emph{$\kappa$-Morse}.  They also defined a useful notion of a \emph{$\kappa$-contracting} set (Definition \ref{def:kappa-contracting}), and showed that it implies $\kappa$-Morseness, though not vice versa.

While these notions are equivalent in CAT(0) spaces \cite{QRT19}, they are not known to be equivalent in general.  We prove that they are for quasi-geodesic rays in LQC spaces:

\begin{thmi}\label{thm:LQC k-Morse equivalence}
Let $X$ be a proper LQC space and $q$ a quasi-geodesic ray.  The following are equivalent:
\begin{enumerate}
    \item $q$ is  $\kappa$-Morse.
    \item $q$ is $\kappa$-contracting.
    \item $q$ is weakly $\kappa$-Morse.
\end{enumerate}
\end{thmi}

Thus, for the purposes of this introduction, the reader can take $\kappa$-Morse to mean any of these equivalent definitions.

Theorem \ref{thm:LQC k-Morse equivalence} is new for Teichm\"uller space with the Teichm\"uller metric and more generally proper HHSes which are not CAT(0).  As mentioned above, it follows for mapping class groups by combining \cite{Petyt21} and \cite{QRT19}.

We note that Qing-Rafi-Tiozzo \cite{QRT20} proved that, in any proper geodesic metric space,  $\kappa$-Morse geodesic rays are $\kappa'$-contracting for some (probably larger) sublinear function $\kappa'$.  The main utility of Theorem \ref{thm:LQC k-Morse equivalence} is then that one can convert between the three properties without changing the associated sublinear function.

Moreover, the following corollary, due to \cite{ABD} in the HHS setting, is an immediate consequence \cite{ACGH,RST18}:

\begin{cori}
Let $X$ be a proper LQC space and let $\alpha$ be a quasi-geodesic ray. The following are all equivalent:

\begin{enumerate}
    \item $\alpha$ is Morse
    \item  $\alpha$ is contracting.
    \item $\alpha$ has at least quadratic divergence.
\end{enumerate}
\end{cori}

 Our proof of the above equivalences is new for HHSes, as it avoids any hierarchical arguments via Theorem \ref{thm:cube model, informal}, and also more general, as it does not require the ambient HHS to possess the bounded domain dichotomy (Definition \ref{defn:bounded domain dichotomy}).

Aside from HHSes and CAT(0) spaces, proper LQC spaces are the first known class where Morse quasi-geodesic rays have at least quadratic divergence.

In CAT(0) spaces, Murray-Qing-Zalloum \cite{Murray-Qing-Zalloum} proved that $\kappa$-Morse geodesics are characterized by having divergence bounded below by a quadratic function.  It would be interesting to know whether some version of this works for LQC spaces:

\begin{questioni}
Are $\kappa$-Morse quasi-geodesic rays characterized by having some form of quadratic divergence in the context of locally quasi-cubical groups?
\end{questioni}

\subsection{Morseness in HHSes} \label{subsec:sublinear in HHS intro}

Associated to any hierarchically hyperbolic space $\calX$ is a robust collection of machinery which largely controls its coarse geometry.  This machinery is built out of a family of uniformly $\delta$-hyperbolic spaces $\left\{\calC(U)\right\}_{U \in \mathfrak S}$, along with a family of uniformly Lipschitz coarse projections $\pi_U:\calX \to \calC(U)$ for each $U \in \mathfrak S$.  One typically studies an object $A \subset \calX$ by projecting it to each of the $\calC(U)$, implementing hyperbolic geometric arguments therein, and then combining the results via certain consistency inequations.  HHSes are coarse median \cite{HHS_II} and LQC \cite{HHS_quasi}; in fact, they are the canonical non-cubical examples of LQC spaces.

This framework generalizes work of Masur-Minsky \cite{MM99, MM00} from the context of the mapping class group $\Mod(S)$ of a finite-type surface $S$.  In that setting, the index set $\mathfrak S$ is the collection of isotopy classes of essential subsurfaces $U \subset S$, with $\calC(U)$ the curve graph of $U$ and $\pi_U:\calX \to \calC(U)$ is the standard subsurface projection \cite{MM00}.  Curve graphs are uniformly $\delta$-hyperbolic via \cite{Aougab_hyp} and $\Mod(S)$ is coarse median via \cite{BM_centroid, Bowditch13}.

Every HHS has a unique hyperbolic space sitting atop the hierarchy, which we denote by $\calC(S)$.  As one might imagine, things are considerably easier when one can restrict to working only in $\calC(S)$.  Prime examples of this philosophy are the \emph{convex cocompact} subgroups of $\Mod(S)$ \cite{FM02}, which are characterized by having quasi-isometrically embedded orbits in $\calC(S)$ \cite{KL08, Ham05}, or equivalently having orbits with uniformly bounded projections to any $\calC(U)$ for proper subsurfaces $U \subsetneq S$, or being \emph{stable} \cite{DT15}.

One can also see this philosophy with the Morse boundary of $\Mod(S)$, i.e. $\partial_1\Mod(S)$, which admits a continuous injection into the Gromov boundary of the curve graph, $\partial \calC(S)$ \cite{Cordes}.  Many of the above results were generalized to HHSes in \cite{ABD}.

Our next theorems generalize these hierarchical characterizations to the sublinearly Morse setting.

\subsection{Encoding sublinear Morseness in the boundary of the curve graph}  Part of the HHS philosophy is that the Gromov boundary of the top level hyperbolic space $\calC(S)$ encodes all of the hyperbolic directions in the ambient space.  Our first theorem implies, for instance, that the Gromov boundary of the curve graph $\calC(S)$ sees the entirety of $\partial_{\kappa}\Mod(S)$ for any sublinear function $\kappa$:

\begin{thmi}\label{thm:HHS inject intro}
Let $\calX$ be a proper HHS with unbounded products endowed with the HHS structure from \cite{ABD}, $G$ a group of HHS automorphisms of $\calX$, and $\kappa$ a sublinear function. The projection $\pi_S:\calX \to \calC(S)$ induces a $G$-equivariant continuous injection $i_{\kappa}: \partial_{\kappa}\calX \to \partial \calC(S)$.

\end{thmi}

The assumption that $\calX$ has \emph{unbounded products} is mild, and excludes none of the main examples.  It is necessary, however, so that we can assume that $\calX$ is endowed with the HHS structure constructed in \cite{ABD}, which is the standard HHS structure for $\Mod(S)$.  See Remark \ref{rem:HHS_auto} for a discussion about HHS automorphisms.

We remark that the map itself $i_{\kappa}$ is defined via a standard Arzel\'a-Ascol\'i argument to produce a median representative, but showing that the resulting map is well-defined, injective (Proposition \ref{prop:the map}) and continuous is nontrivial and uses the cubical models (Theorem \ref{thm:cube model, informal}).  Moreover, continuity is the only place in this paper where the topology on $\partial_{\kappa} \calX$ makes an appearance, and it does so in a weak way.  In particular, we prove a weak median criterion (Proposition \ref{prop:cont criterion}) which we expect will guarantee continuity for any reasonable (e.g., second countable) topology on $\partial_{\kappa} \calX$.

In fact, He \cite{He_topologies} recently proved that the topology from \cite{QRT20} on the $1-$Morse boundary $\partial_{1} \calX$ coincides with the Cashen-Mackay \cite{CashenMackay}.  Hence, we immediately obtain the following corollary, which was previously unknown:

\begin{cori} \label{cor:CM cont}
Let $\calX$ be a proper HHS with unbounded products endowed with the HHS structure from \cite{ABD}, and $G$ a group of HHS automorphisms of $\calX$.  The projection $\pi_S:\calX \to \calC(S)$ induces a $G$-equivariant continuous injection $i:\partial_1 \calX \to \partial \calC(S)$, where $\partial_1 \calX$ is endowed with the Cashen-Mackay topology.
\end{cori}

Another application is to HHSes who top-level hyperbolic spaces are quasi-trees.  Examples include right-angled Artin groups, whose top-level spaces are the contact graphs \cite{Hagen2013} of the associated Salvetti complexes \cite{HHS_I}.  More generally, one can produce examples via various combination theorems \cite{HHS_II, BR_combo}, e.g. $\mathbb Z * (\Mod(S) \times \Mod(S))$ for some hyperbolic surface $S$. The Gromov boundaries of quasitrees are totally disconnected, and hence the following is an immediate consequence of Theorem \ref{thm:HHS inject intro}, generalizing \cite{IMZ21} in the case of RAAGs:

\begin{cori} \label{cor:raag disconnect}
For any sublinear function $\kappa$ and $G$ an HHG for which $\partial \calC(S)$ is totally disconnected, the $\kappa$-Morse boundary $\partial_{\kappa} G$ is totally disconnected.  In particular, the $\kappa$-Morse boundary of any right-angled Artin group is totally disconnected.

\end{cori}

Qing-Rafi-Tiozzo proved that $\log^p$-Morse boundary is a model for the Poisson boundary for random walks on $\Mod(S)$.  Hence, Theorem \ref{thm:HHS inject intro} connects this result to earlier work of Maher \cite{Maher_linear}, which says that $\partial \calC(S)$ plays the same role.  Moreover, Theorem \ref{thm:curves and cubes intro} provides a further connection to work of Fernós \cite{FERNS2017} and Fernós-Lécure-Mathéus \cite{Ferns2018}, who showed that a certain subspace of the Roller boundary defined via a collection of well-separated hyperplanes forms a model for the Poisson boundary of a CAT(0) cube complex.

Finally, Theorem \ref{thm:HHS inject intro} connects to the main result of \cite{IMZ21}, which gives a similar continuous injection of the $\kappa$-Morse boundary of any cubical HHS into the boundary of the hyperbolic graph built from well-separated hyperplanes, as discussed above.

We will discuss the image of the map for $\Teich(S)$ in more detail in Subsection \ref{subsec:UE intro}.

\subsection{Characterizations via projections}

Our next theorem controls the growth rate of projections of $\kappa$-Morse quasi-geodesic rays to the hyperbolic spaces in an HHS, and provides the converse for median rays.  We note that median rays are exactly the \emph{hierarchy rays} \cite{HHS_II, RST18}, which are defined by having unparametrized quasi-geodesic projections to every $\calC(U)$.

We give two characterizations.  Recall that a ray $q$ is said to have a $\kappa$-\textbf{persistent shadow} if for any $s,t \in [0,\infty)$ with $s \leq t$, we have
$$d_S(q(s),q(t)) \succ \frac{t-s}{\kappa(t)}.$$

\begin{thmi}\label{thm:HHS shadow char intro}
Let $\calX$ be a proper HHS with unbounded products with the HHS structure from \cite{ABD}, and $\kappa$ a sublinear function.  There exists $p = p(\calX)>0$ so that for any quasi-geodesic ray $q,$ the  following holds:
\begin{enumerate}
    \item If $q$ is a $\kappa$-Morse, then $q$ has a $\kappa^p$-persistent shadow. 
    
    \item If $q$ is a median ray with a $\kappa$-persistent shadow, then $q$ is $\kappa$-Morse.
    
\end{enumerate}
\end{thmi}

In the case that $\kappa = 1$, having a $\kappa$-persistent shadow is equivalent to requiring that $q$ project to a \emph{parameterized} quasi-geodesic in $\calC(S)$, which is well-known to be equivalent to being $1$-Morse \cite{Behr06, DuchinRafi}.  We recover this case because all $1$-Morse rays are uniformly close to their median representatives in any proper median space.

The proof of Theorem \ref{thm:HHS shadow char intro} uses Theorem \ref{thm:curves and cubes intro} and the cubical model to exploit work of Murray-Qing-Zalloum \cite{Murray-Qing-Zalloum}, who give a complete characterization of $\kappa$-Morseness in CAT(0) cube complexes via well-separated hyperplanes.  Roughly speaking, they showed that a geodesic ray is $\kappa$-Morse if and only if it makes $\kappa$-regular progress through an infinite sequence of $\kappa$-well-separated hyperplanes; see the related Definition \ref{def:excursion}.

Our second characterization is about the growth rate of the projection distance to the hyperbolic spaces lower in the hierarchy.  We say that a quasi-geodesic ray $q$ has $\kappa$-\textbf{bounded projections} if there exists $C>0$ so that $$d_{\calC(U)}\left(\pi_U(q(0)),\pi_U(q(t))\right) < C \cdot \kappa(t),$$ for all $t \in [0,\infty)$ and all proper $U \sqsubsetneq S$.  That is, subsurface projections grow $\kappa$-sublinearly along $q$.  When $\kappa = 1$, this is equivalent to $q$ having uniformly bounded projections. 

\begin{thmi}\label{thm:HHS subproj char intro}
Let $\calX$ be a proper HHS with unbounded products and the HHS structure from \cite{ABD}, and $\kappa$ a sublinear function.  There exists $p = p(\calX)>0$ so that for any quasi-geodesic ray $q,$ the  following holds:
\begin{enumerate}
    \item If $q$ is a $\kappa$-Morse, then $q$ has $\kappa$-bounded projections.
    \item  If $\kappa^p$ is sublinear and $q$ is a median ray with $\kappa$-bounded projections, then $q$ is $\kappa^p$-Morse.
\end{enumerate}
\end{thmi}

Setting  $\kappa = 1$ recovers the Morse case \cite{Behr06, DuchinRafi, ABD}, as every geodesic ray with uniformly bounded projections is median. In Teichm\"uller space with the Teichm\"uller metric, we obtain a characterization for Teichm\"uller geodesics, despite the fact that they are not median; see Theorem \ref{thm:Teich geo intro} below.

We note that Qing-Rafi-Tiozzo \cite{QRT20} introduced a combinatorial version of $\kappa$-bounded projections in mapping class groups, showing that it implies $\kappa^p$-Morseness for median rays.  They prove that a sample path along a random walk $\log^p$-tracks a median ray satisfying their combinatorial condition.  While characterizing $\kappa$-Morseness was evidently not their aim, it does provide an intriguing possibility.

However, this is not the case.  In Proposition \ref{prop:counter MCG}, we prove that $\kappa$-Morse rays need not satisfy this property.   In particular, we produce $\kappa$-Morse rays in the mapping class group which do not have combinatorial $\kappa'$-bounded projections for any sublinear $\kappa'$.  Thus their stronger notion does not characterize $\kappa$-Morseness, as we do in Theorems \ref{thm:HHS shadow char intro} and \ref{thm:HHS subproj char intro}.

\subsection{Characterizing sublinear Morseness for Teichm\"uller geodesics}

The last application of our techniques is to Teichm\"uller geodesics.  For this discussion, let $S$ be a finite-type surface admitting a hyperbolic metric, $\Mod(S)$ its mapping class group and $\Teich(S)$ its Teichm\"uller space with the Teichm\"uller metric.

Work of Masur-Minsky \cite{MM99} and Rafi \cite{Raf05, Raf07, Raf14} shows that Teichm\"uller geodesics project to unparametrized quasi-geodesics in every subsurface curve graph.  This is not enough to make them hierarchy paths, because the hyperbolic spaces associated to annuli in the HHS structure are horoballs over annular curve graphs \cite{Dur16}, and a simple closed curve can become short along a Teichm\"uller geodesic even when no twisting is done around this curve, causing backtracking in the corresponding horoball.

Nonetheless, work of Rafi \cite{Raf05} and Modami-Rafi \cite{Modami_short} says that the length of a simple closed curve along a Teichm\"uller geodesic is controlled by the size of the projections of the subsurfaces it bounds.  In particular, a sublinear bound on subsurface projections gives a sublinear bound on the backtracking in the horoballs over annuli.  This allows us to conclude that a Teichm\"uller ray with $\kappa$-bounded projections (including to annular horoballs) must stay $\kappa$-close to the hull of any median representative.   The techniques for Theorems \ref{thm:HHS shadow char intro} and \ref{thm:HHS subproj char intro} then allow us to deduce the following two characterizations of $\kappa$-Morseness for Teichm\"uller geodesics:

\begin{thmi}\label{thm:Teich geo intro}
There exists $p = p(S)>0$ so that for any sublinear function $\kappa$ the following hold:
\begin{enumerate}
\item If $\gamma$ is a $\kappa$-Morse Teichm\"uller geodesic, then $\gamma$ has $\kappa$-bounded projections.
\item If $\kappa^{2p}$ is sublinear, and $\gamma$ is a Teichm\"uller geodesic with $\kappa$-bounded projections, then $\gamma$ is $\kappa^{2p}$-Morse.
\item If $\kappa^p$ is sublinear and $\gamma$ has $\kappa$-persistent shadow, then $\gamma$ is $\kappa^{p+1}$-Morse.

\end{enumerate}

\end{thmi}

Once again, the powers $2p$ and $p+1$ are artifacts of our proofs, and we suspect they are not optimal.

\subsection{Unique ergodicity and sublinear Morseness} \label{subsec:UE intro}
We end with an intriguing connection to Teichm\"uller dynamics provided by the injective map $i: \partial_{\kappa} \Teich(S) \to \partial \calC(S)$ from Theorem \ref{thm:HHS inject intro} and the control of subsurface projection growth along a $\kappa$-Morse Teichm\"uller ray via Theorem \ref{thm:Teich geo intro}.

Work of Klarreich \cite{Klarreich} identifies the Gromov boundary of the curve graph $\partial C(S)$ with the space of \emph{ending laminations} on $S$.  There is a special subset of \textbf{uniquely ergodic laminations}, those filling minimal laminations which contain no closed leaves and support a unique (up to rescaling) ergodic measure.  Uniquely ergodic laminations are important in Teichm\"uller dynamics, and the existence of non-uniquely ergodic filling minimal laminations has been known for a long time (see, e.g., \cite{keane1977non}). There has been a recent flurry of interesting examples created using curve graph techniques \cite{LLR, BLMR2, BLMR}, see also \cite{CMW}.  We point the reader toward the introduction of \cite{LLR} for a thorough discussion of the history and context.

Our present interest in unique ergodicity relates to random walks.  Let $\mu$ be a probability measure on the mapping class group whose support generates a nonelementary semigroup containing two distinct pseudo-Anosov elements.  Kaimanovich-Masur \cite{KM96} proved that any random walk of the mapping class group on Teichm\"uller space with the Teichm\"uller metric with respect to $\mu$ generically converges to a point in Thurston's compactification of Teichm\"uller space, the space of projectivized measured laminations.  Moreover, they showed that the underlying lamination is (generically) uniquely ergodic.

As noted earlier, Qing-Rafi-Tiozzo \cite{QRT20} proved that almost every random walk of the mapping class group on itself $\log^p$-tracks a $\log^p$-Morse geodesic ray.  We suspect that there is an analogous statement for random walks of the mapping class group on Teichm\"uller space with the Teichm\"uller metric, though it appears to not be immediate from their results.

Translating via work of Rafi \cite{Raf05} and Modami-Rafi \cite{Modami_short}, item (1) of Theorem \ref{thm:Teich geo intro} says that the extremal lengths of simple closed curves grow slowly along a $\kappa$-Morse ray; equivalently, these rays diverge slowly in moduli space.  The fact that sufficiently slow divergence in moduli space implies unique ergodicity is a well-established theme \cite{Cheung_slow, CheungEskin, Trevino, Smith_ue} in Teichm\"uller dynamics, although unique ergodicity is better controlled by the growth rate of flat length \cite[Main Theorem 2]{CT_ue}.

Since Theorem \ref{thm:HHS inject intro} says that each $\kappa$-Morse geodesic ray in $\Teich(S)$ picks out an ending lamination (equivalently, foliation), it is reasonable to ask whether it is possible to diverge sufficiently slowly to be $\kappa$-Morse while insufficiently slowly to have a uniquely ergodic limit:

\begin{questioni} \label{q:non-ue}
Does there exist a $\kappa$-Morse Teichm\"uller geodesic ray whose vertical foliation is non-uniquely ergodic?
\end{questioni}

We expect a positive answer to this question (see, e.g., \cite[Theorem 1]{Cheung_slow}), though an example does not appear to exist in the literature.  For instance, one can use Theorem \ref{thm:Teich geo intro} to show that many of the examples of Teichm\"uller rays with nonuniquely ergodic vertical foliations constructed in \cite{LLR, BLMR2, BLMR} do not have $\kappa$-bounded projections for explicit slow-growing functions (e.g., $\kappa(t) = \sqrt{t}$).  These constructions involve a flexible set of parameters, and while we have not confirmed that any choice of parameters fails to produce a $\kappa$-Morse example, we suspect that this is the case.

We also note that the analogous question for Weil-Petersson geodesic rays is already known, since the rays constructed in \cite{BLMR2, BLMR} have non-annular bounded projections and hence are 1-Morse with respect to the Weil-Petersson metric \cite{Brock_WP}.

Presuming a positive answer to Question \ref{q:non-ue} for Teichm\"uller rays, it becomes reasonable to seek out bounds on sublinear functions $\kappa$ for which unique ergodicity can be guaranteed:

\begin{questioni}
For which sublinear functions $\kappa$ does the image of $i_{\kappa}: \partial_{\kappa} \Teich(S) \to \partial \calC(S)$ consist only of uniquely ergodic laminations?
\end{questioni}

 We suspect that our subsurface projection bounds can be used to provide lower bounds in line with the above question, but we leave that for a future investigation.

Finally, we end with the following question to which our techniques might appear to apply but we found surprisingly stubborn:

\begin{question}
Suppose that $\gamma, \gamma'$ are Teichm\"uller geodesic rays whose projections to $\calC(S)$ are infinite with $[\pi_S(\gamma)]= [\pi_S(\gamma')] \in \partial \calC(S)$.  If $\gamma$ is $\kappa$-Morse for some sublinear function $\kappa$, is $\gamma'$ also $\kappa$-Morse?
\end{question}

Note that if $\gamma$ is $\kappa$-Morse and $\gamma'$ is $\kappa'$-Morse with $\kappa'\geq\kappa$, then injectivity of $i_{\kappa'}:\partial_{\kappa'} \Teich(S) \to \partial \calC(S)$ says that $\gamma \sim_{\kappa'} \gamma'$, which then implies that $\gamma'$ is $\kappa$-Morse by Definition \ref{def:kappa-Morse} of $\kappa$-Morseness.  Hence in the above scenario either $\gamma'$ is $\kappa$-Morse or $\gamma'$ is not $\kappa'$-Morse for any sublinear function $\kappa'$.




    
    


\subsection{Outline of paper and proof sketches}\label{subsec:sketch}

Section \ref{sec:prelim} provides the background to the paper, dealing with CAT(0) spaces, (coarse) median spaces,  cube complexes, and hierarchically hyperbolic spaces.

In Section \ref{sec:CCC}, we convert results about $\kappa$-Morse geodesic rays in CAT(0) cube complexes to results about quasi-geodesics.  This is necessary because when we transport a median ray in an HHS to the cubical model of its hull, it will only be a quasi-median quasi-geodesic ray.  Of particular note in this section are (1) Theorem \ref{thm:hyperplane characterization for quasi-geodesic rays}, which characterizes $\kappa$-Morseness for quasi-geodesics in terms of excursion sequences of well-separated hyperplanes a la \cite{Murray-Qing-Zalloum}, and (2) Proposition \ref{prop:contracting hulls, CAT(0)}, which shows that the median hull of any finite number of $\kappa$-Morse rays in a CAT(0) cube complex is $\kappa$-contracting.

In Section \ref{sec:limiting models}, we prove the limiting model Theorem \ref{thm:cube model, informal} via a nonrescaled ultralimit argument, building on ideas from \cite{HHS_quasi}.  Our main contribution here is connecting the coarse median geometry of the hull of a finite set of rays and points to the median geometry of its cubical model.  In particular, Theorem \ref{thm:gates in LQC} shows that hierarchical/median gate maps to such hulls are readily interpreted as median gates in appropriate cubical models.

This result is essential for our proof that weakly $\kappa$-Morse implies $\kappa$-contracting (Theorem \ref{thm: all equivalent LQC}), which we accomplish in Section \ref{sec:LQC characterization}. The proof involves observing that any weakly $\kappa$-Morse ray $h$ in an LQC space $X$ is $\kappa$-close to a median ray $h$, whose image in the cubical model $Q$ of its hull $H=\hull_X(h)$ is weakly $\kappa$-Morse, and hence $\kappa$-contracting.  This remains true if one adds two external points $ x,y \in X - H$.  This larger hull $H' = \hull_X(h,x,y)$ is median quasi-isometric to a cube complex $Q'$ in which $h$ is weakly $\kappa$-Morse.  Moreover, we prove that the cubical hull of $h$ in $Q'$ is $\kappa$-contracting in Proposition \ref{prop:contracting hulls, CAT(0)}.  It follows that the external points $x,y$ in $Q'$ have a $\kappa$-contracting projection to the hull of $h$ in $Q'$ via the median gate.  By Theorem \ref{thm:gates in LQC}, we can push this forward to a $\kappa$-contracting projection to the hull of $h$ in the ambient HHS, showing that $h$ and $q$ are $\kappa$-contracting since $h\sim_{\kappa} q$. 

In Section \ref{sec:HHS char}, we prove our hierarchical characterizations of $\kappa$-Morseness, Theorems \ref{thm:HHS shadow char intro} and \ref{thm:HHS subproj char intro}.  The proof that $\kappa$-Morse geodesic rays have $\kappa$-bounded projections involves familiar notions of active intervals and a passing-up argument (Lemma \ref{lem:passing-up}) converts this into the $\kappa^p$-persistent shadow property.

The reverse implications for median rays require the full power of our cubical techniques.  The main supporting result here is Theorem \ref{thm:curves and cubes intro}, which relates progress of a median ray in the curve graph to the quality of hyperplane excursion in its cubical model.  With this connection in place, one proves that if a median ray $h$ has a $\kappa$-persistent shadow in $\calC(S)$, then the image of $h$ in its cubical model must cross a $\kappa$-excursion sequence of hyperplanes, making it $\kappa$-Morse by Theorem \ref{thm:hyperplane characterization for quasi-geodesic rays}.  Again, this holds even when adding external points, showing that the median hull of $h$ is $\kappa$-contracting in $\calX$ as above, and hence $h$ is $\kappa$-Morse by Theorem \ref{thm:LQC k-Morse equivalence}.  Theorem \ref{thm:Teich geo intro} is proven by a similar argument, where one now has to take extra care since the median representative of a Teichm\"uller geodesic with a $\kappa$-persistent shadow (or $\kappa$-bounded projections) need not be $\kappa$-close to the geodesic.

\begin{figure}
    \centering
    \includegraphics[width=.7\textwidth]{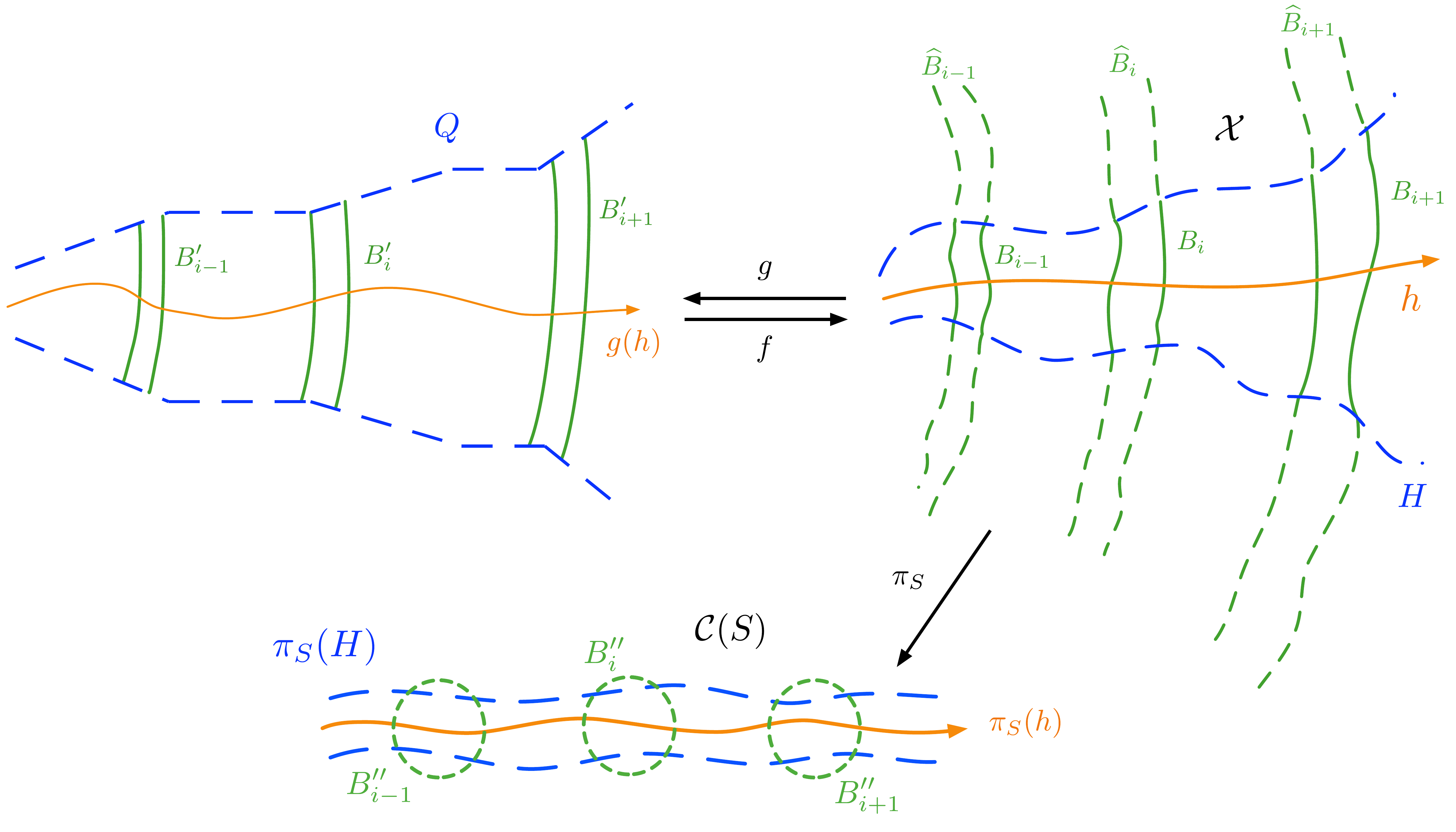}
    \caption{A median ray $h$ which makes infinite progress in the top-level hyperbolic space $\calC(S)$ must cross an infinite sequence of uniformly well-separated hyperplanes.}
    \label{fig:pseudohyperplane_philosophy}
\end{figure}

We end with a sketch of Theorem \ref{thm:curves and cubes intro} (Corollary \ref{cor:curves and cubes}).  Suppose that $h$ is a median ray whose projection $\pi_S(h)$ to the top-level hyperbolic space $\calC(S)$ is an infinite diameter quasi-geodesic ray.  We can take a sequence of balls $B''_i$ centered at points along $\pi_S(h)$, each of which separates the hyperbolic hull, $\hull_S(h)$, in $\calC(S)$ and which are pairwise as far apart as necessary.  These balls pull back to disjoint subsets $B_i$ in $\hull_{\calX}(h) = H$, whose median hulls are pairwise disjoint, separate $H$, and, importantly, so that the median gate of $B_i$ to $B_{j}$ has uniformly bounded diameter for all $i \neq j$ (this uses that the $B''_i$ are sufficiently separated in $\calC(S)$).  The $B_i$ are sent to median quasiconvex subsets $B'_i$ of the cubical model $g:H \to Q$, with pairwise bounded gate diameters via Theorem \ref{thm:gates in LQC}.  A result of Genevois \cite[Proposition 14]{Genevois16} says that these convex subsets are uniformly well-separated, from which we can produce the desired sequence of uniformly well-separated hyperplanes.  Finally, the rates at which $g(h)$ traverses these hyperplanes is comparable to the rate at which $h$ traverses the $B_i$ because $q$ is a quasi-isometry.  See Figure \ref{fig:pseudohyperplane_philosophy}.

\subsection{Acknowledgements} Durham was partially supported by NSF grant DMS-1906487.  Durham would also like to thank Cornell University for its hospitality during the 2021-22 academic year.

The authors would like to thank Shaked Bader, Benjamin Dozier, Talia Fernós, Mark Hagen, Johanna Mangahas, Babak Modami, Harry Petyt, Kasra Rafi, Jacob Russell, Jenya Sapir, Davide Spriano, and Sam Taylor for useful discussions.  In particular, we thank Babak Modami for sharing a preliminary draft of \cite{Modami_short}.  We thank Anthony Genevois for useful comments on an earlier draft of the paper and pointing out Corollary \ref{cor:raag disconnect}.  We also thank Matt Cordes and Alex Sisto for pointing out Corollary \ref{cor:CM cont}.    Durham would also like to thank Ruth Charney and R\'emi Coulon for stimulating discussions about generalizing the Morse property at MSRI during the fall of 2016.


\section{Preliminaries} \label{sec:prelim}

In this section, we will cover the background for the paper, including CAT(0) cube complexes (Subsection \ref{subsec:CCC}), coarse median spaces (Subsection \ref{subsec:coarse median}), local quasi-cubical spaces (Subsection \ref{subsec:LQC prelim}), and HHSes (Subsection \ref{subsec:HHS prelim}).

\subsection{Notations and assumptions} \label{subsec:notation}

Throughout this paper, we will be making calculations that hold up to certain errors.  This motivates the following notational scheme:

\begin{enumerate}

\item For a subset $A \subseteq X$ and a number $r \geq 0,$ we use $N(A,r)$ to denote the set of points in $X$ at distance at most $r$ from $A$.  This is called the \emph{$r$-neighborhood} of $A$. We will often also write $x \underset{r}{\in}A$ to abbreviate $x \in N(A,r).$

\item For two subsets $A,B$ of a metric space $X,$ we will use $d_{Haus}(A,B)$ to denote the Hausdorff distance between $A,B.$

    
    \item If $f:Q \rightarrow H$ is a quasi-isometry with a coarse inverse $g,$ we use the notation $$(g,f,Q \rightarrow H)$$

    
    \item If $\alpha$ is a quasi-geodesic and $p,q \in \alpha,$ we use $[p,q]_\alpha$ to denote the subsegment of $\alpha$ connecting the points $p,q.$
    
    
    
 
 
 \item  For two quantities $A,B$ and a constant $C$ we will use the following notation. 
 
 \begin{itemize}
     \item $A \underset{C}{\succ} B,$ if there exists a constant $C' \geq 1,$ depending only on $C$ and $X$ such that $$A \geq \frac{1}{C'}B-C'.$$
     \item $A\underset{C}{\asymp}B,$ if $A \underset{C}{\succ} B$ and $B \underset{C}{\succ} A.$

     \item $A \succ B$ if there exists a constant $C$, depending only on $X$, such that $A \underset{C}{\succ} B$.
     
     \item $A \asymp B$ if there exists some $C$, depending only on $X$, such that $A \underset{C}{\asymp} B$.
   \end{itemize}
\end{enumerate}

We will assume throughout the paper that all quasi-geodesics under consideration are continuous. If $\beta \from [0,\infty) \to X$ is a continuous $(q, Q)$--quasi-isometric embedding, 
and $f \from X \to Y$ is a $(K, K)$--quasi-isometry then the composition 
$f \circ \beta \from [t_{1}, t_{2}] \to Y$ is a quasi-isometric embedding, but it may 
not be continuous. However, using  Lemma III.1.11 \cite{BH1}, one can adjust the map slightly to make it continuous. Abusing notation, and in light of the coarse nature of our calculations, we denote this continuous new map again by $f \circ \beta$.

We remark that the assumption that quasi-geodesics are continuous may seem at odds with the standard definition of an HHS, which only requires a quasi-geodesic space.  But any quasi-geodesic space is quasi-isometric to a graph, and quasi-isometries preserve HHSes structures, so this assumption does not materially impact the discussion to follow.

\subsection{Sublinear Morseness} \label{subsec:sub Morse background}

For the rest of the paper, we will assume that $X$ is a proper geodesic metric space and $\go \in X$ is a fixed point.  Moreover, unless mentioned otherwise, all quasi-geodesic rays under consideration will be assumed to start at the fixed point $\go.$

Let $\kappa \from [0, \infty) \to [1, \infty)$ be a sublinear function that is monotone increasing and concave.  In particular, sublinearity means that
\[
\lim_{t \to \infty} \frac{\kappa(t)}{t} = 0. \label{subfunction}
\]

The assumption that $\kappa$ is increasing and concave makes certain arguments
cleaner, otherwise they are not needed; see \cite[Remark 2.3]{QRT20}.

For a point $x \in X,$ we define the \emph{norm} of $x$ by
$$\|x\|:=d(x,\go).$$
We will simplify our notation and denote $\kappa(\|x\|)$ by $\kappa(x)$.

%

\begin{definition}[$\kappa$--neighborhood and $\kappa$--fellow traveling]  \label{Def:Neighborhood} 
For a closed set $Z \subseteq X$ and a constant $\nn$ define the \emph{$(\kappa, \nn)$--neighbourhood }
of $Z$ to be 
\[
\calN_\kappa(Z, \nn) = \Big\{ x \in X \ST 
  d_X(x, Z) \leq  \nn \cdot \kappa(x)  \Big\}.
\]

\begin{figure}
    \centering
    \includegraphics[width=.6\textwidth]{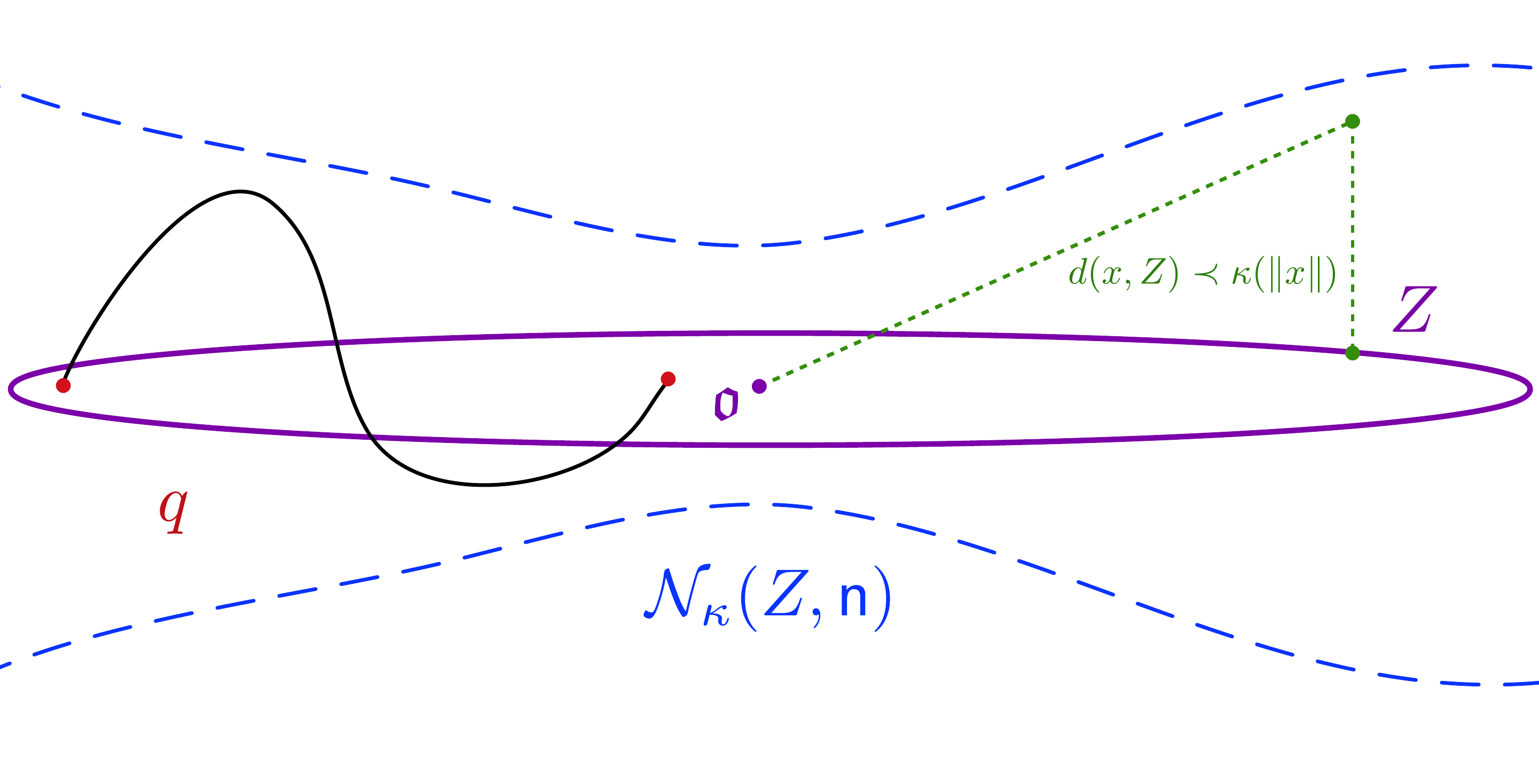}
    \caption{A $\kappa$-neighbourhood of the closed set $Z$ with multiplicative constant $\nn$.}
    \label{fig:kappa_nbhd}
\end{figure}

We will say that $\alpha$ is in a $\kappa$-neighborhood of $\beta$ if there exists $\nn$ such that $\alpha \subset \calN_\kappa(\beta, \nn).$ If $\alpha$ is in some 
$\kappa$-neighbourhood of $\beta$ and $\beta$ is in some 
$\kappa$--neighbourhood of $\alpha$, we say that $\alpha$ and $\beta$ 
\emph{$\kappa$--fellow travel} each other, written $\alpha \sim_{\kappa} \beta$. In fact, Lemma 3.1 in \cite{QRT20} shows that if $\alpha \subset \calN_\kappa(\kappa,\nn)$, then $\beta \subset \calN_\kappa(\kappa,2\nn)$. In particular, two quasi-geodesic rays $\kappa$-fellow travel if and only if one of them is in some $\kappa$-neighborhood of the other.

\end{definition}

We next give three definitions of properties that a geodesic ray or closed set might satisfy, in increasing order of strength.

The first definition is, for us, the most intuitive, since it is a direct sublinear generalization of the definition of Morse:
\begin{definition}(weakly $\kappa$-Morse, \cite[Definition A.7]{QRT20})\label{def:kappa weakly Morse} A closed set $Z$ is said to be \emph{weakly $\kappa$-Morse} if there exists a map $\mm_Z: \mathbb{R}^+ \times \mathbb{R}^+ \rightarrow \mathbb{R}^+$ such that for any $(q,Q)$-quasi-geodesic $\beta$ with end points on $Z,$ we have $ \beta \subset \calN_\kappa(Z, \mm_Z(q,Q)).$
\end{definition}

We also consider the following definition.

\begin{definition} \label{def:kappa-Morse} ($\kappa$-Morse, \cite[Definition 3.2]{QRT20}) Let $\kappa$ be a concave sublinear function and let $\go \in X$. A closed set $Z$  is said to be \emph{$\kappa$-Morse} if there exists a proper function $\mm_Z: \mathbb{R}^+ \times \mathbb{R}^+ \rightarrow \mathbb{R}^+$ such that for any sublinear function $\kappa'$ and for any $r > 0$, there exists $R \geq r$ such that for any $(q, Q)$-quasi-geodesic ray $\beta$ with $\beta(0)=\go$ and $\mm_Z(q, Q)$ small
compared to $r$, we have
$$d_X (\beta_R, Z) \leq  \kappa'(R) \implies \beta |_r \subseteq \mathcal{N}_\kappa(Z, \mm_Z(q, Q)).$$
The function $\mm_Z$ will be called a \emph{Morse gauge} of $Z.$

\end{definition}


Recall that for a closed set $Z \subset X,$ we use $\mathcal{P}(Z)$ to denote the collection of all subsets of $Z.$

\begin{definition} ($\kappa$-projection, \cite[Definition 5.1]{QRT20})\label{def:kappa-projection} Let $(X,d_X)$ be a proper geodesic metric space and $Z \subset X$ be a closed subset. Let $\kappa$ be a concave sublinear function. A map $\pi_Z: X \rightarrow \mathcal{P}(Z)$ is said to be a \emph{$\kappa$-projection} if there exists constants $D_1,D_2$ depending only on $Z$ and $\kappa$ such that for all $x \in X$ and $z \in Z$ we have 

$$\text{diam}(\{x\} \cup \pi_Z(x)) \leq (D_1+1)d_X(x,Z)+D_2\kappa(x).$$

\end{definition}

We can now state the sublinear contraction property:

\begin{definition} ($\kappa$-contracting \cite[Definition 5.3]{QRT20})\label{def:kappa-contracting}
Let $Z$ be a closed set in $X$ and let $\pi_Z:X \rightarrow \mathcal{P}(Z)$ be a $\kappa$-projection. The set $Z$ is said to be \emph{$\kappa$-contracting with respect to $\pi_Z$} if there are constants $C_1,C_2$, depending only on $Z$ and $\kappa,$ such that for any $x,y \in X$ we have 

$$d(x,y) < C_1 d(x,Z) \implies \text{diam}(\pi_Z(x) \cup \pi_Z(y)) \leq C_2 \kappa(x).$$ A set is said to be \emph{$\kappa$-contracting} if there exists a $\kappa$-projection $\pi_Z$ such that $Z$ is $\kappa$-contracting with respect to $\pi_Z.$ The constant $C_2$ in the definition above will be referred to as the \emph{$\kappa$-contraction constant} or simply the \emph{contraction constant}. In the special case where $C_1=1,$ we say that $Z$ is \emph{$\kappa$-strongly contracting}.
\end{definition}

\begin{lemma}\label{lem:invariance_of_neighborhood}
Let $Z, Z'\subset X$ be two closed sets such that $Z'$ is in some $\kappa$-neighborhood of $Z$. We have the following:

\begin{enumerate}
\item $Z$ is weakly $\kappa$-Morse $\iff$ $Z'$ is weakly $\kappa$-Morse.
\item $Z$ is $\kappa$-Morse $\iff$ $Z'$ is $\kappa$-Morse.
\item $Z \text{ is }\kappa\text{-contracting} \iff Z' 
\text{ is }\kappa\text{-contracting.}$

\end{enumerate}

\end{lemma}

\begin{proof} The proof of item (2) is in \cite{QRT20}. We provide an outline for the proof of item (3) and leave item (1) as an exercise for the reader. If $Z$ is $\kappa$-contracting, then there exists a $\kappa$-projection $\pi_Z:X \rightarrow \mathcal{P}(Z)$ as in Definition \ref{def:kappa-projection}. Composing $\pi_Z$ with the nearest point projection map to $Z'$ yields a map $\pi_{Z'}:X \rightarrow \mathcal{P}(Z')$. Since $Z'$ is in a $\kappa$-neighborhood of $Z,$ it is easy to check that the new map $\pi_{Z'}$ is a $\kappa$-projection. Using the definition of $\pi_{Z'}$, if two points $x,y \in X$ have a large projection to $Z'$ under $\pi_{Z'}$, then they must have a large projection to $Z$ under $\pi_{Z}.$ One can combine such observations to show that elements of Definition \ref{def:kappa-contracting} are all met. 
\end{proof}



    


The following theorem explains how these properties are related in the general setting:

\begin{theorem}[{\cite[Theorem 5.5]{QRT20}}, {\cite[Lemma 3.10]{QRT20}}] \label{thm:contracting implies Morse}
Let $X$ b a proper geodesic metric space and let $\kappa$ be a sublinear function. We have the following.

\begin{enumerate}
    \item If $Z$ is a closed $\kappa$-contracting set, then it is $\kappa$-Morse.
    
    \item If $\alpha$ is a $\kappa$-Morse quasi-geodesic ray, then it is weakly $\kappa$-Morse.
\end{enumerate}

\end{theorem}





In Theorem \ref{thm: all equivalent LQC}, we prove that Definitions \ref{def:kappa-contracting}, \ref{def:kappa-Morse}, and \ref{def:kappa weakly Morse} are all equivalent in LQC spaces for quasi-geodesic rays.

In \cite{QRT20}, the authors define the $\kappa$-\emph{boundary} $\partial_{\kappa} X$ of a geodesic metric space $X$ to be the set of all $\kappa$-Morse rays up to $\kappa$-fellow-traveling.  They use the stronger $\kappa$-Morse property to allow them to use ideas from \cite{CashenMackay} for defining the topology on $\partial_{\kappa} X$.  This topology only arises in our work in one place, Subsection \ref{subsec:continuity}, so we will delay discussing it until then.

We end this subsection with some useful observations:

\begin{lemma}[{\cite[Lemma 3.2]{QRT19}}]\label{lem: relating sublinearness} Let $X$ be a geodesic metric space and let $\kappa$ be a sublinear function. For any $D_0,$ there exists $D_1,D_2$, depending only on $\kappa$ and $D_0$ such that for any $x,y\in X$

$$d(x,y) \leq D_0 \kappa (x) \implies D_1\kappa(x) \leq  \kappa(y) \leq D_2 \kappa(x).$$

\end{lemma}

\begin{remark} \label{rmk:relating sublinearness}
The above Lemma \ref{lem: relating sublinearness} will be used frequently in the paper. An immediate consequence of the statement is that if $x,y$ are within distance $C\kappa(x)$, then they must also be within distance $C'\kappa(y)$, where $C'$ depends only on $\kappa$ and $C$. For instance, any set $Z$ which satisfies Definition \ref{def:kappa weakly Morse}, must also satisfy $d(p, Z) \leq C \kappa(p'),$ where $p' \in \pi_Z(p),$ and $C$ is a constant depending only on $\kappa$ and $\mm_Z(q,Q).$
\end{remark}

The following proposition is a useful consequence of Lemma \ref{lem: relating sublinearness}.

\begin{proposition} \label{prop: Refining Morse}
Let $X$ be a proper geodesic metric space and $h \subset X$ a $(q',Q')$-quasi-geodesic ray which is  weakly $\kappa$-Morse.  For every constants $q, Q$ there exists a constant $K$ so that if $\alpha:[0,A] \rightarrow X$ is a $(q,Q)$ quasi-geodesic with end points $h(t_1), h(t_2) \in h$ then the following hold for all $s \in [0,A]$:
\begin{enumerate}
    \item $d(\alpha(s), h) \leq K \kappa(p_s),$ where $p_s$ is a closest point projection of $\alpha(s)$ to $h.$
    \item $d(\alpha(s), h) \leq K \kappa(t_2).$
    \item $d(\alpha(s), h([t_1,t_2]) \leq K \kappa(t_2).$
\end{enumerate}

\end{proposition}

\begin{proof}

Item (1) follows immediately from Remark \ref{rmk:relating sublinearness}.

For item (2), observe that since $h$ and $\alpha$ are quasi-geodesics, for any $s \in [0,A]$, we have $d(\alpha(s),h(t_2)) \leq B|t_2-t_1|$ for a constant $B$ that depends only on $q,Q, q', Q'.$ Since $p_s$ is a nearest point projection of $p$ to $h$, we have $$d(\alpha(s), p_s) \leq d(\alpha(s),h(t_2)) \leq B|t_2-t_1| \leq Bt_2. $$ Hence, by the triangle inequality, we have that $d(h(t_2), p_s) \leq 2Bt_2.$ It follows that

\begin{align*}
    \|p_s\|&=d(h(0), p_s)\\
    &\leq d(h(0), h(t_2))+d(h(t_2),p_s)\\
    &\leq Ct_2+2Bt_2,\\
\end{align*}
where $C$ is a constant depending only on the quasi-geodesic constants of $h.$ Thus, we have $\|p_s\| \leq (C+2B)t_2 \leq Dt_2$, where $D=\text{max}\{1,C+2B\}$, and so

$$d(\alpha(s), h)= d(\alpha(s), p_s) \leq \dd \kappa(p_s) \leq D\dd \kappa (t_2),$$
proving item (2).

For item (3), let $\beta = h|_{[t_1,t_2]}.$  By part (2), we have $d(\alpha(s), h) \leq E \kappa(t_2),$ for any point $s \in [0,A].$ Let $h_1:=h|_{[0,t_1]}$, and $h_2=h|_{[t_2, \infty)}$. Notice that $h=h_1 \cup \beta \cup h_2$. We know that that $\alpha(0)=h(t_1)$ is in the $E \kappa(t_2)$-neighborhood of $h_1$. Define 

$$r=\text{sup}\{s \in [0,A] \,\,|\,\, \alpha(s) \in N(h_1, E\kappa(t_2))\}.$$

By definition of $r$, for any $0<\epsilon<1$, there exist $0\leq t' \leq t_1 \leq t''\leq A$ so that $d(h(t'), \alpha(r-\epsilon)) \leq E\kappa(t_2)$ and $d(h(t''), \alpha(r+\epsilon)) \leq E\kappa(t_2)$, where $h(t'') \in \beta \cup h_2.$ By the triangle inequality, we have


\begin{align*}
d(h(t'), h(t''))&\leq d(h(t'), \alpha (r-\epsilon))+d(\alpha(r-\epsilon), \alpha (r+\epsilon))+d(\alpha(r+\epsilon), h(t''))\\
&\leq E \kappa(t_2)+(2\epsilon q+ Q)+E \kappa(t_2)\\
& \leq 2E\kappa(t_2)+2q+Q.
\end{align*}

For $D'=2E\kappa(t_2)+2q+Q$, the above shows that $d(h(t'), h(t'')) \leq D'$. Since $t' \leq t_1 \leq t''$, we have $d(h(t'),h(t_1))\leq P D'$ for some constant $P$ depending only on $q',Q'$. Hence, we have 

\begin{align*}
 d(\alpha(0), \alpha(r)) &=d(h(t_1), \alpha(r))\\
 &\leq d(h(t_1), h(t'))+d(h(t'), \alpha(r-\epsilon))+d(\alpha(r-\epsilon), \alpha(r))\\ &\leq PD'+ E \kappa(t_2) + \epsilon q+Q\\
 &\leq P D'+ E \kappa(t_2) + q+Q \\
 &\leq  P (2E\kappa(t_2)+2q+Q)+ E \kappa(t_2) + q+Q\\
 &\leq (P (2E+2q +Q)+E+q+Q)\kappa(t_2),
\end{align*} where the last inequality holds as $\kappa \geq 1$ by definition. Let $E'=(P(2E+2q +Q)+E+q+Q)$. Notice that $E'$ depends only on $q,Q,q',Q'$ and $\kappa.$ We have just shown that $d(\alpha(0), \alpha(r)) \leq E' \kappa(t_2)$. Further, since $\alpha$ is a $(q, Q)$-quasi-geodesic, we get $r \leq q(E' \kappa(t_2)+Q).$ Hence, 
for any $l \leq r$, we have 

\begin{align*}
d(\alpha(0), \alpha(l)) &\leq q l+Q\\ 
& \leq q r+Q\\ 
&\leq q^2(E' \kappa(t_2)+Q)+Q\\
&\leq (q^2(E'+Q)+Q)\kappa(t_2),
\end{align*}
where the last inequality holds as $\kappa \geq 1.$ As $\alpha(0) \in \beta,$ we have $d(\alpha(l), \beta) \leq (q^2(E'+Q)+Q) \kappa(t_2)$ for all $l \in [0,r]$. An identical argument shows that if $r'=\text{inf}\{s \in [0,A] \,\,|\,\, \alpha(s) \in N(h_2, E\kappa(t_2))\},$ then, for all $l' \geq r'$, we have $d(\alpha(l'), \beta)) \leq M \kappa(t_2)$ for a constant $M$ that depends only on $q, Q, \kappa,q'$ and $Q'$. This finishes the proof.

\end{proof}

\subsection{Visibility} In this subsection, we prove that for any proper geodesic metric space, the sublinearly Morse boundary is a visibility space. In order to do so, we will need the following two statements.

\begin{lemma}\label{lem:union is Morse}
Let $\alpha_1, \alpha_2$ be two quasi-geodesic rays with $\alpha_1(0)=\alpha_2(0)=\go$ and let $Z= \alpha_1 \cup \alpha_2$. We have the following:

\begin{enumerate}
    \item If $\alpha_1, \alpha_2$ are $\kappa$-Morse, then $Z$ is $\kappa$-Morse.
    \item If $\alpha_1, \alpha_2$ are weakly $\kappa$-Morse, then $Z$ is weakly $\kappa$-Morse.
\end{enumerate}

\end{lemma}

\begin{proof}
 To see part (1), we define $\mm_Z:= \text{max}\{\mm_{\alpha_1}, \mm_{\alpha_2}\}$ and we let $\kappa'$ be any sublinear function and $r>0.$ Since both $\alpha_1, \alpha_2$ are $\kappa$-Morse, there exists $R_{\alpha_1},R_{\alpha_2} \geq r$ such that the conclusion of Definition \ref{def:kappa-Morse} holds. Define $R:=\text{min}\{R_{\alpha_1},R_{\alpha_2}\}$. If $\beta'$ is a $(q,Q)$-quasi-geodesic ray starting at $\go$ with $\mm_Z(q,Q)$ small compared to $r$ and $d_X(\beta'_R, Z) \leq \kappa'(R),$ then $d_X(\beta'_R,\alpha_1) \leq \kappa'(R)$ or $d_X(\beta'_R,\alpha_2) \leq \kappa'(R)$ as $Z= \alpha_1 \cup \alpha_2.$ Since $\alpha_1,\alpha_2$ are both $\kappa$-Morse, we get that $\beta'|_r \subseteq \calN_\kappa(\alpha_1, \mm_\alpha(q,Q)) \subset \calN_\kappa(Z, \mm_\alpha(q,Q))$ or $\beta'|_r \subseteq \calN_\kappa(\alpha_2, \mm_\alpha(q,Q)) \subset  \calN_\kappa(Z, \mm_\alpha(q,Q))$ which proves that $Z$ is $\kappa$-Morse. To see part (2), we let $\beta$ be a finite $(q,Q)$-quasi-geodesics with end points $p_1,p_2$ on $Z.$ If $p_1,p_2$ are both  on $\alpha_1$ or $\alpha_2$ then the conclusion follows as each $\alpha_i$ is weakly $\kappa$-Morse. It remains to consider the case where $p_1 \in \alpha_1$ and $p_2 \in \alpha_2.$ In this case, let $p$ be a nearest point projection of $\go$ to $\beta$ and let $[\go,p]$ be a geodesic connecting $\go,p,$ then, using Lemma 2.5 of \cite{QRT19}, $\beta_1=[\go,p] \cup [p,p_1]_\beta$ is a $(3q,Q)$-quasi-geodesic with end points on $\alpha_1$. Similarly, $\beta_2=[\go,p] \cup [p,p_2]_\beta$ is a $(3q,Q)$-quasi-geodesic with end points on $\alpha_2.$ The conclusion follows as $\alpha_1,\alpha_2$ are each weakly $\kappa$-Morse and $\beta \subset \beta_1 \cup \beta_2$.

\end{proof}

\begin{remark} \label{rmk:bounded kappa intersection} Let $\alpha_1,\alpha_2$ be two quasi-geodesic rays starting at $\go$ such that $\alpha_2$ is $\kappa$-Morse, and $\alpha_1,\alpha_2$ don't $\kappa$-fellow travel each other. We remark that for any constant $D$ and any sublinear function $\kappa'$ the set $\alpha \cap \calN_{\kappa'}(\alpha_2, D)\}$ is bounded, as otherwise, Definition \ref{def:kappa-Morse} would imply that $\alpha_1$ and $\alpha_2$ do $\kappa$-fellow travel violating the assumption that they don't. To summarize, given $\alpha_1,\alpha_2$ as above, for any constant $D$ and any sublinar function $\kappa',$ there exists a constant $D',$ depending only on $D$, $\kappa',$ $\alpha_1$ and $\alpha_2$ such that $d_X(\go, \alpha_1(t)) \leq D'$ for all $\alpha_1(t) \in \calN_{\kappa'}(\alpha_2, D)).$  This remark will be used in the proof of Theorem \ref{thm:general_visibility} below.
\end{remark}

We show that the $\kappa$-boundary of any proper geodesic metric space is a visibility space.

\begin{theorem}\label{thm:general_visibility}(Visibility of sublinear boundaries)
Let $\alpha_1, \alpha_2$ be two $\kappa$-Morse quasi-geodesic rays starting at $\go$ in $X$ which do not $\kappa$-fellow travel each other. There exists a geodesic line $\beta:(-\infty, \infty) \rightarrow X$ such that for any $p \in \beta,$ if $\beta_1:=\beta|_{[p, \infty)}$ and $\beta_2:=\beta|_{(-\infty,p]}$, then 
$[\beta_1]=[\alpha_1]$ and $[\beta_2]=[\alpha_2].$
\end{theorem}

\begin{proof} Since $\alpha_i$ is $\kappa$-Morse for each $i=1,2$, Theorem \ref{thm:contracting implies Morse} gives that $\alpha_i$ is weakly $\kappa$-Morse for $i=1,2.$ Define $Z= \alpha_1 \cup \alpha_2,$ using Lemma \ref{lem:union is Morse}, the set $Z$ is weakly $\kappa$-Morse. We let $\mm_Z$ denote the function as in Definition \ref{def:kappa weakly Morse} and we fix $m:=\mm_Z(1,0).$ We let $\beta_n:[a_n,b_n] \rightarrow X$ be a sequence of geodesic segments starting and ending on $\alpha_1(n), \alpha_2(n)$ respectively. Since $Z=\alpha_1 \cup \alpha_2$ is weakly $\kappa$-Morse, we get that $d_X(\beta_n(t_n),\alpha_1 \cup \alpha_2) \leq m \kappa (\beta_n(t_n))$ for all $n$ and all $t_n \in [a_n,b_n].$ If we define $c_n:=\text{inf}\{t_n|t_n \in [a_n,b_n] \text{ and } \beta_n(t_n) \in \calN_{\kappa}(\alpha_2,m)\}$, then we have 
    
    $$d_X(\beta_n(c_n), \alpha_1)) \leq m\kappa(\beta_n(c_n))+1$$ and $$d_X(\beta_n(c_n), \alpha_2))\leq m \kappa(\beta(c_n))+1.$$ This yields two points $p_n,p_n'$ in $\alpha_1,\alpha_2$ respectively such that 
    \begin{align*}
            d_X(\beta_n(c_n), p_n)) &\leq m\kappa(\beta_n(c_n))+1\\
            &\leq (m+1) \kappa(\beta_n(c_n)),\text{ and}\\
d_X(\beta_n(c_n), p_n')) &\leq m\kappa(\beta_n(c_n))+1\\
            &\leq (m+1) \kappa(\beta_n(c_n)).
    \end{align*}

    Applying Lemma \ref{lem: relating sublinearness} to the second equation above gives us that $\kappa(\beta_n(c_n)) \leq C \kappa(p_n')$ for a constant $C$ depending only on $\kappa$ and $m$. The triangle inequality gives
    
    \begin{align*}
    d(p_n, p_n') &\leq  d_X(p_n, \beta_n(c_n))+d(\beta_n(c_n),p_n')\\
    &\leq 2m \kappa(\beta_n(c_n))+2\\ 
    &\leq (2m+2) \kappa(\beta_n(c_n))\\
    &\leq (2m+2)C \kappa(p_n').
            \end{align*}
            
 Now, if we let $D:=(2m+2)C$, we get that $p_n \in \calN_{\kappa}(\alpha_2, D)$. Using Remark \ref{rmk:bounded kappa intersection} above, there exists $D'$, depending only on $\kappa, D, \alpha_1$ and $\alpha_2$ such that $d(p_n, \go) \leq D'.$ On the other hand, since  $d_X(\beta_n(c_n), p_n)) \leq (m+1)\kappa(\beta_n(c_n))$, using Lemma \ref{lem: relating sublinearness}, we get a constant $C'$,  depending only on $\kappa$ and $m$ such that $\kappa(\beta_n(c_n))\leq C' \kappa(p_n).$ Hence, we have 
 
 \begin{align*}
 d_X(\beta_n(c_n), p_n)) &\leq (m+1)\kappa(\beta_n(c_n))\\
 &\leq (m+1)C' \kappa(p_n)\\
 &= (m+1)C' \kappa(d_X(p_n, \go))\\
 &\leq (m+1)C' \kappa(D').
 \end{align*}
 Now, the triangle inequality gives us that
 
 \begin{align*}
     d_X(\beta_n(c_n), \go) &\leq d_X(\beta_n(c_n), p_n))+d_X(p_n, \go)) \\
     &\leq (m+1)C' \kappa(D')+D'.
  \end{align*}
  
  This shows that for each $n$, the point $\beta_n(c_n) \in \beta_n$ is at a bounded distance from $\go$ where the bound is independent of $n.$ Therefore, applying Arzelà–Ascoli to $\{\beta_n\}$ gives a subsequence $\{\beta_{n_k}\}$ and a geodesic line $\beta$ with $\beta_{n_k} \rightarrow \beta$ uniformly on compact sets. The line $\beta$ satisfies the conclusion of the theorem.

\end{proof}

\subsection{CAT(0) cube complexes} \label{subsec:CCC} Our goal for this section is to recall some definitions and facts regarding CAT(0) spaces and cube complexes. For a more detailed introduction to $\CAT$ cube complexes, see \cite{Sageev}. 

Let $(X,d_X)$ be a metric space. A metric space is called \emph{proper} if closed balls are compact. It is called \emph{geodesic} if any two points $x, y \in X$ can be connected by a geodesic segment. A proper, geodesic metric space $(X, d_{X})$ is $\CAT$ if geodesic triangles in $X$ are at 
least as thin as triangles in Euclidean space with the same side lengths. To be precise, for any 
given geodesic triangle $\triangle pqr$, consider (up to isometry) the unique triangle 
$\triangle \overline p \overline q \overline r$ in the Euclidean plane with the same side 
lengths. For any pair of points $x, y$ on the triangle, for instance on edges $[p,q]$ and $[p, r]$ of the 
triangle $\triangle pqr$, if we choose points $\overline x$ and $\overline y$  on 
edges $[\overline p, \overline q]$ and $[\overline p, \overline r]$ of 
the triangle $\triangle \overline p \overline q \overline r$ so that 
$d_X(p,x) = d_{\mathbb{E}^2}(\overline p, \overline x)$ and 
$d_X(p,y) = d_{\mathbb{E}^2}(\overline p, \overline y)$, then
\[ 
d_{X} (x, y) \leq d_{\mathbb{E}^{2}}(\overline x, \overline y).
\] 

A \emph{cube} is a Euclidean unit cube $[0,1]^n$ for some $n \geq  0$. A \emph{midcube} of a cube $c$ is a subspace obtained by restricting exactly one coordinate in $[0,1]^n$ to $\frac{1}{2}$. For $n \geq 0$, let $[0,1]^n$ be an $n$-cube equipped with the Euclidean metric. We obtain a \emph{face} of a $n$--cube by choosing some indices in $\{ 1, \dots, n \}$ and considering the subset of all points where we for each chosen index $i$, we fix the $i$-th coordinate either to be zero or to be one. A \emph{cube complex} is a topological space obtained by gluing cubes together along faces, i.e. every gluing map is an isometry between faces. A CAT(0) cube complex $X$ is said to be \emph{finite dimensional} if there is an integer $v$ such that every cube in $X$ is of dimension at most $v.$

Any cube complex can be equipped with a metric as follows: the $n$--cubes are equipped with the Euclidean metric, which allows us to define the length of continuous paths inside the cube complex by partitioning every path into (finitely-many) segments which lie entirely within one cube and add the lengths of those segments using the Euclidean metric on each cube.  We define
\[ d^{(2)}(x,y) := \inf \{ length(\gamma) \vert \gamma \text{ a continous path from $x$ to $y$} \}. \]
The map $d^{(2)}$ defines a metric on $X$. We sometimes call $d^{(2)}$ the metric induced by the Euclidean metric on each cube.

\begin{definition}[$\CAT$ cube complexes]
Let $X$ be a cube complex and $d^{(2)}$ the metric induced by the Euclidean metric on each cube. We say that $X$ is a $\CAT$ cube complex if $(X,d^{(2)})$ is a $\CAT$ space.

\end{definition}

In what follows, we will be interested in the \emph{combinatorial metric} $d$ on $X$, which is the metric on the $1$-skeleton of $X$.  The combinatorial metric has an alternative description in terms of \emph{hyperplanes}, which we now describe.

\begin{definition}(Hyperplanes, half spaces and separation) Let $X$ be a $\CAT$ cube complex. A \emph{hyperplane} is a connected
subspace $h \subset X$ such that for each cube $c$ of $X$, the intersection $h \cap c$ is either empty or a
midcube of $c$.  For each hyperplane $h$, the complement $X \setminus h$ has exactly two components $h^+, h^-$ called \emph{half-spaces}
associated to $h$ . A hyperplane $h$ is said to \emph{separate} the sets $U,V \subseteq X$ if $U \subseteq h^+$ and $V \subseteq h^-.$
\end{definition}

The following lemma is a standard lemma about how the $l^1$-metric relates to the CAT(0) metric of a CAT(0) cube complex, for example, see \cite{CapriceSageev}.:

\begin{lemma} \label{lem: quasi-isometry between combinatorial and CAT(0) metrics}
If $X$ is a finite-dimensional $\CAT$ cube complex, then the $\CAT$ metric $d^{(2)}$ and the combinatorial metric $d$ are bi-Lipschitz equivalent and complete. In particular, if all cubes in $X$ have dimension $\leq m$, then $d^{(2)} \leq d \leq \sqrt{m} d^{(2)}$.
 Furthermore, for two vertices $x,y \in X^{(0)}$, we have
 
 $$d(x,y)=|\{\text{\emph{hyperplanes} $h \subseteq X$ \emph{which separate the vertices} }x,y\}|.$$
\end{lemma}


In light of the above lemma and the coarseness of our calculations, we will often not distinguish between the two metrics.

\begin{definition}(Combinatorial geodesics/CAT(0) geodesics)
A path in the $1$--skeleton of $X$ is called a \emph{combinatorial geodesic} if it is a geodesic between vertices of $X$ with respect to the combinatorial metric. A \emph{CAT(0) geodesic} is a geodesic with respect to the CAT(0) metric.
\end{definition}

There is an alternative description of a combinatorial geodesic coming from the median structure on a cube complex, which we now describe.

\begin{definition}(Medians)\label{def:CCC median} We define the \emph{median} of the vertices $x, y, z \in X$ to be the unique vertex $m(x,y,z) \in X$ obtained by associating to every hyperplane of $X$ its half-space that contains the majority of the points $x, y, z$. Equivalently, $m(x,y,z)$ is the unique point that lives in the intersection of the sets of all combinatorial geodesics connecting $\{x,y\},\{x,z\}$ and $\{y,z\}.$ For two vertices $x,y \in X,$ we define the \emph{median interval} $[x,y]$ by 

$$[x,y]=\{m(x,y,z)| z \in X^{(0)}\}.$$ Equivalently, $[x,y]$ is the union of all combinatorial geodesics connecting the vertices $x,y$.

\end{definition}

The notion of a median above gives rise to the following.

\begin{definition}(Convexity) A subset $Y$ of a CAT(0) cube complex $X$ is said to be \emph{combinatorially convex} if every combinatorial geodesic connecting vertices $x,y \in Y$ remains inside $Y$. Equivalently, $Y$ is combinatorially convex if $m(x,y,z) \in Y$ for any vertices $x,y,z \in X$ with $x,y \in Y.$
\end{definition}

In order to define the notion of well-separated hyperplanes, we need to introduce the following definition.

\begin{definition}(Facing triples)  A \emph{facing triple} is a collection of three disjoint hyperplanes such that none of them separates the other two.

\end{definition}
\begin{remark}
It is immediate by the definition of a facing triple that a geodesic cannot cross a facing triple.
\end{remark}

In contrast to the notion of a facing triple is that of a chain.

\begin{definition}[Chain] \label{defn:chain} A sequence of hyperplanes $\{h_i\}$ is said to form a \emph{chain} if each $h_i$ separates $h_{i-1}$ from $h_{i+1}.$

\end{definition}

The following notion was introduced by Genevois \cite{Genevois16} to capture the hyperbolic-like aspects of a CAT(0) cube complex.

\begin{definition}(Well-separated sets/hyperplanes) \label{def: well-separated hyperplanes}
Two disjoint combinatorialy convex sets $Y_1,Y_2$ are said to be \emph{$L$-well-separated} if the number of hyperplanes meeting them both and containing no facing triple has cardinality at most $L$. 
\end{definition}

We will mostly be interested in the special case where $Y_1,Y_2$ in Definition \ref{def: well-separated hyperplanes} above are both hyperplanes. Given a convex set $Y$ in a CAT(0) cube complex, one can define a notion of a projection from $X$ to $Y$.

\begin{definition}[Combinatorial gate map]\label{def:comb_proj}\label{def:combinatorial_projection} Let $Y$ be a combinatorially convex set in a CAT(0) cube complex $X,$ and let $x$ be a vertex in $X$. The \emph{combinatorial projection} of $x$ to $Y$, denoted $P_Y(x),$ is the vertex minimizing the distance $d(x,Y).$ Such a vertex is unique (for instance, by Lemma 1.2.3 \cite{Genevois2015HyperbolicDG}) and it is characterized by the property that a hyperplane $h$ separates $x,Y$ if and only if it separates $x,P_Y(x).$ For such a characterization, see Lemma 13.8 in \cite{Haglund2007}.

\end{definition}

For a CAT(0) cube complex $X$ and for a combinatorially convex set $Y$ in $X$, the combinatorial nearest point projection can be described purely in terms of medians. Namely, we have the following, which is immediate from the description of the combinatorial projection above.

\begin{lemma}\label{lem:gate median CCC}

Let $X$ be a CAT(0) cube complex and let $Y$ be a set which is combinatorially convex. If $P_Y:X \rightarrow Y$ is the combinatorial nearest point projection, then

$$P_Y(x) \in \underset{y \in Y}{\bigcap}[x,y],$$ where $[x,y]$ is the median interval between $x,y.$ Furthermore, $P_Y(x)$ is the only point in $Y$ with such a property.

\end{lemma}

The following is \cite[Proposition 14]{Genevois16}. Although it's stated in the special case where $Y_1,Y_2$ are hyperplanes, the proof only uses the fact that hyperplanes are combinatorially convex.

\begin{proposition}\label{prop:Genevois} Two combinatorially convex sets $Y_1, Y_2$ in a finite dimensional CAT(0) cube complex $X$ are $L$-well-separated if and only if $P_{Y_1}(Y_2)$ and $P_{Y_2}(Y_1)$ have diameters at most $cL$, where $c$ is a constant depending only on the dimension of $X.$

\end{proposition}

We recall the following theorem from \cite{Murray-Qing-Zalloum}, which characterizes $\kappa$-Morseness in terms of well-separated hyperplanes:

\begin{theorem}[{\cite[Theorem B]{Murray-Qing-Zalloum}}] \label{thm: characterize_geodesic_rays_in cube_complexes}
Let $X$ be a finite dimensional CAT(0) cube complex and let $\kappa$ be a sublinear function. A geodesic ray $b$ is $\kappa$-Morse if and only if there exist a constant $c$, a chain of hyperplanes $\{h_i\}$ and points $x_i \in b \cap h_i$ with:
\begin{enumerate}
    \item $d(x_i,x_{i+1}) \leq c\kappa(x_i),$ and 
    \item $h_i,h_{i+1}$ are $c\kappa(x_i)$-well-separated.
\end{enumerate}

\end{theorem}

\subsection{Coarse median spaces} \label{subsec:coarse median}

Coarse median spaces were introduced by Bowditch \cite{Bowditch13} to simultaneously generalize properties of hyperbolic spaces and cube complexes.  See also \cite{Bowditch_medianbook} for background on median spaces.

\begin{definition}[Coarse median space]\label{defn:coarse median}
A metric space $(X,d)$ is said to be a \emph{coarse median space} if there exist a map $m_X:X^3 \rightarrow X$ and a function $h:\mathbb{N} \rightarrow \mathbb{N}$ such that the following holds:

\begin{enumerate}
    \item For any $x,y,z,x',y',z' \in X,$ we have 
    
    $$d(m_X(x,y,z),m(x',y',z')) \leq h(1)(d(x,x')+d(y,y')+d(z,z')).$$
    \item For any $n \in \mathbb{N},$ if $A \subseteq X$ has cardinality at most $n,$ then there exist a finite cube complex $Q$ and maps $f:Q \rightarrow A,$ $g:A \rightarrow Q$ such that 
    
    \begin{enumerate}
        \item $d(f(m_Q(a,b,c),m_X(f(a), f(b), f(c)))\leq h(n)$ for all $a,b,c \in Q.$
        \item $d(a, f\circ g(a)) \leq h(n)$ for all $a \in A. $
    \end{enumerate}
\end{enumerate}

\end{definition}
For two points $x,y$ in a coarse median space, the following definition describes the collection of median points that lives between $x,y$.

\begin{definition}[Intervals] Given $x,y$ in a coarse median space $X$, we define the \emph{median interval} from $x$ to $y$ by
$$[x,y]=\{m(x,y,z)| z \in X\}.$$ 

\end{definition}

Similarly to the cubical context, we can use the median structure to define a notion of convexity:

\begin{definition}[Median convexity] \label{def:median convex} Let $X$ be a coarse median space, a subset $Y \subseteq X$ is said to be \emph{$K$-median convex} if there exists a constant $K \geq 0$ such that for any $x,y \in Y$ and $z \in X,$ we have $d(m(x,y,z), Y) \leq K$.  Equivalently, $Y$ is $K$-median convex if there exist some $K$ such that $[x,y] \in N(Y,K)$ whenever $x,y \in Y$. We will say that $Y$ is \emph{median convex} if it is $K$-median convex for some $K \geq 0.$
\end{definition}

For $x,y \in X,$ it is natural to expect that $[x,y]$ is a median convex set. This is indeed the case.

\begin{lemma} [{\cite[Lemma 2.21]{Niblo2019}}] \label{lem: intervals are convex}
For any $z \in [x,y],$ we have $d(\mu(x,y,z),z) \leq C$ for a constant $C$ depending only on the parameters of $X$. In particular, median intervals are themselves median convex.
\end{lemma}

Starting with a set $A,$ it is natural to ask about the smallest median convex set containing $A$. This requires introducing the following definition.

\begin{definition}(Joins)
Let $A$ be a subset of a coarse median space $X$. We define the \emph{median join} of $A$ by 

$$J(A):= \underset{x,y \in A}{\bigcup}[x,y].$$ We define $J^n(A)$ inductively by setting $J^0(A)=A$ and $J^n(A)=J(J^{n-1}(A)).$
\end{definition}

The process above terminates in finitely many steps. More precisely, we have the following.

\begin{lemma}[{\cite[Lemma 6.1]{Bowditch2019CONVEXITY}}]\label{lem: terminating joins} If $X$ is a coarse median space of rank $v,$ then there exists a constant $k$ depending only on the parameters of $X$, including $v$, such that $J^{v+1}(A) \subseteq N(J^v(A),k).$ In the special case where $X$ is a CAT(0) cube complex, the constant $k=0,$ that is, $J^{v+1}(A)=J^v(A)$.
\end{lemma}

In light of Lemma \ref{lem: terminating joins}, one can define the median hull of a finite set $A$ as follows.

\begin{definition} (Hulls)\label{def:median_hulls} Let $X$ be a coarse median space of rank $v$ and let $A$ be a subset of $X$. The \emph{median hull} of $A$, denoted $\hull(A)$ is defined by $$\hull(A):=J^v(A).$$

We remark that in the special case where $X$ is a CAT(0) cube complex, $\hull(A)$ can be equivalently defined by taking the intersection of all half spaces which properly contain $A.$ Consequently, a hyperplane $h$ in $X$ separates points in $A$ if and only if it intersects $\hull(A).$ Further, $\hull(A)$ is convex with respect to both the CAT(0) and combinatorial metrics on $X.$

\end{definition}

It is immediate from Lemma \ref{lem: terminating joins} that the median hull, $\hull(A)$, is median convex. In fact, in \cite{Bowditch2019CONVEXITY} it is shown that $\hull(A)$ is the smallest median convex set that contains $A$ in the following sense.

\begin{proposition}[{\cite[proposition 6.1]{Bowditch2019CONVEXITY}}]\label{prop: comparing hulls} Let $X$ be a coarse median space with rank $v$. There exists an $r \geq 0,$ depending on the parameters of $X$ such that the following holds: 

\begin{enumerate}
    \item (Convexity) The set $\hull(A)$ is $r$-median convex.
    \item (Comparing hulls) If $H$ is an $s$-median convex set containing $A$, then $\hull(A) \subseteq N(H,s')$ where $s'$ depends only on $s$ and the parameters of $X.$

\end{enumerate}

\end{proposition}

\begin{remark}

In the special case where $ X$ is a CAT(0) cube complex and $A \subseteq X,$ the set $\hull(A)=J^v(A)$ is combinatorially convex (i.e., convex in the $l^1$-metric). In particular, the constant $r$ in Proposition \ref{prop: comparing hulls} is $r=0.$ 
\end{remark}

We will be interested in maps which preserve the median structure, in particular median convexity.

\begin{definition} (Median maps) Let $K \geq 0$ and let $X,Y$ be two coarse median spaces. A map $f:X \rightarrow Y$ is said to be \emph{$K$-median} if for any $x,y,z \in X,$ we have 

$$d(f(m(x,y,z)), m(f(x), f(y), f(z)))<K.$$ 

\end{definition}

When dealing with coarse median spaces, the natural morphism to consider is the following.

\begin{definition}Let $X,Y$ be two coarse median spaces. A map $f:X \rightarrow Y$ is said to be a \emph{$K$-median quasi-isometric embedding} if:

\begin{enumerate}
    \item $f$ is a $K$-median map, and
    \item $f$ is a $(K,K)$-quasi-isometric embedding.
\end{enumerate}
If $f$ also satisfies $d_{Haus}(f(X),Y)\leq K$, then $f$ is said to be a \emph{$K$-median quasi-isometry}.
\end{definition}

The following lemma is left as an exercise for the reader:

\begin{lemma}\label{lem:preserving convexity}
Let $X,Y$ be coarse median spaces and let $f:X \rightarrow Y$ be a $K$-median quasi-isometry. We have the following:

\begin{enumerate}
    \item If $H \subseteq X$ is an $r$-median convex set, then $f(H)$ is an $r'$-median convex set where $r'$ depends only on $K,r$ and the parameters of $X,Y$.
    
    \item If $H' \subseteq Y$ is an $r$-median convex set, then $f^{-1}(H')$ is an $r'$-median convex set where $r'$ depends only on $K,r$ and the parameters of $X,Y$.
\end{enumerate}

\end{lemma}

The following lemma says that taking median hulls coarsely commutes with applying median quasi-isometries.  The proof, which we include for completeness, is a straight-forward consequence of Proposition \ref{prop: comparing hulls}(2) and Lemma \ref{lem:preserving convexity}.

\begin{lemma} \label{lem: commuting with hulls} Let $X,Y$ be two coarse median spaces of rank $v$ and let $f:X \rightarrow Y$ be some $K$-median quasi-isometry. There exists a constant $K'$, depending only on $v,K$, such that for any set $A \subset X$, we have $$d_{\text{Haus}}(f(\hull(A)), \hull (f(A)) \leq K'.$$

\end{lemma}

\begin{proof} Let $f$ be as in the statement. Using Proposition \ref{prop: comparing hulls} and Lemma \ref{lem:preserving convexity}, if $A \subseteq X,$ then  $f(\hull(A))$ is $r$-median convex set, where $r$ depends only on $K$ and the parameters of $X,Y.$ Since $f(\hull(A))$ is an $r$-median convex containing $f(A)$ , item (2) of Proposition \ref{prop: comparing hulls} yields $$\hull(f(A)) \subseteq N(f(Hull(A)), r'),$$ for some $r'$ depending only on the parameters of $Y$ and $r$.

On the other hand, using Lemma \ref{lem:preserving convexity}, the set $f^{-1}(\hull(f(A))$ is an $s$-median convex set containing $A$, where $s$ depends only on $r,K$ and the parameters of $X,Y$. Thus, using item (2) of Proposition \ref{prop: comparing hulls}, we have  $\hull(A) \subseteq N(f^{-1}(\hull(f(A)),s')$ where $s'$ depends only on $K$ and the parameters of $X,Y.$ In other words, every point of the set $H=\hull(A)$ is within $s'$ of $H'=f^{-1}(\hull(f(A))$. Since $f$ is a quasi-isometry, there exists $s''$, depending only on $K$ and $s'$ such that every point in $f(H)$ is within $s''$ of $f(H')$. But notice that $f(H')=f(f^{-1}(\hull(f(A))) \subseteq \hull(f(A)).$ Hence, $f(\hull(A))=f(H) \subseteq N(\hull (f(A), s'').$ Hence

$$f(\hull(A))\subseteq N(\hull (f(A)), s'').$$ 

\end{proof}

\subsection{LQC spaces} \label{subsec:LQC prelim}

The category of objects we will consider in this subsection are coarse median spaces which are locally approximated by cube complexes: \emph{locally quasi-cubical} (LQC) spaces.

Our motivation comes from work of Behrstock-Hagen-Sisto \cite[Theorem F]{HHS_quasi}, where they show that hierarchically hyperbolic spaces have this property; see also \cite{Bowditch2019CONVEXITY}.

The following two definitions (Definition \ref{def: cubical approximations} and Definition \ref{def: locally quasi-cubical spaces}) extract this property to the general setting.  It generalizes the perspective of \cite{HHP}, who studied coarse median spaces with quasicubical median intervals.  This LQC property is strictly stronger than the quasicubical interval property, and arguably more useful, as exhibited by our techniques that require the flexibility of adding multiple additional points.  The theory of LQC spaces as interesting spaces in their own right is currently being developed by Petyt, Spriano and the second author in the forthcoming \cite{PSZ}.

\begin{definition} (Cubical approximations)\label{def: cubical approximations} Let $X$ be a coarse median space and let $v \in \mathbb{N} \cup \{0\},$ $K \geq 0$. A finite subset $F \subseteq X$ is said to admit a \emph{ $(v,K)$-cubical approximation} if there exist a CAT(0) cube complex $Q$ of dimension $v$ and a $K$-median quasi-isometry $f:Q \rightarrow \hull (F)$.

\end{definition}

We now are ready to introduce the notion of a locally quasi-cubical space.

\begin{definition} (locally quasi-cubical spaces)\label{def: locally quasi-cubical spaces}
A coarse median space $X$ is said to be \emph{locally quasi-cubical} if there exists a positive integer $v$ such that every finite set of points $F \subseteq X$ admits some $(v,K)$-cubical approximation where $K$ depends only on $|F|.$ The integer $v$ is referred to as the \emph{dimension} of $X.$

\end{definition}

The natural types of paths to consider in LQC spaces are those which are coarsely preserve medians.

\begin{definition}[median paths/rays]\label{def:median ray} Let $\lambda \geq 1.$ A \emph{$\lambda$-median path} is a $(\lambda, \lambda)-$quasi-isometric embedding $\alpha:[a,b] \rightarrow X$ which is $\lambda$-median. Similarly, a \emph{$\lambda$-median ray} is a $(\lambda, \lambda)$-quasi-isometric embedding $h: [0, \infty) \rightarrow X$ which is $\lambda$-median.
\end{definition}

We will use the following lemma frequently, and we record its proof for clarity. Similar arguments have appeared in \cite{HHS_quasi} and \cite{DMS20}.

\begin{lemma}\label{lem: existence of median paths}
Let $X$ be a locally quasi-cubical space. There exists a $\lambda \geq 1$ such that every two points of $X$ are connected by some $\lambda$-median path.
\end{lemma}

\begin{proof}
Let $x,y$ be points in $X$, let $f:Q \rightarrow [x,y]$ be a $\lambda$-median quasi-isometry and let $g$ be a coarse inverse of $f$. For any geodesic $\alpha$ connecting $g(x), g(y)$ in $Q,$ the path $ f\circ \alpha$ is a $\lambda$-median path. Since $\lambda$ is uniform, we are done.
\end{proof}

\subsection{Hierarchically hyperbolic spaces}\label{subsec:HHS prelim}

We conclude the preliminaries with a brief discussion of \emph{hierarchically hyperbolic spaces} (HHSes), which are the main examples of non-cubical LQC spaces.  Our eventual main purpose in this subsection is to describe the coarse median structure on an HHS, as well as the existence and structure of gate maps to median convex subsets of HHSes. In light of the complexity of the definition of an HHS, we will only highlight the features that we require for our analysis.  See \cite{sisto} for an overview of the theory.

Roughly, an HHS is a pair $(\calX, \mathfrak S)$ where $\calX$ is a geodesic metric space and $\mathfrak S$ is set indexing a family of uniformly hyperbolic spaces $\calC(U)$ for each $U \in \mathfrak S$ such that the following hold:

\begin{enumerate}

\item For each \emph{domain} $W \in \mathfrak S$ there exist an $E$-hyperbolic space $\calC W$ and a coarsely surjective $(E,E)$-coarsely Lipschitz map $\pi_W: \calX \rightarrow \calC W,$ with $E$ independent of $W$.
    \item The set $\mathfrak S$ satisfies the following:
    \begin{itemize}
        \item It is equipped with a partial order called \emph{nesting}, denoted by $\nest$. If $\mathfrak S \neq \emptyset,$ the set $\mathfrak S$ contains a unique $\nest$-maximal element $S.$
        
        \item It has a symmetric and  anti-reflexive relation called \emph{orthogonality}, denoted by $\perp$. Furthermore, domains have \emph{orthogonal containers}: if $U \in \mathfrak S$ and there is a domain orthogonal to $U$ then there is a domain $W$ such that $V \nest W$ whenever $V \perp U.$
        
        \item For any distinct non-orthogonal $ U, V \in \mathfrak S$, if neither is nested into the other, then we say $U,V$ are \emph{transverse} and write $U  \pitchfork V.$ 
        \item There exists a map $\theta:[0,\infty) \rightarrow [0,\infty)$ such that for any $D \geq 0,$ if $d_\calX(x,y) \geq \theta(D),$ then $d(\pi_U(x),\pi_U(y))>D$ for some $U \in \mathfrak S.$ We will refer to this property as the \emph{uniqueness} property. \label{def:uniqueness}
        \end{itemize}
        
        \item There exists an integer called the \emph{complexity} of $\calX$ and denoted $\xi(\mathfrak S)$ such that whenever $U_1,U_2,\cdots U_n$ is a collection of pairwise non-transverse domains, then $n \leq \xi(\mathfrak S).$
        
        \item If $U  \sqsubset V$ or $U  \pitchfork V,$ then there exists a set $\rho^U_V$ of diameter at most $E$ in $\calC V.$
        
        \item If $U  \sqsubset V$, there is a map $\rho^V_U:\calC(V) \rightarrow \calC(U)$ which satisfies the following: For $x \in \calX$ with $d_{ V}(\pi_V(x),\rho^U_V)>E,$ we have $\pi_U(x) \underset{E}{\asymp} \rho^V_U(\pi_V(x)).$

\end{enumerate}

The constant $E$ above will be referred to as \emph{the HHS constant.}
For two points $x,y \in \calX,$ it is standard to use $d_U(x,y)$ to denote $d_{\calC U}(\pi_U(x), \pi_U(y)),$ and similarly for subsets of $\calX.$ For a subset $A \subseteq \calX,$ we will also use $\diam_U(A)$ to denote the diameter of the set $\pi_U(A) \subset \calC(U)$.

We will frequently use the following HHS axiom, so we record it separately:

\begin{axiom}(Bounded geodesic image)\label{BGIA}  If $U  \sqsubset V$ and $\gamma \in \calC(V)$ is a geodesic with $d_{V}(\rho^U_V,\gamma)>E$, then diam$(\rho^V_U(\gamma))\leq E.$ Furthermore, for any $x,y \in \calX$ and any geodesic $\gamma$ connecting $\pi_V(x),\pi_V(y)$, if $d_{V}(\rho^U_V,\gamma)>E$, then $d(\pi_U(x), \pi_U(y)) \leq E.$
\end{axiom}

 The following theorem of Behrstock-Hagen-Sisto is the main inspiration for the definition of local quasi-cubicality. 

\begin{theorem}[{\cite[Theorem F]{HHS_quasi}}]\label{thm:HHS LQC}
Every hierarchically hyperbolic space is locally quasi-cubical. Furthermore, there exists some $\lambda \geq 0$, depending only on the HHS constant $E$ such that for any domain $U,$ the map $\pi_U: \calX \rightarrow \calC(U)$ is $\lambda$-median.
\end{theorem}

 See also \cite{DMS20} for a stabilization of the LQC property for most HHSes, and \cite{Petyt21} for a proof that mapping class groups are globally (nonequivariantly) quasicubical.

Another useful HHS construction are \emph{hierarchy paths} \cite{MM00, Dur16} and \cite[Theorem 4.4]{HHS_II}. These are uniform quasi-geodesics which project to unparametrized quasi-geodesics in $\calC U$ for each $U \in \mathfrak S$. Similarly, \emph{hierarchy rays} are quasi-geodesic rays which projects to unparametrized quasi-geodesics in each $\calC U$.
\begin{lemma}[{\cite[Lemma 1.37]{HHS_quasi}} and {\cite[Lemma 3.3]{DHS17}}]\label{lem:infinite_projection_median} Let $\calX$ be an HHS, we have the following. Every $\lambda$-median path in $\calX$ is a $\lambda'$-hierarchy path, where $\lambda'$ depends only on $\lambda$ and on the constants of $\calX$. Furthermore, for any median ray $h$, there exists $U \in \mathfrak S$ such that $\diam_U(h)=\infty.$

\end{lemma}

In fact, in light of the above lemma and Lemma \ref{lem: existence of median paths}, which allows us to build median paths from the cubical models, we will never mention of hierarchy paths/rays again in this paper. Our canonical paths connecting pairs of points are just median paths and those are hierarchy paths by Lemma \ref{lem:infinite_projection_median}.

Another fundamental structure of an HHS is its family of standard product regions, which correspond to stabilizers of multicurves in the mapping class group setting \cite{MM00}.   

\begin{definition} (Standard product region, dominion)\label{def:standard_product_region}. For $U \in \mathfrak S$ the \emph{standard product region}
associated to $U$ is the set $P_U:=\{x \in \calX|\,d_V(x,\rho^U_V)\leq E \text{ for } U \nest V \text{ or } U \pitchfork V\}.$ Fix a basepoint
$p \in \calX$. The \emph{dominion} of $U$ (relative to $p$), denoted $F_U$, is the set of points $x \in P_U$ for which $d_V(x, p)\leq E$ for all $V \perp U$ and for the orthogonal container of $U$.
\end{definition}

\begin{remark}\label{rmk:product_region_properties} We list some properties of the dominion $F_U,P_U$ for more details, see \cite{HHS_II}.

\begin{enumerate}
    \item The sets $F_U, P_U$ are median convex. 
    
    \item The dominion $F_U$ is an HHS with respect to the collection $\mathfrak S_U=\{V \in \mathfrak S|V \sqsubseteq U\}$. In fact, if $x \in \calX$, there exists a coarsely unique $x' \in F_U$ such that $d_V(x', x) \leq E$ for all $V \sqsubseteq U.$
    
    \item If $\calX$ is an HHS and $U$ is a domain, then there exists a median convex subset of $\calX$, denoted $E_U$, such that there is an $E$-median quasi-isometry $P_U \to F_U \times E_U$, where both $F_U,E_U$ are given with respect to the restriction of the distance function on $\calX.$

\end{enumerate}

\end{remark}

The geometry of the median hull of two points is encoded in its \emph{relevant domains}:

\begin{definition}
For two points $x,y \in \calX$ and a constant $\theta \geq 0,$ we define $$\mathrm{Rel}_{\theta}(x,y):=\{U \in \mathfrak S|d_U(x,y)> \theta \}.$$
\end{definition}

The following distance formula, generalizing work of Masur-Minsky \cite{MM00}, is one of the main HHS tools.  The second item follows directly from the first, but we record it for later convenience:

\begin{theorem}\label{thm:distance formula} Let $\calX$ be an HHS. There exist a constant $\theta_0$, and map $K:[\theta_0, \infty) \rightarrow [0,\infty)$ such that for each $\theta \geq \theta_0$ the following hold:

\begin{enumerate}
    \item For any $x,y \in \calX$, we have $$d_\calX(x,y) \underset{K(\theta)}{\asymp}\underset{U \in \mathrm{Rel}_{\theta}(x,y)}{\sum} d_U(x,y).$$
    
    \item In particular, for any domain $V$ and $x,y \in F_V$, we have $$d_\calX(x,y) \underset{K(\theta)}{\asymp}\underset{U \in \mathrm{Rel}_{\theta}(x,y) \cap \mathfrak S_V}{\sum} d_U(x,y).$$
\end{enumerate}

\end{theorem}

We remark that Bowditch \cite{Bowditch2019CONVEXITY} obtained an analogue of the above distance formula as a consequence of the LQC property in a slightly more general axiomatic setting than HHSes.

There is a direct connection between the relevant set $\mathrm{Rel}_{\theta}(x,y)$ and the product regions that a median path or ray must visit:

\begin{proposition}[Active sub-paths, Corrected version of {\cite[proposition 5.17]{HHS_II}}]\label{prop: active path}

For all sufficiently large $D$ (in terms of the HHS constants), there exists
$v$ such that the following holds. Let $x, y \in \calX$, let $\gamma$ be a $D$-median path connecting $x,y$
path and let $U$ be a domain with $d_U(x,y) \geq 200DE.$ Then, $\gamma$ has a subpath $\beta$ such that:

\begin{enumerate}
    \item $\beta \subseteq \mathcal{N}_v(P_U).$
    
    \item $\pi_U$ is $v$-coarsely constant on any sub-path of $\gamma$ disjoint from $\beta.$
\end{enumerate}

\end{proposition}

We will frequently use the following ``passing up'' lemma:

\begin{lemma}[{\cite[Lemma 2.5]{HHS_II}}]\label{lem:passing-up}

Let $\calX$ be an HHS with constant $E$.
For every $C>0$ there is an integer $m=m(C)$ such that if $V \in \mathfrak S$ and $x,y \in \calX$ satisfy $d_{U_i}(x,y)>E$ for a collection of domains $\{U_i\}_{i=1}^m$ with $U_i \in \mathfrak S_V$, then there exists $W \in \mathfrak S_V$ with $U_i \sqsubsetneq W$ for some $i$ such that $d_W(x,y)>C.$
\end{lemma}

In order to get our characterization of the $\kappa$-Morse property for a median ray in an HHS via the sublinear growth of subsurface projections (Theorem \ref{thm:bounded projection characterization}), we will need to know when a ray passes close to some $P_U$, then that $P_U$ is actually a quasiflat. 

The following definition eliminates this potential pathology of general HHSes:

\begin{definition}[Bounded domain dichotomy] \label{defn:bounded domain dichotomy}
An HHS $\calX$ has the \emph{bounded domain dichotomy} if there exists $B>0$ so that if $U \in \mathfrak S$ has $\mathrm{diam}(\calC(U)) > B$, then $\mathrm{diam}(\calC (U))  = \infty$.
\end{definition}

All of the relevant HHSes (and all HHGs) satisfy the bounded domain dichotomy.  It was introduced in \cite{ABD} to prove the following:

\begin{theorem}[ABD structure] \label{thm:ABD}
Any hierarchically hyperbolic space $\calX$ satisfying the bounded domain dichotomy admits an HHS structure $(\calX, \mathfrak S)$ so that for all $U \in \mathfrak S$ with $U \neq S$, we have that both $F_U$ and $E_U$ are infinite diameter.
\end{theorem}

Going forward, we make the blanket assumption that $\calX$ is a proper HHS with unbounded products and its ABD structure as in the theorem above.

The next lemma explains the basic properties of  medians, as well as median convex sets and their gate maps:

\begin{lemma}[{\cite[Lemma 5.5]{HHS_II}}, {\cite[Corollary 5.12]{RST18}}]\label{lem:gate}
Let $\calX$ be an HHS.  There exists $E' = E'(\calX) >0$ so that for any median convex set $Y$, there exists a map $\gate_Y: \mathcal{X} \rightarrow Y$, called the \emph{gate}, satisfying the following properties:

\begin{enumerate}

    \item For any $x,y,z \in \calX$, the median $m = m_{\calX}(x,y,z)$ satisfies $d_U(m, m_U(x,y,z)) < E'$. 

    \item If $Y$ is median convex, then $\pi_U(Y)$ is $E'$-quasiconvex for each $U \in \mathfrak S$.

    \item $\gate_Y$ is $E'$-coarsely Lipschitz, that is $d_{\calX}(\gate_Y(x),\gate_Y(x')) \leq E'd_{\calX}(x,x') + E'$ for all $x,x' \in \calX$.
    
    \item For each $x \in \calX,$ the set $\pi_U(\gate_Y(x))$ coarsely agrees with the projection of $\pi_U(x)$ to $\pi_U(Y).$
    
    \item For $x \in \calX,$ if $x'$ is a nearest point projection of $x$ to $Y,$ then $$d(x,x') \underset{E'}{\asymp}d(x,\gate_H(x)).$$
    
\end{enumerate}

\end{lemma}

Another useful consequence of this fact and the preceding discussion is the following lemma, which gives an alternative way to define the gate map to a median convex subset.

\begin{lemma}\label{lem: gates are median}
Let $Y \subset \calX$ be a $K$-median convex subset where $\calX$ is an HHS. There exists a constant $C$, depending only on $K$ and $\calX$ such that the following holds. 

\begin{enumerate}
    \item For any $x \in \calX$ and any $y \in Y$, we have that $m(x, \gate_Y(x), y) \underset{C}{\asymp} \gate_Y(x)$. That is, $$\gate_Y(x)\underset{C}{ \in } \bigcap_{y \in Y} [x,y].$$
    
    \item For any $z \in Y$ with $\displaystyle z \underset{C}{ \in } \bigcap_{y \in Y} [x,y]$, we have $ z\underset{C}{\asymp} \gate_Y(x)$.
    
\end{enumerate}

\end{lemma}

\begin{proof}

Let $x \in \calX$, $y \in Y$ and let $z=\gate_Y(x)$. We claim that $m(x,z, y)$ coarsely agrees with $z$. By construction, $\pi_U(z)$ coarsely coincides with the closest point projection of $\pi_U(x)$ to $\pi_U(Y)$ , for every $U$. Furthermore, by hyperbolicity of $\calC(U)$, for any $x' \in \pi_U(x), y' \in \pi_U(Y)$ and any $z'$ living in the closest point projection of $\pi_U(x)$ to $\pi_U(Y)$, the point $m_U(x',z',y')$ coarsely agrees with $z'$.  Hence $z$ coarsely coincides with $m(x,z,y)$ by the uniqueness property \eqref{def:uniqueness} of HHSes.  This finishes the proof of (1).

For (2), if $z \in Y$ satisfies $z \underset{C}{ \in } \underset{y \in Y}{\bigcap } [x,y]$, we have $z \in [x,\gate_Y(x)]$. Hence, we have $\gate_Y(x)\underset{C}{\asymp}m(x, \gate_Y(x), z)\underset{C}{\asymp}z$ where the last equality holds by Lemma \ref{lem: intervals are convex}.  This completes the proof.


\end{proof}

Combining item (2) of Lemma \ref{lem: gates are median} with Lemma \ref{lem:gate median CCC} shows that for a combinatorialy convex set $Y$ of a CAT(0) cube complex $Q$ which is also an HHS, the map $\gate_Y$ coarsely agrees with the combonatorial nearest point projection $P_Y.$

Item (2) above yields a description of gates via medians, as it says that if $Y$ is a median convex set, the gate of a point $x \in \calX$ to $Y$ is coarsely the unique point in $Y$ which lives in the median interval $[x,y]$ for all $y \in Y$.  Hence median quasi-isometries  not only preserve distances, but also gate maps. In particular, this yield the following corollary.

\begin{corollary} \label{cor: comparing gates}
Let $\mathcal{X}, \mathcal{Y}$ be two HHSes, and $f:\mathcal{Y} \rightarrow \mathcal{X}$ be $K$-median quasi-isometric embedding.  Suppose that $Y_1,Y_2 \subset \mathcal{Y}$ are two  $K$-median convex sets. There exists a constant $C$, depending only on $K,$ $\calX$, and $\mathcal{Y}$ such that:

    \begin{enumerate}
        \item For each $x \in Y_1$, we have $f(\gate_{Y_2}(x))\underset{C}{\asymp}\gate_{f(Y_2)}(f(x)).$

  \item $\diam_{\mathcal{Y}}( \gate_{Y_2}(Y_1))\underset{C}{\asymp} \diam_{\calX}( \gate_{f(Y_2)}(f(Y_1))).$

\end{enumerate}

\end{corollary}

\begin{proof} Since $Y_1,Y_2$ are $K$-median convex, we get a constant $C$ such that Lemma \ref{lem: gates are median} holds. Let $y' \in f(Y_2)$ and let $y \in Y$ with $f(y)=y'.$ Using Lemma \ref{lem: gates are median}, we have $\gate_{Y_2}(x)\underset{C}{\asymp} m(x,y,\gate_{Y_2}(x))$. Thus, $f(\gate_{Y_2}(x))\underset{C}{\asymp}f(m(x,y,\gate_{Y_2}(x)))\underset{C}{\asymp}m(f(x), y', f(\gate_{Y_2}(x)))$. This shows that for an arbitrary $y' \in f(Y_2),$ we have $f(\gate_{Y_2}(x))\underset{C}{\asymp} m(f(x), y', f(\gate_{Y_2}(x))).$ Hence, by part (2) of Lemma \ref{lem: gates are median}, we get that 

$$f(\gate_{Y_2}(x))\underset{C}{\asymp}\gate_{f(Y_2)}(f(x)).$$ 

Part (2) follows immediately from part (1) and the assumption that $f$ is a quasi-isometry.
\end{proof}

\section{Preliminaries on sublinear Morseness in CAT(0) spaces and cube complexes} \label{sec:CCC}

In this section, we will work in a fixed CAT(0) space $X$.  The first subsection assumes that $X$ is a CAT(0) space, but the last two subsections are about CAT(0) cube complexes.

Since most of the results in this paper involve exporting problems to appropriate CAT(0) cube complexes via median quasi-isometries, the work in this section plays an important supporting role.

\subsection{Simple description of Morseness in CAT(0) spaces} The following theorem states that in a CAT(0) space $X,$ all notions of $\kappa$-Morseness are equivalent.

\begin{theorem}[{\cite[Theorem 3.10]{QRT19}}] \label{thm: CAT(0) all equivalent}
Let $\alpha$ be a quasi-geodesic ray in a CAT(0) space. The following are all equivalent:

\begin{enumerate}
    \item $\alpha$ is $\kappa$-Morse.
    \item $\alpha$ is $\kappa$-contracting.
    \item $\alpha$ is weakly $\kappa$-Morse.

\end{enumerate}
\end{theorem}

\subsection{Median of quasi-geodesic rays in CAT(0) cube complexes}

The main goal of this subsection is to characterize $\kappa$-fellow traveling of two quasi-geodesic rays in a cube complex via the median (Corollary \ref{cor:medians_vs_morseness_CAT(0)}).   It is essential for establishing injectivity of the map from the $\kappa$-boundary of an HHS into the boundary of its top level curve graph (Theorem \ref{thm:map}).

For the rest of this section, we now assume that $X$ is a finite dimensional CAT(0) cube complex.

We note that versions of many of the statements in this subsection were proven in \cite{IMZ21} for geodesics, though we need them for quasi-geodesics, which requires some extra work.

\begin{lemma} \label{lem: crossing well-separated hyperplanes implies linear divergence}

Let $q,q'$ be two quasi-geodesic rays in $X$ starting at the same vertex $\go$ and let $h_1,h_2$ be two $L$-well-separated hyperplanes with $\go \in h_i^-$ for $i=1,2.$ If there exists $t_0$ such that $q(t) \in h_2^+$ for $t \geq t_0$ and $q'(t) \in h_1^-$, then $d(q(t), q'(t))$ is bounded below by a linear function in $t.$

\end{lemma}

\begin{proof}
Let $t_0$ be as in the statement. Notice that among every hyperplane separating  $q(t_0),q(t)$ at most $L$ can intersect both $h_1,h_2.$ Therefore, if $t \geq t_0,$ then, the vertices $q(t), q'(t)$ are separated by at least $(d(q(t), q
(t_0))-L) \asymp t$ hyperplanes, this gives the desired statement.
\end{proof}

This gives the following corollary.

\begin{corollary} \label{cor: no sneak-like behaviour for Morse quasi geodesics}
Let $q$ be a continuous weakly $\kappa$-Morse quasi-geodesic ray which crosses two $L$-well-separated hyperplanes $h_1,h_2$ in that order. Then $q$ can cross both $h_1,h_2$ only finitely many times.

\end{corollary}

\begin{proof}
Suppose for the sake of contradiction that $q$ crosses $h_1,h_2$ infinitely many times. Since $h_1,h_2$ are $L$-well-separated, they must be disjoint. Thus, without loss of generality, we can assume that $(h_2)^+ \subset (h_1)^+$. Since $q$ crosses both $h_1,h_2$ infinitely many times, there must exist two infinite sequences $\{s_i\}, \{t_j\}$ such that $q(s_i) \in  (h_1)^- $ and $q(t_j) \in  (h_2)^+$. We may assume that $\go=q(0) \in (h_1)^-$.

Now, let $g_i$ be a sequence of  combinatorial geodesic segments connecting $q(0)$ to $q(s_i)$. Let $p$ be the combinatorial projection of the vertex $\go$ to $h_2.$ Let $w_j$ be a sequence of combinatorial geodesic segments connecting $p$ to $q(t_j)$ and let $g_j$ be the concatenation of $w_j$ with a geodesic connecting $\go$ to $p$. Since $p$ is the combinatorial projection of $\go$ to $h_2,$ this concatenation is a geodesic. That is, the path $g_j$ is a geodesic segment starting at $\go$ and ending at $q(t_j).$ Applying Arzelà–Ascoli to the two sequences of geodesic segments $g_i, g_j$ yields two $\kappa$-fellow traveling combinatorial geodesic rays $g_1,g_2$ such that $g_1 \in (h_1)^-$ and $g_2[t_0,\infty) \in h_2^+$ for some $t_0.$ However, using Lemma \ref{lem: crossing well-separated hyperplanes implies linear divergence}, the geodesics $g_1, g_2$ must diverge linearly which is a contradiction.
\end{proof}

\begin{lemma}\label{lem: unique hull point}

Let $q$ be a continuous $\kappa$-Morse quasi-geodesic ray and let $Y=\hull(q)$ be its median hull.  If $\partial Y$ denotes the visual boundary of $Y$, then $|\partial Y|=1.$

\end{lemma}

\begin{proof}
Recall that $Y=\hull(q)$ is convex in both the CAT(0) and the combinatorial metrics. Let $[q(0), q(i)]$ be a sequence of CAT(0) geodesic segments. Applying Arzelà–Ascoli gives a $\kappa$-Morse CAT(0) geodesic ray $b$ living in $Y$.

Let $c$ be another geodesic ray in $Y$ starting at $b(0)$ (not necessarily $\kappa$-Morse).  We claim that $c=b.$ Suppose not. Since $\kappa$-Morse geodesic rays define visibility points in the visual boundary, we get a combinatorial geodesic line, denoted $l$, connecting the points $b(\infty), c(\infty)$ in the visual boundary of $Y$. Without loss of generality, suppose that $l(\infty)=b(\infty)$ and $l(-\infty)=c(\infty).$ Since the geodesic ray $b'=[l(0), l(\infty))$ is $\kappa$-Morse, by Theorem \ref{thm: characterize_geodesic_rays_in cube_complexes}, it must cross an infinite sequence of well-separated hyperplanes $\{h_1,h_2,h_3, \dots\}$. In particular, the hyperplanes $h_1,h_2$ are $L$-well-separated for some $L$. Using Corollary \ref{cor: no sneak-like behaviour for Morse quasi geodesics}, there exists some $t_0$ such that $b([t_0, \infty))$ does not cross $h_2.$ On the other hand, the geodesic ray $c'=[l(0), l(\infty))$ crosses an infinite chain of hyperplanes $\{k_1,k_2,\dots\}$ without any facing triples.  Since $h_1,h_2$ are $L$-well-separated and since $\{k_1,k_2,\dots\}$ contain no facing triples, at most $L$ hyperplanes among  $\{k_1,k_2,\dots\}$ can cross both $h_1, h_2$. This implies that there exists an integer $n$ such that $\{k_n,k_{n+1},k_{n+2}\dots\}$ does not cross $h_2$. Therefore, $h_2$ separates $\{k_n,k_{n+1},k_{n+2}\dots\}$ from $b([t_0, \infty))$. This contradicts the fact that the hyperplanes $\{k_n,k_{n+1},k_{n+2}\dots\}$ are all met by $q$, by definition of $\hull(q)$.

\end{proof}

\begin{lemma} \label{lem: median relates to hyperplanes}
If $q,q'$ are quasi-geodesic rays starting at $\go$ such that $m(\go, q(t_i), q'(t_j)) \rightarrow \infty$ as $t_i,t_j \rightarrow \infty$, then both $q,q'$ cross infinitely many of the same hyperplanes.
\end{lemma}

\begin{proof} 

Recall that the median $m(x,y,z)$ is characterized by being the unique vertex resulting from orienting every hyperplane $h \subset X$ towards the majority of $x,y,z$ and taking the intersection of the resulting half-spaces (Definition \ref{def:CCC median}). Let $a_{i,j}=m(\go, q(t_i), q(t_j))$ and let $\alpha_{i,j}$ be a geodesic connecting $\go$ to $a_{i,j}$. Using the definition of the median, every hyperplane crossing $\alpha_{i,j}$ must separate $\go$ from $\{q(t_i), q'(t_j)\}$. Hence, every such hyperplane crosses both $q,q'.$
\end{proof}

\begin{lemma} \label{lem: infinity many hyperplanes implies fellow travelling}
If $q,q'$ are $\kappa$-Morse quasi-geodesic rays that cross infinitely many of the same hyperplanes, then $q \sim_{\kappa} q'$.
\end{lemma}

\begin{proof} Let $\go=q(0)=q'(0)$ and let $Y_1, Y_2$ denote $\hull(q_1)$, $\hull(q_2)$ respectively. As the quasi-geodesics $q_1,q_2$ cross infinitely many of the same hyperplanes, these same hyperplanes also cross both of $Y_1,Y_2$, by definition of the hull.  Denote these hyperplanes by $\{h_1,h_2, h_3, \dots\}$.

Since $Y_1 \cap Y_2, Y_1 \cap h_i, Y_2 \cap h_i$ are all non-empty convex subcomplexes, the Helly property implies that there exists a common intersection point $x_i \in Y_1 \cap Y_2 \cap h_i$. Now, consider the sequence of CAT(0) geodesic segments $[\go, x_i]$. Using convexity of both $Y_1, Y_2$, we have $[\go, x_i] \in Y_1 \cap Y_2$. Up to passing to a subsequence, the sequence $[\go, x_i]$ converges to a geodesic ray $b \in Y_1 \cap Y_2$. Since $|\partial Y_1|=1$ and $q$ is $\kappa$-Morse, the geodesic ray $b$ must be $\kappa$-Morse and it $\kappa$-fellow travels $q.$ Since $q'$ is also $\kappa$-Morse, we have $|\partial Y_2|=1$, therefore, $b$ is the unique geodesic ray living in $Y_2$ with $b(0)=\go.$ Hence, $q'$ must $\kappa$-fellow travel $b$ as well. Since both $q$ and $q'$ do $\kappa$-fellow travel $b$, they must $\kappa$-fellow travel, completing the proof.
\end{proof}

As a corollary of the previous two lemmas, we obtain the main result of this subsection.

\begin{corollary}\label{cor:medians_vs_morseness_CAT(0)} Let $q,q'$ be continuous $\kappa$-Morse quasi-geodesic rays with $q(0)=q'(0)$. We have  $$d(\go, m(q(t_i),q'(t_j), \go)) \rightarrow \infty \iff q \sim_{\kappa} q'.$$

\end{corollary}

\begin{proof} The forward direction follows by combining Lemma \ref{lem: median relates to hyperplanes} and Lemma \ref{lem: infinity many hyperplanes implies fellow travelling}. For the backwards direction, we argue by contradiction. Assume that there exist points $t_i,t_i \rightarrow \infty$ and a bounded set $M$ containing $\go$ and  such that $m(t_i,t_j, \go) \in M$ for all $i,j$. This implies that there exists a combinatorial geodesic $g_{i,j}$ connecting $t_i,t_j$ which goes through $M$, hence, by applying Arzel\`a-Ascoli we get a geodesic line $l:(-\infty, \infty) \rightarrow X$. After possibly enlarging $M$ by a finite amount, we can choose a point $ l(s) \in M$. It is immediate from the definitions that $l_{[s,\infty)}$ and $l_{(-\infty,s]}$ are $\kappa$-Morse geodesic rays representing $q, q'$, hence, $l_{[s,\infty)}, l_{(-\infty,s]}$ $\kappa$-fellow travel each other which is not possible because $l$ is a geodesic.

\end{proof}

\subsection{Excursions}

In this subsection, we give a characterization of $\kappa$-Morse quasi-geodesics in a CAT(0) cube complex via a sequence of a certain type of hyperplanes we call \emph{excursion}.

The definition is motivated by the definition of $\kappa$-excursion for geodesics in \cite{Murray-Qing-Zalloum} (see the statement of Theorem \ref{thm: characterize_geodesic_rays_in cube_complexes}).  In our setting, we need to work with quasi-geodesics, and it is more convenient to encode the excursion property into hyperplanes, as follows.

\begin{definition}\label{def:excursion}
Let $X$ be a CAT(0) cube complex and let $\kappa$ be a sublinear function. A chain of hyperplanes $\{h_i\}_{i \in \mathbb{N}}$ is said to be \emph{$\kappa$-excursion} if there exists a constant $c$ and points $x_i \in h_i$ such that:

\begin{enumerate}
    \item $h_i,h_{i+1}$ are $c \kappa(x_i)$-well-separated.
    \item $d(x_i, x_{i+1}) \leq c \kappa(x_i).$ 
\end{enumerate}

We note that the assumption that a collection of $\kappa$-excursion hyperplanes is a chain (Definition \ref{defn:chain}) is necessary and important.

\begin{remark}
Given a chain of $\kappa$-excursion hyperplanes $\{(h_i,x_i)\}_{i \in \mathbb N}$, we note that since the $h_i,h_{i+1}$ are $c\kappa(x_i)$-well-separated, the images of the combinatorial gate map $P_{h_{i+1}}|_{h_i}:h_i \to h_{i+1}$ has diameter bounded above by $c'c\kappa(x_i)$ where $c'$ depends only on the dimension of the cube complex (Proposition \ref{prop:Genevois}).  Since the gate map is a closest point projection and $d(x_i,x_{i+1}) \leq c \kappa(x_i)$, it follows that $d(x_{i+1}, P_{h_{i+1}}(h_i)) \leq 2c\kappa(x_i)$.  That is, $x_{i+1}$ is $\kappa$-close to $P_{h_{i+1}}(h_i)$.

This aligns philosophically with how we use the cubical models in Section \ref{sec:HHS char} to  characterize $\kappa$-Morse median paths in an HHS.
\end{remark}

The points $x_i$ above are referred to as the \emph{excursion points}. Similarly, the constant $c$ is called the \emph{excursion constant}. We will sometime use the notation $\{(h_i,x_i)\}_{i \in \mathbb{N}}$ to refer to the hyperplanes $h_i$ and the points $x_i$ above.

\end{definition}

The first lemma says that any geodesic crossing an infinite chain of excursion hyperplanes crosses them $\kappa$-close to the excursion points.

\begin{lemma} \label{lem: key to contraction}
Let $\{h_i\}_{i \in \mathbb{N}}$ be a chain of $\kappa$-excursion hyperplanes with excursion points $\{x_i\}_{i \in \mathbb{N}}$ and excursion constant $c.$ If $b$ is a geodesic connecting two points $x,y$ such that $x$ is between $h_{i-2},h_{i-1}$ and $y$ is between $h_{i+1}, h_{i+2}$, then every geodesic (combinatorial or CAT(0)) connecting $x,y$ must cross $h_{i}$ at a point $z$ with $d(z, x_{i}) \leq 4c \kappa (x_i).$
\end{lemma}

\begin{proof} The argument for this is essentially the same as the proof of \cite[Lemma 3.2]{IMZ21}; see also \cite[Lemma 4.15]{Murray-Qing-Zalloum}, but we include it for completeness. Let $[x,y]$ be a geodesic which crosses $h_i$ at $z$, as in the statement. We need to bound the number of hyperplanes separating $x_i$ and $z.$ Denote the collection of such hyperplanes by $\calH.$ Every hyperplane in $\calH$ which separates neither $x_{i-1},x_i$ nor $x_i, x_{i+1}$ must cross either $h_{i-1}$ or $h_{i+1}$. Hence, the cardinality of such hyperplanes is at most $c\kappa(x_{i-1})+ c \kappa(x_i) \leq 2c \kappa(x_i)$ This gives us

\begin{align*}
    |\calH|&\leq d(x_{i-1}, x_i)+d(x_i, x_{i+1})+2c \kappa(x_i)\\
    &\leq c\kappa(x_{i-1})+c \kappa(x_i)+2c \kappa(x_i)\\
    &\leq 4c \kappa(x_i).
\end{align*}

\begin{figure}
    \centering
    \includegraphics[width=.5\textwidth]{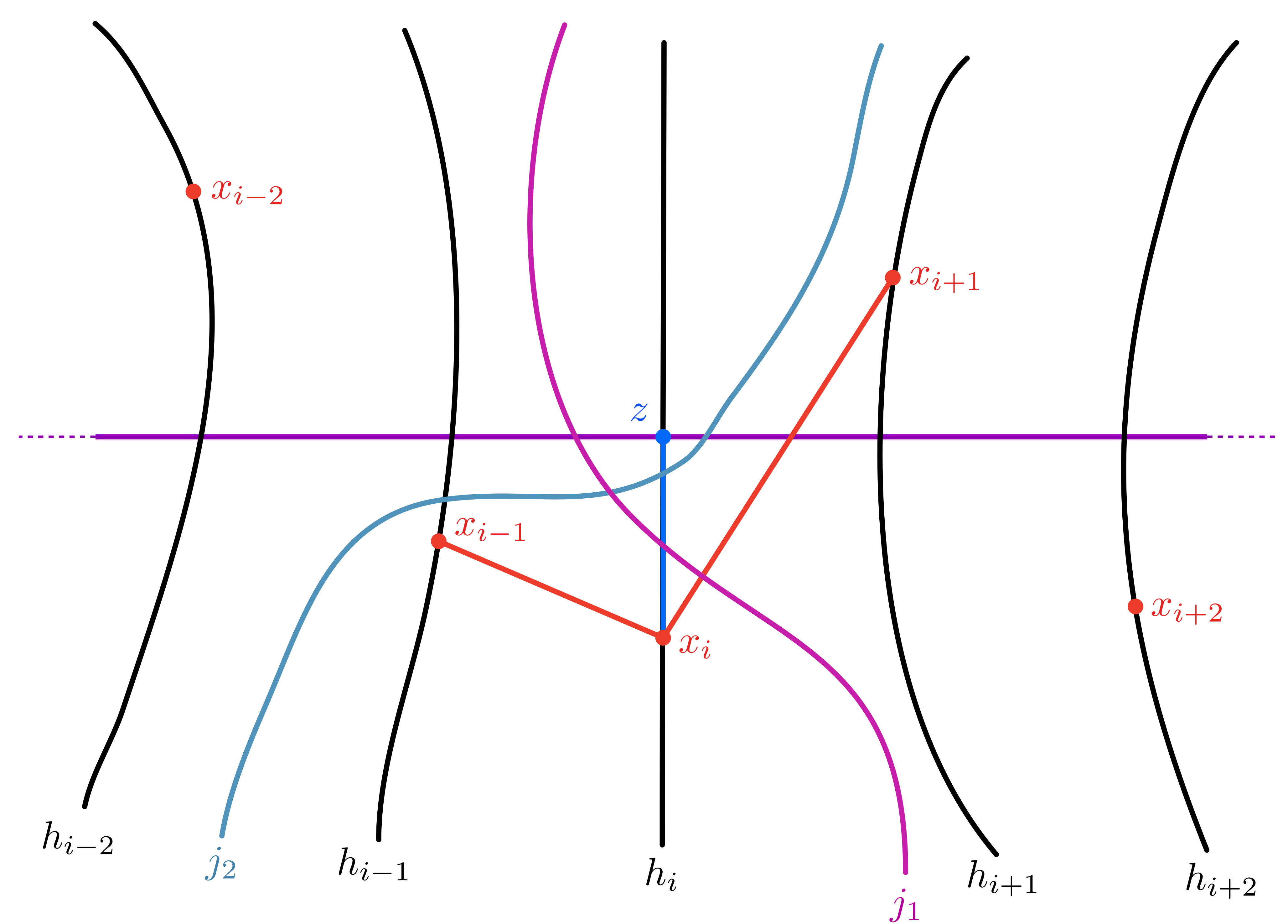}
    \caption{A picture for Lemma \ref{lem: key to contraction}: If a hyperplane separates $z$ from $x_i$ must either cross $h_{i-1},h_i$ or $h_{i}, h_{i+1}$ like $j_2$, or separate $x_{i-1},x_i$ or $x_i,x_{i+1}$ like $j_1$.  In each case, the number of such hyperplanes is $\kappa$-bounded.}
    \label{fig:Lemma_3.16_vec}
\end{figure}

\end{proof}

We remark the following.

\begin{remark} \label{rmk: median paths}

Observe that if $h$ is a $\lambda$-median path in a CAT(0) cube complex $X$ and $h$ intersects a hyperplane $k$ at points $h(t_1),h(t_2)$, then for any $h(t)$ between $h(t_1), h(t_2),$ we have $$h(t) \underset{\lambda}{\asymp} m_X(h(t_1),h(t), h(t_2)) \underset{\lambda}{\in}k.$$ In words, the above equation states that if a $\lambda$-median path $h$ crosses a hyperplane $k$ twice, then the distance it can travel in a given half space of $k$, say $k^+,$ is bounded above by a constant depending only on $\lambda.$ This can be thought of as a coarsening of the property that \emph{no geodesic can cross a hyperplane twice}. In light of this, the conclusion of Lemma \ref{lem: key to contraction} still holds if the geodesic $b$ is replaced by a continuous $\lambda$-median path $h$. That is, given $\{(h_i,x_i)\}, x,y$ as in Lemma \ref{lem: key to contraction}, if $h$ is a continuous $\lambda$-median path connecting $x,y$, then $h$ crosses $h_i$ at a point $z$ with $d(z,x_i) \leq c' \kappa (x_i)$ where $c'$ depends only on the excursion constant $c$ and on $\lambda.$
\end{remark}

The next lemma says that a $\kappa$-excursion chain of hyperplanes can be represented by a geodesic.

\begin{lemma}\label{lem:excursion implies a ray} Let $\{h_i\}_{i \in \mathbb{N}}$ be a chain of $\kappa$-excursion hyperplanes with excursion points $\{x_i\}_{i \in \mathbb{N}}$ and excursion constant $c.$ There exists a $\kappa$-Morse, CAT(0) geodesic ray $b$ such that $x_i \in \calN_\kappa(b,4c).$
\end{lemma}

\begin{proof}
The proof is a straight-forward application of Arzel\`a-Ascoli and Lemma \ref{lem: key to contraction}. In particular, if we let $b_i=[x_1,x_i]$ be the unique CAT(0) geodesic segment connecting $x_1,x_i$, then up to passing to a subsequence $b_i$ converges to a geodesic ray $b.$ By Lemma \ref{lem: key to contraction}, the geodesic ray $b$ must cross each $h_i$ at a point $z_i$ with $d(z_i,x_i) \leq 4c \kappa(x_i).$ This implies that $d(z_i, z_{i+1}) \leq 4c\kappa(x_i)+c \kappa(x_i)+4c \kappa(x_{i+1}) \leq 9c \kappa(x_{i+1}).$ Since $d(x_{i+1}, z_{i+1}) \leq 4c \kappa (x_{i+1})$, Lemma \ref{lem: relating sublinearness} gives us that $d(z_i, z_{i+1}) \leq d \kappa (z_{i+1})$, for a constant $d$, depending only on $c$ and $\kappa.$ Thus, the ray $b$ must be $\kappa$-Morse by \ref{thm: characterize_geodesic_rays_in cube_complexes}.
\end{proof}

We are now ready to characterize $\kappa$-Morseness for quasi-geodesics in terms of crossing a chain of $\kappa$-excursion hyperplanes.

\begin{theorem} \label{thm:hyperplane characterization for quasi-geodesic rays}
A quasi-geodesic ray $\alpha$ is $\kappa$-Morse if and only if there exists an infinite collection of $\kappa$-excursion hyperplanes such that the $\kappa$-excursion points $x_i \in \alpha.$
\end{theorem}

\begin{proof}

Suppose that $\alpha$ and $\{x_i\}_{i \in \mathbb{N}}$ are as in the assumption of the backwards direction.  Lemma \ref{lem:excursion implies a ray} provides a $\kappa$-Morse geodesic ray $b$ such that $x_i \in \calN_\kappa(b).$ As each $x_i \in \alpha,$ and $b$ is $\kappa$-Morse, we get that $\alpha \subseteq \calN_\kappa(b).$ Since $b$ is $\kappa$-Morse, $\alpha$ is also $\kappa$-Morse.

For the forwards direction, assume that $\alpha$ is $\kappa$-Morse with gauge $\mm_\alpha$, and $H=\hull(\alpha)$.  Then $\alpha$ is $\kappa$-Morse in the CAT(0) cube complex $H.$ Hence, by \cite[Proposition 3.10]{QRT19}, there exists a $\kappa$-Morse CAT(0) geodesic ray $b$ in $H$ representing $\alpha.$ In particular, we have $\alpha \subset \calN_\kappa(b,\nn),$ where $\nn$ depends only on the gauge $\mm_\alpha$.  Hence, using Theorem \ref{thm: characterize_geodesic_rays_in cube_complexes}, $b$ crosses an infinite sequence of $\kappa$-excursion hyperplanes $\{h_i\}_{i \in \mathbb{N}}$ with excursion points $y_i \in b$ and excursion constant $c$ depending only on $X$, $\kappa$ and $\mm_\alpha.$

Since $b \in H,$ every hyperplane crossing $b$ must also cross $\alpha$. Let $x_i \in \alpha \cap h_i$. Using item (3) of \cite[ Corollary 4.16]{Murray-Qing-Zalloum}, if $x_i'$ is the projection of $x_i$ to $b,$ then $d(y_i,x_i') \leq D \kappa(y_i),$ where $D$ depends only on $X,$ $\kappa$ and $\mm_\alpha.$ Hence, for each $i$ we have

\begin{align*}
    d(x_i, y_i) &\leq d(x_i, x_i')+d(x_i', y_i)\\
    &\leq \nn \kappa(x_i)+D \kappa(y_i)\\
    &\leq (\nn+D)\kappa(\max \{\|x_i\|, \|y_i\|\}).
\end{align*}

Lemma \ref{lem: relating sublinearness} gives us a constant $D_1 \geq 1$, depending only on $\kappa,$ $\mm_\alpha$, and $X$ (and not on $\max \{\|x_i\|, \|y_i\|\}$) such that $\kappa(y_i) \leq D_1 \kappa(x_i)$ for each $i.$ Hence

\begin{align*}
d(x_i, x_{i+1})&\leq d(x_i, x_i')+d(x_i', y_i)+d(y_i,y_{i+1})+d(y_{i+1}, x_{i+1}')+d(x_{i+1}',x_{i+1})\\
&\leq \nn \kappa (x_i)+D\kappa(y_i)+c \kappa(y_i)+D \kappa(y_{i+1})+ \nn \kappa(x_{i+1})\\
&\leq (2n+2D+c)D_1 \kappa(x_{i+1}).
\end{align*}

Applying Lemma \ref{lem: relating sublinearness} again gives us a constant $D_2$ such that 
$d(x_i,x_{i+1}) \leq D_2 \kappa (x_i)$ which concludes the proof of the forward direction.

\end{proof}

\begin{remark} \label{rmk: boundinghyperplanes} Using the proof of Theorem \ref{thm:hyperplane characterization for quasi-geodesic rays}, we observe the following:

\begin{enumerate}

\item In the backwards direction of the proof of Theorem \ref{thm:hyperplane characterization for quasi-geodesic rays}, the Morse gauge of $\alpha$, denoted $\mm_\alpha$, depends only on $\kappa$, $X$, the excursion constant $c,$ and $d(\alpha(0),h_1).$ Conversely, for a fixed $\kappa$, the Morse gauge $\mm_\alpha$ determines the excursion constant $c$ and $d(\alpha(0),h_1
).$

    \item Suppose that $\alpha$ is a $\kappa$-Morse quasi-geodesic ray, the forward direction of the previous theorem provides a sequence of $\kappa$-excursion hyperplanes $\{h_i\}_{i \in \mathbb{N}}$ and excursion points $x_i \in  h_i \cap \alpha.$ The proof of the forward direction shows that for any points $y_i \in h_i \cap \alpha,$ we must have $d(y_i,y_{i+1}) \leq D' \kappa(y_i),$ for a constant $D'$ depending only on $\mm_\alpha$, $\kappa$ and $X$. In particular, if $t_i$ is the first time $\alpha$ meets $h_i$ and $t_i'$ is the last time $\alpha$ meets $h_{i+1},$ then $d(\alpha(t_i), \alpha(t_i'))$ is bounded above by $D' \kappa(y_i).$ Consequently, the number of hyperplanes that can intersect the subpath $[\alpha(t_i), \alpha(t_i')]_\alpha$ is bounded above by $D'' \kappa(y_i),$ where $D''$ depends only on $D'$ and $(\qq,\sQ),$ where $\qq,\sQ$ are the quasi-geodesic constants of $\alpha.$

\end{enumerate}
\end{remark}

\subsection{Contraction of hulls}

The main statement we prove in this subsection is Proposition \ref{prop:contracting hulls, CAT(0)}. It states that if $\alpha$, $\beta$ are two $\kappa$-Morse quasi-geodesic rays, then the set $H= \hull(\alpha \cup \beta)$ is $\kappa$-strongly contracting with respect to the combinatorial nearest point projection $P_H: X \rightarrow H.$

The proof proceeds as follows: Each ray determines a family of $\kappa$-well-separated hyperplanes, which can be used to partition the hull $H$ into a family of subsets whose diameters are bounded by $\kappa$.  Any point in $X$ lives in the intersection of the half-spaces associated to two adjacent hyperplanes in the family, and this determines the subset of $H$ to which it projects; see Figure \ref{fig:thin_corridor,_2}.  When considering two points $x,y$ outside of the hull, one uses $\kappa$-well-separation to bound the number of hyperplanes that can separate any geodesic between $x,y$ from the hull.  This says that such a geodesic passes $\kappa$-close to the hull, proving that the hull is $\kappa$-contracting.

The following lemma is the main technical statement supporting the proof:

\begin{lemma} \label{lem:2-thin-corridors}
Let $\alpha, \beta$ be two $\kappa$-Morse quasi-geodesic rays starting at $\go$ and let $\{(h_i,x_i)\}_{i \in \mathbb{N}}$ $\{(k_i,y_i)\}_{i \in \mathbb{N}}$ be the two sequences of $\kappa$-excursion hyperplanes corresponding to $\alpha, \beta.$ If $\alpha$ and $\beta$ do not $\kappa$-fellow travel each other, and $H=\hull( \alpha \cup \beta),$ then there exist constants $c, L$ and some $n_0$ such that for any fixed $i > n_0$, we have:

\begin{enumerate}
  \item The hyperplanes $k_i,h_{i}$ are $L$-well-separated.
  
  \item  The hyperplanes $\cdots k_{i-2}, k_{i-1}, k_i, h_i,h_{i+1}, h_{i+2}\cdots $ form a chain.

 \item $\diam (k_i^- \cap h_{i}^- \cap H)<\infty,$ where $k_i^-, h_{i}^-$ are the half spaces containing $\go.$
 
   \item For any $x,y \in H,$ if $x,y$ are between $h_i,h_{i+1},$ then $d(x,y) \leq c \kappa(x_i).$
   
\item For any $x,y \in H,$ if $x,y$ are between $k_i,k_{i+1},$ then $d(x,y) \leq c \kappa(y_i).$

\end{enumerate}
\end{lemma}

\begin{proof}

Since $\alpha, \beta$ do not $\kappa$-fellow travel, Lemma \ref{lem: infinity many hyperplanes implies fellow travelling} gives us that only finitely many hyperplanes can intersect both $\alpha$ and $\beta.$ Let $m$ be large enough to ensure that $h_i \cap \beta =\emptyset$ and $k_i \cap \alpha =\emptyset$ for all $i \geq m.$ Observe that as $h_{m},h_{m+1}$ are disjoint and neither of them crosses $\beta$, if $h_{m+1}$ crosses a hyperplane $k_i$, then so does $h_m.$ Since $h_{m},h_{m+1}$ are well-separated, only finitely many hyperplanes, say $K$, from the collection $\{k_i\}_{i \in \mathbb{N}}$ can be crossed by $h_{m+1}.$ Therefore, since the $h_i, k_i$ form chains, the hyperplanes $h_{m+2}, k_{m+K}$ are $L$-well-separated for some $L$.  Choosing $n_0= \max\{m+2,m+K\}$ concludes the proof of parts (1) and (2).

\begin{figure}
    \centering
    \includegraphics[width=.5\textwidth]{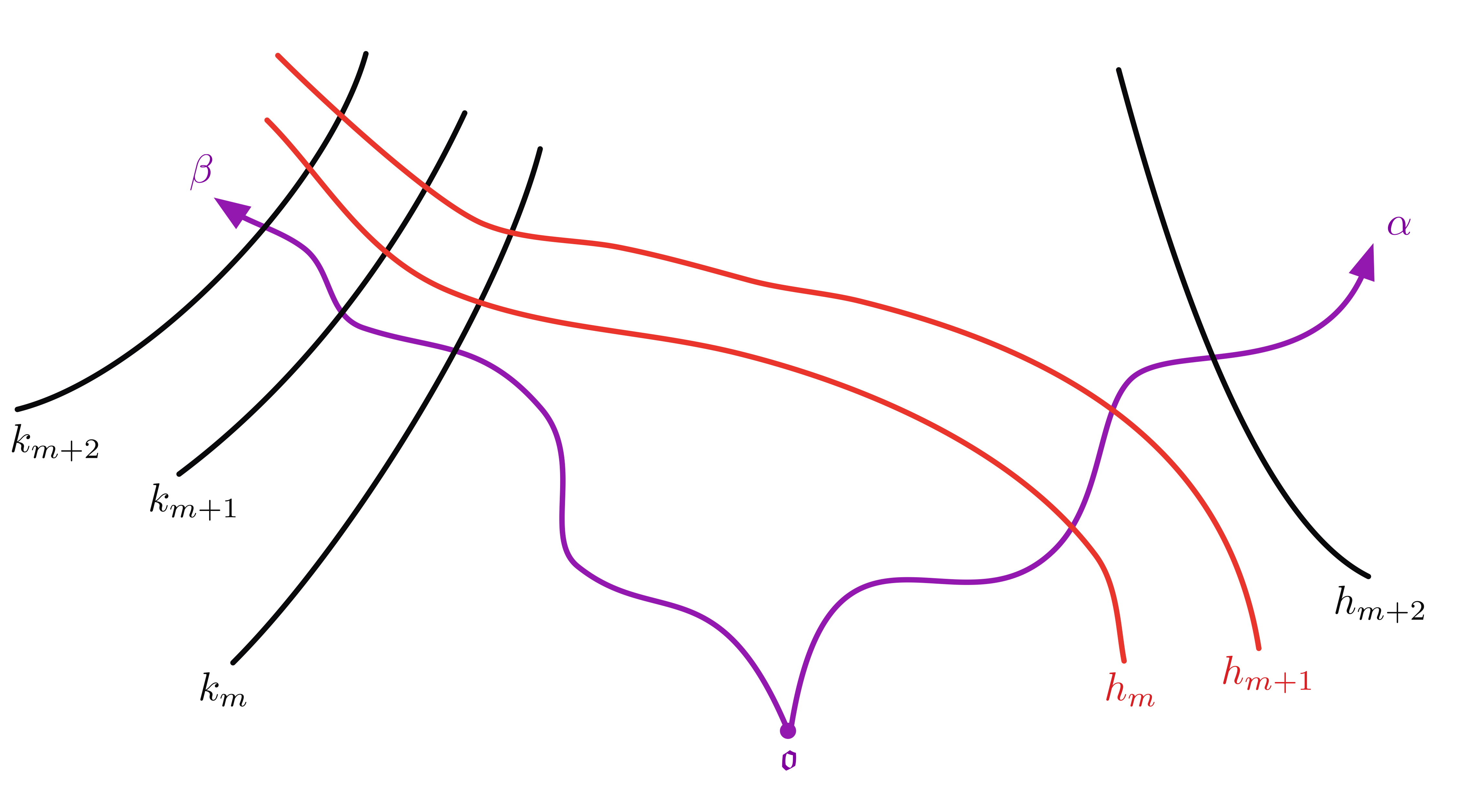}
    \caption{Lemma \ref{lem:2-thin-corridors}, parts (1) and (2): If $h_{m+1}$ crosses some $k_i$, then $h_m$ must as well, since both avoid $\beta$.}
    \label{fig:thin_corridor,_1}
\end{figure}

For part (3), fix some $ i \geq n_0$ and let $p$ be the last point on $\alpha$ which intersects $h_{i+1}$.  Similarly, we define $p'$ to be the last point on $\beta$ which intersects $k_{i+1}$ and let $A$ be the number of hyperplanes intersecting the union of $\alpha \cup \beta$ between $p'$ and $p$. Since $h_i,h_{i+1}$ are well-separated, at most $L_1$ hyperplanes, with no facing triples, can meet them. Similarly, at most $L_2$ hyperplanes, with no facing triples, can meet $k_i,k_{i+1}$.

We will now bound $\diam (k_i^- \cap h_i^- \cap H)$ in terms of $A,L_1$ and $L_2$.  Let $x,y \in k_i^- \cap h_i^- \cap H,$ and let $\mathcal{H}$ be the collection of  separating $x,y.$ Such a collection contains no facing triple as each hyperplane in $\calH$ separates $x,y$. Since every hyperplane in $\calH$ intersects $H$, it must intersect $\alpha \cup \beta$. There are exactly three possibilities for where a hyperplane $h \in \mathcal{H}$ crosses $\alpha \cup \beta$:
\begin{enumerate}
    \item In $h_{i+1}^+$, and hence crosses both $h_i, h_{i+1}$,
    \item In $k_{i+1}^+$, and hence crosses both $k_i, k_{i+1}$, or
    \item In $h_{i+1}^- \cap k_{i+1}^- \cap H$, of which there are $A$ such hyperplanes.
\end{enumerate}

Since $h_{i},h_{i+1}$ and $k_{i},k_{i+1}$ are well-separated with constants $L_1,L_2$, we have the following:
\begin{align*}
    |\mathcal{H}|&\leq |\mathcal{H} \cap h_{i+1}^+|+|\mathcal{H} \cap k_{i+1}^+|+|\mathcal{H} \cap h_{i+1}^- \cap k_{i+1}^-|\\
    &\leq L_1+L_2+A.
\end{align*}

Hence, $\diam(k_i^- \cap h_i^- \cap H) \leq L_1+L_2+A.$ 

\begin{figure}
    \centering
    \includegraphics[width=.5\textwidth]{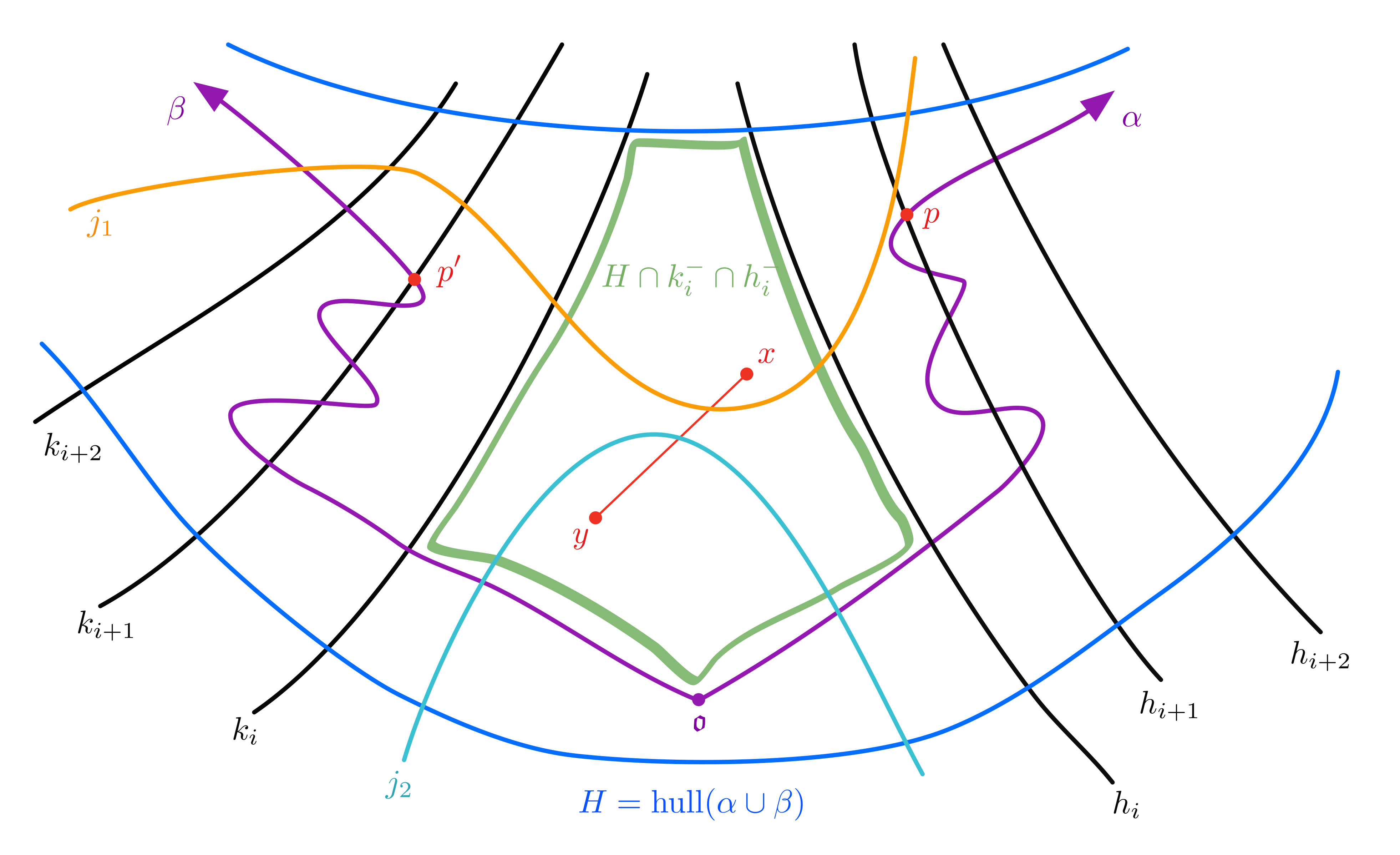}
    \caption{Lemma \ref{lem:2-thin-corridors}, part (3): Hyperplanes which separate $x$ and $y$ have to intersect $\alpha \cup \beta$, and this forces them either between $p,p'$ or to cross $k_{i}, k_{i+1}$ or $h_{i}, h_{i+1}$.}
    \label{fig:thin_corridor,_2}
\end{figure}

Now we prove part (4). Let $n_0$ be as above and consider the chain

$$\cdots k_{n_0+2},k_{n_0+1},h_{n_0+1},h_{n_0+2},\cdots.$$ We denote this chain by $\calJ.$

Let $i >n_0+1$ be fixed and suppose $x,y$ are between $h_i,h_{i+1}$, as in the part (4) of the statement. The combinatorial distance between $x,y$ is the number of hyperplanes separating them. Denote by $\mathcal H(x,y)$ the collection of such hyperplanes, and notice that it contains no facing triple as each hyperplane in $\calH$ separates $x,y$. Using the definition of $\hull (\alpha \cup \beta)$, every hyperplane in $\mathcal H(x,y)$ must meet either $\alpha$ or $\beta$. Since $\beta \subset h_j^-$ for each $j > n_0$ and since $\calJ$ is a chain, every hyperplane $h \in \mathcal H (x,y)$ must satisfy one of the following:
\begin{enumerate}
    \item $h$ meets $\alpha$ between $h_{i-1}, h_{i+2}$,
    \item $h$ meets $(\alpha \cup \beta) \cap h_{i-1}^-$ and thus crosses both $h_{i-1},h_i$, or
    \item $h$ meets $\alpha \cap h_{i+2}^+$ and thus crosses both $h_{i+1}, h_{i+2}.$
\end{enumerate}

Let $\calJ_1, \calJ_2, \calJ_3$ denote the hyperplanes $h \in \calH(x,y)$ satisfying conditions (1),(2), and (3), respectively. In particular, we have $\calJ = \calJ_1 \cup \calJ_2 \cup  \calJ_3.$ Using Theorem \ref{thm:hyperplane characterization for quasi-geodesic rays} and part (2) of Remark \ref{rmk: boundinghyperplanes}, $|\calJ_1|$ is bounded above by $D_1 \kappa(x_{i+1}),$ for a constant $D_1$ depending only on $\kappa$, $\mm_\alpha$, $X$ and the quasi-geodesic constants of $\alpha.$ Further, using the well-separation of $h_i,h_{i-1},$ and $h_{i+1}, h_{i+2},$ we have $|\calJ_2| \leq D_2 \kappa(x_{i-1})$ and $|\calJ_3| \leq D_2 \kappa(x_{i+1})$. Hence, we get the following:

    \begin{align*}
d(x,y)&=|\mathcal{H}(x,y)|\\
&\leq |\calJ_1|+|\calJ_2|+|\calJ_3|\\
&\leq D_1 \kappa(x_{i+1})+D_2 \kappa(x_{i-1})+D_2 \kappa(x_{i+1})\\
&\leq \text{max}(D_1,D_2) \kappa(x_{i+1}).
\end{align*} 
Using Lemma \ref{lem: relating sublinearness}, we get a constant $D_4$, depending only on $\kappa,\mm_\alpha,$ $X$ and the quasi-geodesic constants of $\alpha$ such that $d(x,y) \leq D_4 \kappa(x_{i}).$ This concludes the proof of (4). The proof of (5) is identical.

\begin{figure}
    \centering
    \includegraphics[width=.5\textwidth]{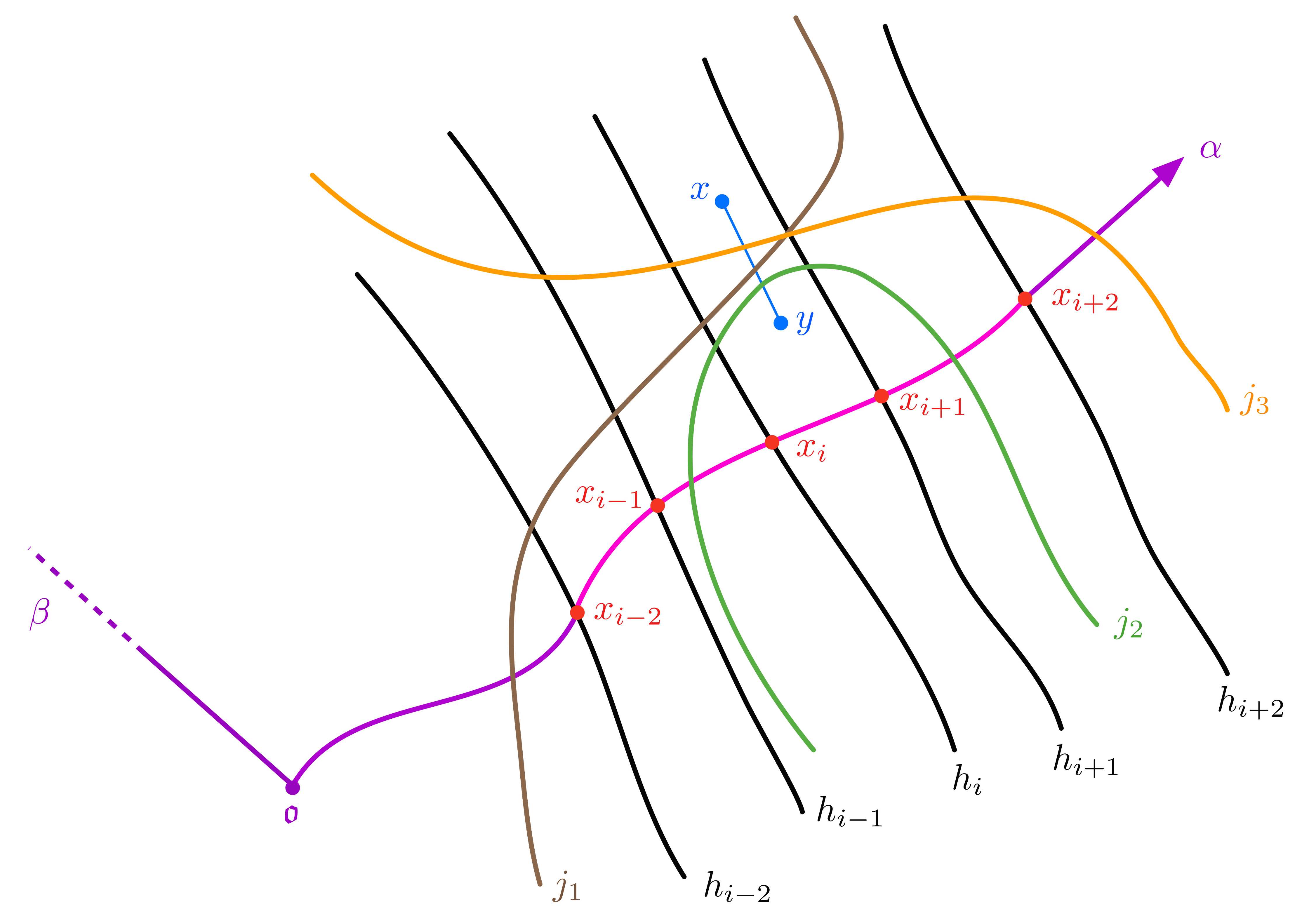}
    \caption{Lemma \ref{lem:2-thin-corridors}, part (4): Hyperplanes separating $x,y$ either cross the $\kappa$-bounded piece of $\alpha$ between $h_{i-1}$ and $h_{i+2}$ like $j_2$, cross $h_{i-2},h_{i-1}$ like $j_1$, or $h_{i+1},h_{i+2}$ like $j_3$.}
    \label{fig:thin_corridor,_3}
\end{figure}

\end{proof}

We remark that an identical argument to the one given in item (4) of Lemma \ref{lem:2-thin-corridors} above shows the following.
\begin{lemma} \label{lem:1-thin-corridor}
Let $\alpha$ be a $\kappa$-Morse quasi-geodesic ray, $H=\hull(\alpha)$ and let $\{(h_i,x_i)\}_{i \in \mathbb{N}}$ be the sequence of excursion hyperplanes crossed by $\alpha$. There exists a constant $c$ such that for any $i \in \mathbb{N},$ if $x,y \in H$ are between $h_i,h_{i+1}$, then
$$d(x,y) \leq c \kappa(x_i).$$ In particular, $H \subset \calN_\kappa(\alpha, c).$

\end{lemma}

We now are ready to prove the main proposition of this subsection.

\begin{proposition} \label{prop:contracting hulls, CAT(0)}

Let $X$ be a CAT(0) cube complex and let $\alpha, \beta$ be two $\kappa$-Morse quasi-geodesic rays which begin at $\alpha(0)=\beta(0)= \go.$ There exists a constant $C_1$ such that for any $x,y \in X$ if $d(x,y) \leq d(x,P_H(x)),$ then we have $d(P_H(x), P_H(y)) \leq C_1 \kappa(x)$, where $P_H:X \rightarrow H$ is the combinatorial nearest point projection. In particular, the set $H$ is $\kappa$-strongly-contracting.
\end{proposition}

\begin{proof} We will only prove the statement in the case where $\alpha$ and $\beta$ do not $\kappa$-fellow travel each other. The proof for when they do is similar and we leave it as an exercise for the reader.

Let $\alpha, \beta$ be two $\kappa$-Morse quasi-geodesic rays starting at $\go$, $\{(h_i,x_i)\}_{i \in \mathbb{N}}$ $\{(k_i,y_i)\}_{i \in \mathbb{N}}$ be the two sequences of $\kappa$-excursion hyperplanes corresponding to $\alpha, \beta$ and let $D$ be the maximum of the two excursion constants. If $n_0$ is the constant provided by Lemma \ref{lem:2-thin-corridors}, then we get a bi-infinite chain $\calH$

$$\cdots k_{n_0+2},k_{n_0+1},h_{n_0+1},h_{n_0+2},\cdots.$$

We claim that the combinatorial nearest point projection of any point $x$ that lives between any two consecutive hyperplanes of the above chain, say $h_{i+1},h_{i+2}$, must project to a point $P_H(x)$ that satisfies $d(P_H(x),x_i) \leq c \kappa(x_i),$ where the constant $c$ is uniform over $i \geq n_0$. To see this, first observe that the claim holds if $P_H(x)$ lies in $h_{i}^+ \cap h_{i+3}^-$ by Lemma \ref{lem: relating sublinearness}. Otherwise, let $b$ be a combinatorial geodesic connecting $x$ to $P_H(x)$, and assume, without loss of generality, that $P_H(x)$ lies in $h^-_j$ for $j < i-1$.  Then $b$ crosses $h_{i-1}$, and so Lemma \ref{lem: key to contraction} says that the point $z_{i-1}$ at which $b$ crosses $h_{i-1}$ must satisfy $d(z_{i-1},x_{i})<c \kappa(x_i)$.  Since $x_i \in H$, this says that $P_H(x) \in h^+_{i-2}$.  Lemma \ref{lem: relating sublinearness} then says that $P_H(x)$ satisfies $d(P_H(x), p) \leq c_0 \kappa(x_i)$ for some point $p \in H$ between $h_{i-1},h_i$ and some constant $c_0$ that is uniform over $i$. However, item (4) of Lemma \ref{lem:2-thin-corridors} along with Lemma \ref{lem: relating sublinearness} gives us a uniform constant $c$ such that $d(P_H(x),x_i) \leq c \kappa(x_i)$ concluding the proof of the claim.

To see that $H$ is $\kappa$-strongly contracting, we let $x,y \in X$ with $d(x,y)< d(x,P_H(x))$. The point $x$ must lie between two consecutive hyperplanes in $\calH$.  Since the rest of the proof only uses the fact that we have a chain of excursion hyperplanes, we can assume that $x$ lies between $h_i,h_{i+1}$.  In particular, the argument when $x$ lies between $k_i,k_{i+1}$ or $k_{n_0+1},h_{n_0+1}$ are identical.

If the point $y$ lies between $h_{i-2}, h_{i+2}$, then, we are done using the previous claim and Lemma \ref{lem: relating sublinearness}. If not, then $y \in h_{i-2}^-$ or $y \in h_{i+2}^+$, without loss of generality, assume $y \in h_{i-2}^-$, which is the half space containing $\go.$ If $b$ is a combinatorial geodesic connecting $x,y$, then $b$ must meet a point $z \in h_{i-1}$. Let $b_1, b_2$ be the subsegments of $b$ connecting $x,z$ and $z,y$ respectively. Using Lemma \ref{lem: key to contraction}, we have $d(z,x_{i-1}) \leq c'\kappa(x_{i-1})$ where $c'$ is uniform over $i.$ Hence, we get the following:
\begin{align*}
    d(x,z) + d(z,y)&=d(x,y)\\
    &< d(x,P_H(x))\\
    &\leq d(x,z)+d(z,x_{i-1})\\
    &\leq d(x,z)+c'\kappa(x_{i-1}).
\end{align*}

\begin{figure}
    \centering
    \includegraphics[width=.5\textwidth]{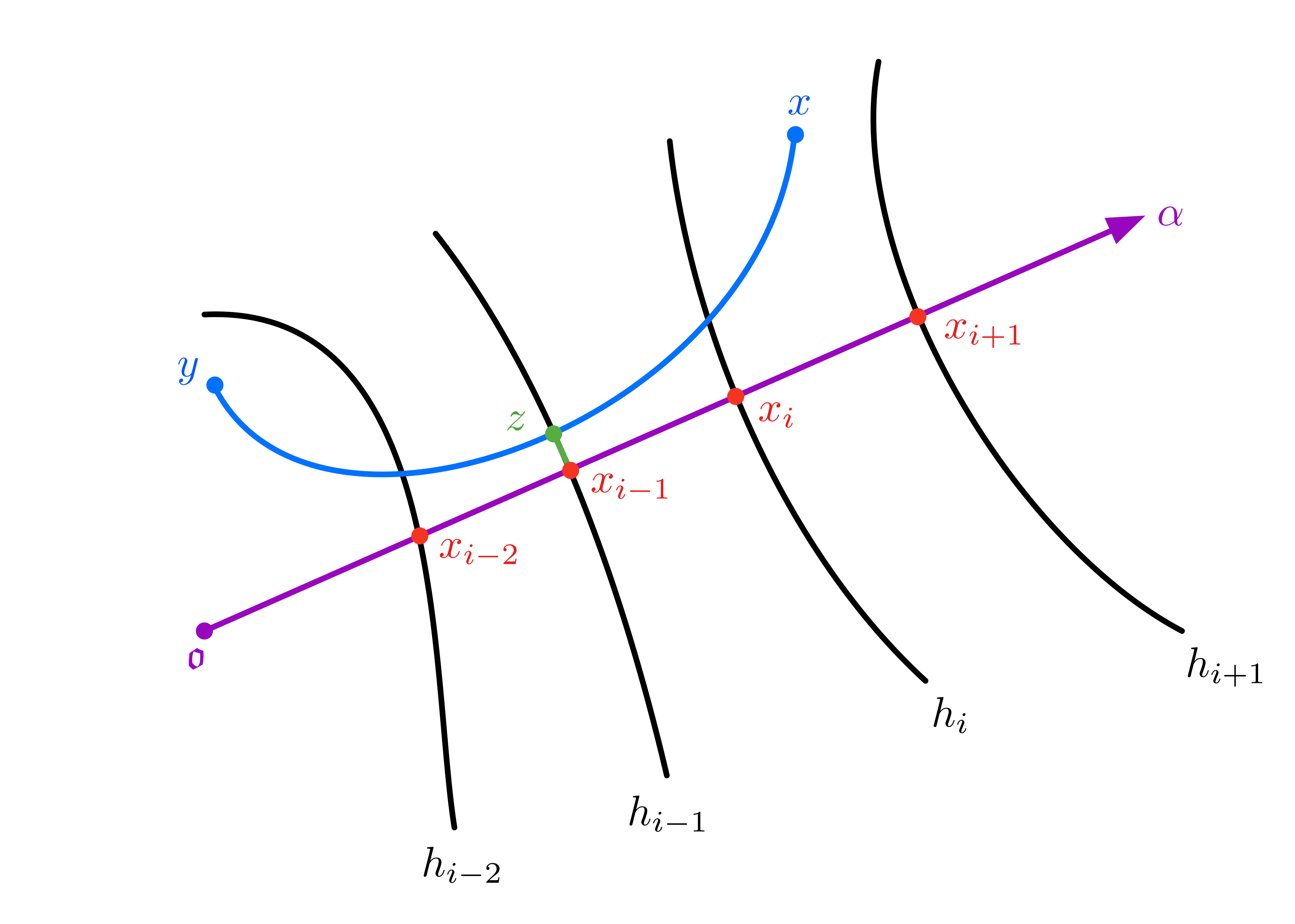}
    \caption{Proposition \ref{prop:contracting hulls, CAT(0)}: If $x,y$ are separated by many hyperplanes}
    \label{fig:Prop._3.22}
\end{figure}
Hence, we have $d(z,y) \leq c'\kappa(x_{i-1})$, and so $d(P_H(z), P_H(y)) \leq c'\kappa(x_{i-1}).$ This gives us
\begin{align*}
d(P_H(x), P_H(y))&\leq d(P_H(x), P_H(z))+d(P_H(z), P_H(y))\\
&\leq d(P_H(x), x_i)+d(x_i,x_{i-1})+d(x_{i-1}, P_H(z))+c'\kappa(x_{i-1})\\
&\leq c \kappa(x_i)+D\kappa(x_i)+c \kappa(x_{i-1})+c'\kappa(x_{i-1})\\
&\leq (2c+D+c') \kappa(x_i).
\end{align*}
To conclude the proof, we first observe that since $d(P_H(x), x_i) \leq c \kappa(x_i)$, Lemma \ref{lem: relating sublinearness} provides a constant $C$, depending only on $c$ and $\kappa,$ such that $\kappa(x_i) \leq C \kappa (P_H(x)).$ Hence, using the above, we get 
\begin{align*}
d(P_H(x), P_H(y))&\leq (2c+D+c') \kappa(x_i)\\
&\leq C(2c+D+c') \kappa(P_H(x))\\
&\leq C(2c+D+c') \kappa(x).
\end{align*} The last inequality holds as $P_H$ is distance-decreasing. In other words, $\|x\|=d(\go, x) \geq d(\go, P_H(x))=\|P_H(x)\|$. Since $\kappa$ is non-decreasing, we get $\kappa(P_H(x)) \leq \kappa(x).$ This concludes the proof.

\end{proof}


\section{Limiting models and gates in LQC spaces} \label{sec:limiting models}

In the rest of the section, we develop limiting models for hulls of median rays in any LQC space, and use them to develop a gate map to these hulls by pushing forward the cubical gate map.  Both of these constructions will be used in subsequent sections.

For the rest of this section, fix an LQC space $X$ and a sublinear function $\kappa$.

\subsection{Limiting models}

We begin by observing that one can always find a median ray representative for any $\kappa$-ray:

\begin{lemma}\label{lem: existence of median rays} There exists $\lambda \geq 1$ depending only on $X$ so that the following holds. For any (weakly) $\kappa$-Morse quasi-geodesic ray $\alpha:[0,\infty) \rightarrow X$, there exists a $\lambda$-median ray $h:[0, \infty) \rightarrow X$ with $[h]=[\alpha]$. 

\end{lemma}

\begin{proof} Let $h_n$ be some $\lambda$-median path connecting $h(0), h(n)$ for $n \in \mathbb{N}.$ By Arzelà–Ascoli, $h_n$ has a convergent subsequence, $h_{n_m} \rightarrow h.$ Since each $h_{n_m}$ is $\lambda$-median, the ray $h$ is a $\lambda$-median ray. Further, since $\alpha$ is weakly $\kappa$-Morse, each $h_{n_m} \subset \calN_{\kappa}(\alpha, \mm_\alpha(\lambda,\lambda))$ which gives us that $h \subset \calN_{\kappa}(\alpha,\mm_\alpha(\lambda, \lambda))$.  On the other hand, if $\alpha$ is $\kappa$-Morse, then it is weakly $\kappa$-Morse, and so $h \sim_{\kappa} \alpha$, making $h$ $\kappa$-Morse as well by Lemma \ref{lem:invariance_of_neighborhood}.
\end{proof}

\begin{remark}\label{rmk:median_ray_constants}
We remark that the constant $\lambda$ in Lemma \ref{lem: existence of median paths} and Lemma \ref{lem: existence of median rays} only depends on $X$. Hence, independently of $\kappa,$ every point in the $\kappa$-Morse boundary contains a representative which is a $\lambda$-median ray. In light of this, unless we mention otherwise, for the rest of the paper, we assume that all median rays under consideration are $\lambda$-median rays for the $\lambda$ provided by Lemma \ref{lem: existence of median rays}. In particular, the constant $\lambda$ of the median ray is determined by $X.$
\end{remark}

The following explains how to build a limiting version of the cubical models when the finite set of points contains points at infinity.  It is one of the main tools of the paper.

\begin{theorem} \label{thm:limiting model}

Let $X$ be a locally quasi-cubical space of dimension $v$ and let $\lambda \geq 1, m,n \in \mathbb{N}$ be given.  There exists a constant $K$, depending only on $\lambda, m,n$ and the LQC parameters of $X$, such that the following holds. For any $x_1,x_2,\dots, x_m \in X$ and any finite set of $\lambda$-median rays $h_1, h_2, \dots, h_n$, the median hull $H$ of the set $\displaystyle \bigcup_{i=1}^n h_i \cup \{x_1,x_2 \cdots x_m\}$ admits a $(v,K)$-cubical approximation.
\end{theorem}

Our limiting models are defined as ultralimits of an appropriate sequence of finite models.  We briefly introduce ultralimits now in the context we will need them.  See \cite{DK_book} for a detailed discussion.


Let $(M_n, d_n)$ be a sequence of metric spaces with a sequence $(m_n)$ of basepoints, and $\omega \subset 2^{\mathbb N}$ a non-principal ultrafilter.  The \emph{ultralimit} $(\mathbf{M}, \mathbf{d}) = \displaystyle \lim_{\omega} (M_n, m_n, d_n)$ is a metric quotient of $\displaystyle \Pi_n M_n$, defined as follows.  For $\displaystyle x = (x_n), y = (y_n) \in \Pi_n M_n$, we set $\displaystyle \mathbf{d}(x,y) = \lim_{\omega} d_n(x_n,y_n)$, and we quotient out $\displaystyle \Pi_n M_n$ by identifying tuples $x,y$ for which $\mathbf{d}(x,y) = 0$.

Given a sequence of pointed maps $f_n:(M_n, m_n) \to Y$ to a fixed metric space, there is a $\omega$-limiting map $f:\mathbf{M} \to H$, which preserves coarse metric properties, e.g. $f$ is a (uniform) quasi-isometry when the $f_n$ are uniform quasi-isometries.

\begin{proof}[Proof of Theorem \ref{thm:limiting model}]

  For simplicity of notation, we may assume that the basepoints of the rays $h_i$ are contained in $x_1, \dots, x_m$.  For each $t \in \mathbb N$, set $F_t = \{x_1, \dots, x_m, h_1(0), \dots h_n(0), h_1(t), \dots, h_n(t)\}$, $H_t = \hull(F_t)$. Using local quasi-cubicality of $X$, there exist a constant $K,$ depending only on $m+2n$, and a $K$-median quasi-isometry $f_{t}:Q_t \to H_t$ with a coarse inverse $g_t$. Notice that $H_t \subset H$ for each $t$, and hence there exists a $K$-median quasi-isometric embedding $f_{t}:Q_t \to H$. We note that since the $H_t$ coarsely exhaust $H$, the images $f_t(Q_t)$ also coarsely exhaust $H$. Fix a non-principal ultrafilter $\omega$.  For each $t$, let $y_{1,t}=g_t(x_1)$ be a point in $Q_t$ with $d_{X}(x_1, f_t(y_{1,t})) <K$.

Let $Q =  \lim^{\omega}_t (Q_t, y_{1,t})$ and let $f:Q \to H$ be the $\omega$-limit of the pointed maps $f_t:(Q_t, y_{1,t}) \to (H, x_1)$.  Since each $f_t$ is a $K$-median quasi-isometric embedding whose images coarsely exhaust $H$, it follows from some easy calculations that $f:Q \to H$ is a $(K,K)$-quasi-isometry.  Moreover, it is $K$-quasimedian since medians in the limit are limits of medians, which are preserved by the $f_t$ and hence by $f$.  The fact that the dimension of $Q$ is bounded by $v$ follows from the fact that the dimensions of the $Q_t$ are. This completes the proof.

\end{proof}

The following is a special case of Lemma \ref{lem: commuting with hulls}. We write it here explicitly since we will be making frequent references to it.

\begin{lemma} \label{lem: hull in hull}

Let $h_1,h_2$ be two median rays starting at $\go$ and let $F$ be a finite set of points in $X$.  If $f: Q \rightarrow \hull(h_1 \cup h_2 \cup F)$ is the cubical model, $g=f^{-1}$, and $F' \subseteq F$ then there exists some $K$, depending only on $|F|$, such that

\begin{enumerate}
    \item $d_{Haus}(\hull\{g(h_1) \cup g(h_2) \cup g(F')\}, g(\hull\{h_1 \cup h_2 \cup F'\}) \leq K.$

   \item  $d_{\text{Haus}}(f(\hull\{g(h_1) \cup  g(h_2) \cup g(F')\}),\hull\{h_1 \cup h_2 \cup F'\}) \leq K.$
    \end{enumerate}
    In particular, $Q$ and $\hull\{g(h_1) \cup g(h_2) \cup g(F)\}$ are at Hausdorff distance at most $2K.$

\end{lemma}

\begin{assumption}\label{assumption: respecting hulls} In light of the above lemma, we may assume that if $f: Q\rightarrow \hull\{h_1 \cup h_2 \cup F\}$ is a $(K,K)$-cubical model with a coarse inverse map $g$, then they satisfy
$$f(\hull\{g(h_1) \cup g(h_2) \cup g(F')\}) \subseteq \hull \{h_1 \cup h_2 \cup F'\}.$$
For the rest of the paper, we will make this assumption.

\end{assumption}

\subsection{Gate maps via limiting models}

In this subsection, we explain how to define a gate map in a LQC space $X$ to the median hull of two rays via an appropriate cubical model.  The proof works for the hull of any finite number of rays and points, but we restrict to two of each for simplicity.

The idea of the construction is as follows: We take two rays $h_1, h_2$ based at $\go$ and add the point $x \in X$ that we want to project to $H = \hull(h_1 \cup h_2)$.  In the cubical model $Q'$ for $H' = \hull(h_1 \cup h_2 \cup \{x\})$, there is a combinatorial projection for the image of $x$, which we then pull back to $H$ in $X$.

In Section \ref{sec:LQC characterization}, we will want to use this idea to compare two such projections at once, hence the additional complexity of the following construction.

\begin{construction} \label{construction: gate map LQC} Let $h_1,h_2$ be two (possibly equal) median rays in $X$ starting at $\go$ and let $H=\hull\{h_1 \cup h_2\}$. We will define a map $G_H: X\rightarrow H$ as illustrated in the following steps:

\begin{enumerate}

\item Fix $x \in X$. Our goal is to define $G_H(x).$
    \item Select some $y \in X$ and define $H_y:=\hull\{h_1 \cup h_2 \cup \{x,y\}\}$, $f_y:Q_y \rightarrow H_y$ be the cubical model for $H_y$ and let $g_y:H_y \rightarrow Q_y$ be the coarse inverse for $f_y.$ We denote $x_y=g_y(x).$
    
    \item Let $A_y=\hull\{g_y(h_1) \cup g_y(h_2)\} \subseteq Q_y$. Using Assumption \ref{assumption: respecting hulls}, we have $f_y(A_y) \subseteq H$.

     \item Define $G_H(x)=   f_y \circ P_{A_y} \circ g_y(x)$, where $P_{A_y}:Q_y \rightarrow A_y$ is the combinatorial nearest point projection to $A_y.$
\end{enumerate}

\end{construction}

In order to show that $G_H:X \to H$ is well-defined, we need to prove that it is coarsely independent of the choice of $y$.  We continue with the same notation.

\begin{lemma} \label{lem: gates are well-defined} There exists a constant $C$ such that for any $x,y,z \in X$, we have

$$d_X(f_y \circ P_{A_y} \circ g_y(x), f_z \circ P_{A_z} \circ g_z(x)) \leq C.$$

\end{lemma}

\begin{proof}

Let $B_y=\hull\{g_y(h_1) \cup g_y(h_1)\cup g_y(x)\} \subset Q_y$ and let $B_z=\hull\{g_z(h_1)\cup g_z(h_1)\cup g_z(x)\} \subset Q_z$. Notice that $A_y \subseteq B_y$ and $A_z \subseteq B_z$. Let $x_y, x_z$ denote $g_y(x), g_z(x)$ respectively. We start by proving the following claim.

\begin{claim} If $\phi:B_y \rightarrow B_z$ is a $K$-median quasi-isometry with $d_{Haus}(\phi(A_y), A_z) \leq K$ and $d_{B_z}(\phi(x_y),x_z) \leq K$ then

$$d_{B_z}(\phi \circ P_{A_y}(x_y), P_{A_z}(\phi(x_y)) \leq K' \text{ and } d_{B_z}(\phi \circ P_{A_y}(x_y), P_{A_z}(x_z)) \leq K', $$ where $K'$ depends only on $K.$

\end{claim}

\begin{proof}[Proof of Claim]

 Let $\phi:B_y \rightarrow B_z$ be some $K$-median quasi-isometry as in the statement of the claim. Let $w_y=P_{A_y}(x_y)$, let $w_z=P_{A_z}(x_z).$ Using the assumption on the map $\phi,$ there exists a point $w \in A_y$ with $\phi(w) \underset{K}{\asymp} w_z$. Now, observe that since $w_y$ is the combinatorial nearest point projection of $x_y$ to $A_y,$ we have
   $ w_y =m(x_y, w_y,w)$. Furthermore, using the assumptions on $\phi,$ there exists $p \in A_z$ with $\phi(w_y)\underset{K}{\asymp}p$. Hence,

\begin{align*}
    \phi(w_y)& \underset{K}{\asymp} m(\phi(x_y), \phi(w_y), \phi(w))\\
    &\underset{K}{\asymp} m(x_z,\phi(w_y), w_z)\\
    &\underset{K}{\asymp} m(x_z,p, w_z)\\
    &=w_z.
\end{align*}

Thus, $\phi(w_y)\underset{K}{\asymp} w_z$. Notice that since combinatorial projections are distance-decreasing, we have 

\begin{align*}
    d_{B_z}(w_z,P_{A_z}(\phi(x_y))&=d_{B_z}(P_{A_z}(x_z),P_{A_z}(\phi(x_y))\\
    &\leq d_{B_z}(x_z, \phi(x_y))\\
    &\leq K.
\end{align*}
Hence, we get that $\phi(w_y)  \underset{K}{\asymp} w_z =P_{A_z}(x_z) \underset{K}{\asymp} P_{A_z}(\phi(x_y))$. Therefore, we have $d_{B_z}(\phi(w_y), P_{A_z}(\phi(x_y)) ) \leq C'$ for some $C'$ depending only on $K$, which proves the claim.

\end{proof}

\vspace{4mm}

Now, the map $\phi=g_z\circ f_y:B_y \rightarrow B_z$ satisfies the conditions of the previous claim with constant $C=C(5).$ Hence, we have $\phi(P_{A_y}(x_y)) \underset{C}{\asymp} P_{A_z}(x_z)$ or $(g_z \circ f_y)(P_{A_y}(x_y)) \underset{C}{\asymp} P_{A_z}(x_z)$. Since $f_z$ is the coarse inverse of $g_z,$ applying $f_z$ to the last equation gives us $f_y(P_{A_y}(x_y)) \underset{C}{\asymp} f_z(P_{A_z}(x_z))$ which proves the statement.

\end{proof}

We summarize the properties of the gate map $G_H$ in Construction \ref{construction: gate map LQC} and Lemma \ref{lem: gates are well-defined} in the following theorem.

\begin{theorem}\label{thm:gates in LQC} Let $h_1,h_2$ be two median rays in a locally quasi-cubical space $X$ and let $H=\hull(h_1 \cup h_2)$. There exists a constant $C$, depending only on $X$ and a map $G_H:X \rightarrow H$ such that the following holds:

\begin{enumerate}
    \item $G_H$ is $C$-coarsely Lipschitz. 
    \item For $x \in X,$ if $x'$ is a nearest point projection of $x$ to $H,$ then $$d_X(x,x') \underset{C}{\asymp}d_X(x,G_H(x)).$$
    
    \item We have $G_H(x)\underset{C}{\in} \underset{p \in H}{\bigcap}[x,p].$ In fact, for any $z \in Y $, if $z \underset{C}{\in} \underset{p \in H}{\bigcap}[x,p],$ then $z\underset{C}{\asymp}G_H(x).$
\end{enumerate}

\end{theorem}

\begin{proof}
Item (1) is obvious by the Construction \ref{construction: gate map LQC} and the fact that combinatorial projections in CAT(0) spaces are distance-decreasing. For item (2), let $x'$ be a nearest point projection $x \in X$ to $H$. Let $y \in X$, and let $Q_y, A_y$ be as in Construction \ref{construction: gate map LQC}, so that  $d_X(x,x')\underset{C}{\asymp}d_{Q_y}(x_y,x_y')$ for $x_y' = g_y(x') \in A_y.$ Hence,

\begin{align*}
d_X(x,x') &\underset{C}{\asymp}d_{Q_y}(x_y,x_y')\\ &\geq d_{Q_y}(x_y, P_{A_y}(x_y))\\
&\underset{C}{\asymp}d_X(x, G_H(x)),
\end{align*}
 where the last equality holds by the definition of the map $G_H.$ For item (3), let $p \in H$ and let $p' \in A_y$ be such that $f_y(p')\underset{C}{\asymp}p$ where $f_y$ is as in Construction \ref{construction: gate map LQC}. Since $P_{A_y}(x_y)$ satisfies $m_{Q_y}(x_y,P_{A_y}(x_y), p')=P_{A_y}(x_y)$ and since $f_y$ in Construction \ref{construction: gate map LQC} is a median map, we have 
 
 \begin{align*}
     G_H(x)&=f_y(P_{A_y}(x_y))\\
     &=f_y(m_{Q_Y}(x_y,P_{A_y}(x_y), p'))\\
     &\underset{C}{\asymp}m_{X}(f_y(x_y),f_y(P_{A_y}(x_y)), f_y(p'))\\
     &\underset{C}{\asymp}m_{X}(x, G_H(x), p).
 \end{align*}
 
 This shows that $G_H(x) \underset{C}{\in} \underset{p \in H}{\bigcap}[x,p]$. Now, for any other $z$ with $z \underset{C}{\in} \underset{p \in H}{\bigcap}[x,p]$, we must have $z \underset{C}{\in} [x,G_H(x)]$ since $G_H(x) \in H$. That is, there exists $z' \in [x, G_H(x)]$ with $z\underset{C}{\asymp}z'$. Since $z' \in [x, G_H(x)]$, Lemma \ref{lem: intervals are convex} gives us that $z'\underset{C}{\asymp} m_X(z', x, G_H(x))$. Hence, 
 
 \begin{align*}
     z&\underset{C}{\asymp}z'\\
     &\underset{C}{\asymp}m_X(z', x, G_H(x))\\
     &\underset{C}{\asymp}m_X(z, x, G_H(x))\\
     &\underset{C}{\asymp}G_H(x).
 \end{align*}
\end{proof}


\section{Characterizing sublinear Morseness in LQC spaces} \label{sec:LQC characterization}

In this section, we show that being $\kappa$-Morse (Definition \ref{def:kappa-Morse}), weakly $\kappa$-Morse (Definition \ref{def:kappa weakly Morse}), and $\kappa$-contracting (Definition \ref{def:kappa-contracting}) are all equivalent for quasi-geodesic rays in an LQC space, while these notions are a priori different in general.  The following is Theorem \ref{thm:LQC k-Morse equivalence} from the introduction:

\begin{theorem} \label{thm: all equivalent LQC} Let $X$ be LQC and let $q \subset X$ be a quasi-geodesic ray. The following statements are all equivalent.

\begin{enumerate}
    \item $q$ is $\kappa$-Morse.
    \item $q$ is $\kappa$-contracting.
    \item $q$ is weakly $\kappa$- Morse.
\end{enumerate}

Moreover, if any of the three above conditions holds and $h$ is a median ray representing $q,$ then there exist a coarsely Lipschitz map $G_H: X \rightarrow H=\hull(h)$, and constants $C_2,C_3$ such that if $x,y \in X$ with $d(x,y) \leq C_2d(x,H),$ then $d(G_H(x), G_H(y)) \leq C_3 \kappa(x).$

\end{theorem}

\begin{proof} The statements $ (2) \implies (1)$, $(2) \implies (3)$ and $(1) \implies (3)$ follow using Theorem \ref{thm:contracting implies Morse}. Hence, it suffices to show that $(3) \implies (2).$ This is Corollary \ref{cor:old_orse is contracting in LQC}. The moreover part is Proposition \ref{prop: hulls are contracting} of this section.

\end{proof}

The proof of Corollary \ref{cor:old_orse is contracting in LQC} relies on the limiting cubical model Theorem \ref{thm:limiting model}. 

We begin by observing the following immediate consequence of Lemma \ref{lem:1-thin-corridor} and Theorem \ref{thm:limiting model}.

\begin{corollary}\label{cor:1-thin-corridor in LQC} 
Let $h$ be a weakly $\kappa$-Morse median ray in $X$ and $H=\hull_X(h)$. Then $H \subset \calN_\kappa(h, c).$
\end{corollary}

The idea for the next proposition is that the median gate map (Construction \ref{construction: gate map LQC}) is defined by passing to an appropriate cubical model.  We note that this argument is the reason for the complexity of the definition in Construction \ref{construction: gate map LQC}.  In the cubical model, the median hull of a $\kappa$-Morse median ray is $\kappa$-contracting (Proposition \ref{prop:contracting hulls, CAT(0)}), and hence so is the gate map.

\begin{figure}
    \centering
    \includegraphics[width=.7\textwidth]{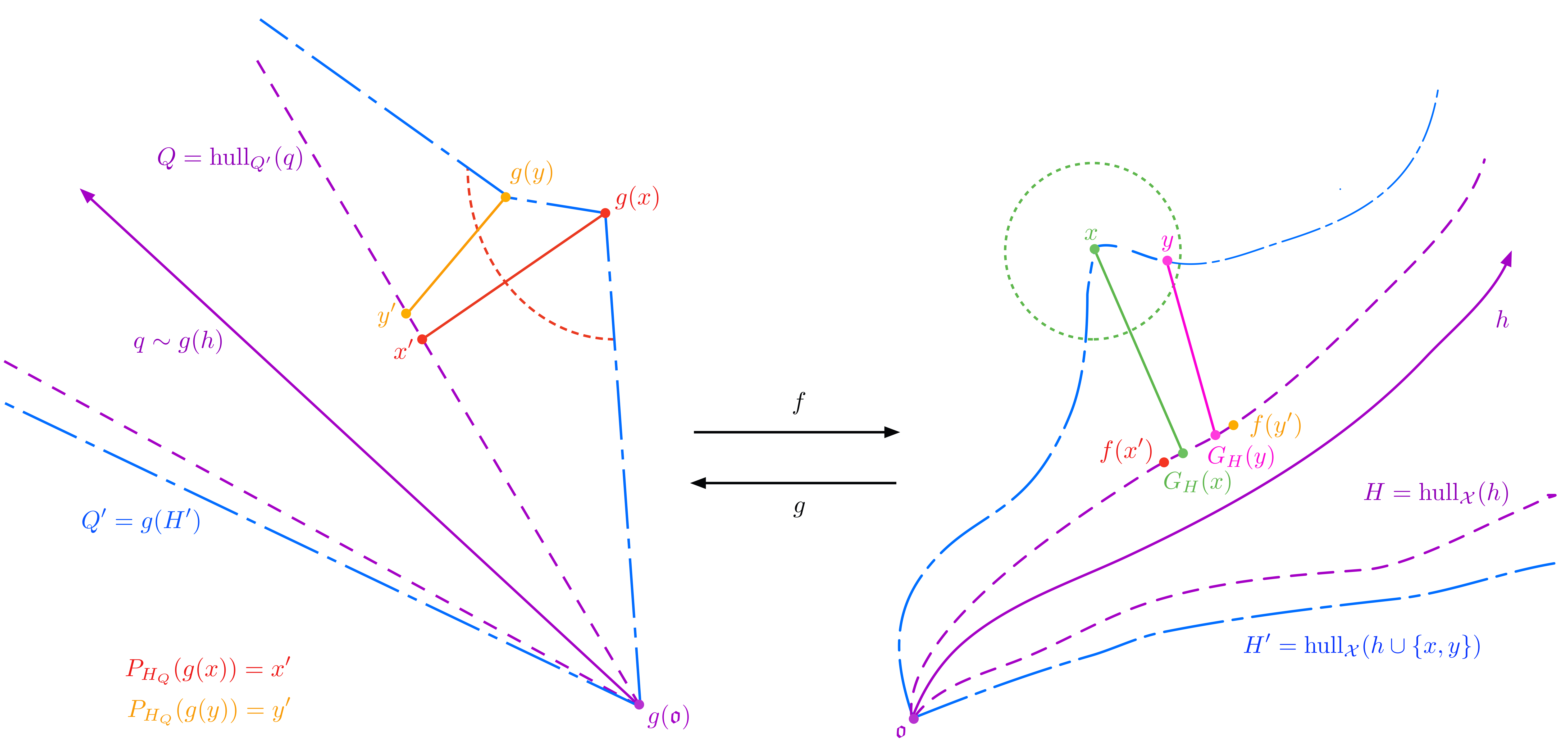}
    \caption{Proof of Proposition \ref{prop: hulls are contracting}: Proving $\kappa$-contraction of the median gate map $G_H$ to the hull $H$ of a $\kappa$-median ray $h$ uses the larger cubical model $Q'$ of $H' = \hull_X(h \cup \{x,y\})$.  The median gate to the cubical hull of $g(h)$ in $Q'$ is $\kappa$-contracting, and the gate map $G_H:X \to H$ is coarsely the image of the cubical gate $f \circ P_{g(H)}\circ g$.}
    \label{fig:Contracting_via_model}
\end{figure}

\begin{proposition} \label{prop: hulls are contracting}
Let $h,h'$ be two median rays starting at $\go$ and let $H=hull(h \cup h').$ If $h,h'$ are weakly $\kappa$-Morse, then $H$ is $\kappa$-contracting with respect to the map $G_H:X \rightarrow H$. In particular, if $h=h'$ is a weakly $\kappa$-Morse median ray, then $H=\hull(h)$ is $\kappa$-contracting.
\end{proposition}

\begin{proof}
Let $G_{H}: X \rightarrow H$ be the map defined in Construction \ref{construction: gate map LQC}, let $B=B_5$ the quasi-isometry constant from the cubical approximation theorem with $|F|=5$ and let $K$ be the constant from Lemma \ref{lem: hull in hull}. Recall that $K$ depends only on the number $5$ and the LQC constants for $X$.

Let $x,y \in X$ with $d_X(x,H) >\max\{(1+B+K)B),1\},$ and $1<d_X(x,y)<Rd_X(x,H)$ for a constant $R=\frac{1}{2B^2(1+B+K)}$.  Let $f:Q \rightarrow \hull(h \cup h' \cup \{x,y\})$ be a $B$-median quasi-isometry where $Q$ is a CAT(0) cube complex and let $g$ be a coarse inverse for $f$. Using Theorem \ref{lem: hull in hull}, the sets $g(H)$ and $\hull_Q(g(h_1) \cup g(h_2))$ are at Hausdorff distance $K$.

Let $H = \hull_X(h \cup h')$ and $H' = \hull_Q(g(h) \cup g(h'))$.  We have $H \subset \hull_X(h \cup h' \cup \{x,y\})$ and hence 
$$d_Q(g(x), \hull_Q\{g(h) \cup  g(h')\}) \geq d_Q((g(x),g(H)))-K \geq \frac{1}{B}d_X(x,H)-B-K \geq  1$$
where the first inequality holds as the Hausdorff distance between the respective sets is at most $K$, the second inequality holds as $g$ is a $(B,B)$-quasi-isometry, and the last inequality holds as $d(x,H) \geq B(1+B+K).$

Let $x'=g(x), y'=g(y)$, $H'=\hull_Q(g(h) \cup  g(h'))$.  Observe that 

\begin{align*}
    d_Q(x',y') \leq Bd_X(x,y)+B &\leq Bd_X(x,y)+Bd_X(x,y)\\
    &=2Bd_X(x,y)\\
    &\leq 2BRd_X(x,H)\\
    &\leq 2BRd_Q(x',H')+B+K)B\\
    &=2B^2R(d_Q(x',H')+B+K))\\
    &\leq 2B^2R(d_Q(x',H')+(B+K)d_Q(x',H'))\\
    &=2B^2R(1+B+K)d_Q(x',H').
\end{align*}

The first inequality holds as $f,g$ are $(B,B)$ quasi-isometries, the second holds as $d_X(x,y) \geq 1$, the fifth hold as $d_Q(x',H') \geq \frac{1}{B}d(x,H)-B-K$ as shown above. The last two inequalities hold as $d_Q(x',H') \geq 1$ as shown above. Thus, we have 

$$d_Q(x',y') \leq 2B^2R(1+B+K)d_Q(x',H') \leq d_Q(x',H'),$$

where the last inequality holds by our choice of $R.$  Since $g(h), g(h')$ are weakly $\kappa$-Morse, Theorem \ref{thm: CAT(0) all equivalent} and Proposition \ref{prop:contracting hulls, CAT(0)} provide a constant $C_1>0$ and the conclusion that the set $H'$ is $\kappa$-contracting, that is $d_Q(P_{H'}(x'), P_{H'}(y')) \leq C_1 \kappa(\|x'\|_Q)$. Using the definition of $G_H:X \rightarrow H,$ we have $d_X(G_H(x),G_H(y)) \leq Bd_Q(P_{H'}(x'), P_{H'}(y')))+B \leq B C_1 \kappa(\|x'\|_Q)+B$. But $\|x'\|_Q \leq B\|x\|_X+B$ and hence

\begin{eqnarray*}
d_X(G_H(x),G_H(y)) &\leq& BC_1 \kappa(\|x'\|_Q)+B\\
&\leq& BC_1 \kappa(B\|x\|_Q+B)+B\\
&\leq& BC_1 \kappa(B\|x\|_X+B\|x\|_X)+B\\
&\leq& 2B^2 C_1  \kappa(\|x\|_X)+B\kappa(\|x\|_X)\\
&=&(2B^2C_1 +B)\kappa(\|x\|_X),
\end{eqnarray*}
completing the proof.
\end{proof}

\begin{remark}\label{rmk:technical_conclusion}
We remark that the proof of Proposition \ref{prop: hulls are contracting} shows the following more general statement. Fix a median ray $h$ in $X$ and let $f_{x,y}:Q_{x,y} \rightarrow \hull_X(h \cup \{x,y\})$ be the cubical model for $\hull(h \cup \{x,y\})$ with a coarse inverse $g_{x,y}$. If the ray $g_{x,y}(h) \subset Q$ is $\kappa$-contracting where the contraction constant is independent of $x,y,$ then $\hull_X(h)$ is $\kappa$-contracting with respect to the map $G_H$ defined in Construction \ref{construction: gate map LQC} where $H=\hull(h \cup \{x,y\})$.
\end{remark}



\begin{corollary}\label{cor:old_orse is contracting in LQC} If $q$ is a weakly $\kappa$-Morse quasi-geodesic ray, then $q$ is $\kappa$-contracting.

\end{corollary}

\begin{proof}
Let $q$ be a weakly $\kappa$-Morse and let $h$ be a median ray representing $q$ which is assured to exist by Lemma \ref{lem: existence of median rays}. Using Lemma \ref{lem:invariance_of_neighborhood}, $h$ is weakly $\kappa$-Morse with a map $\mm_h$ as provided by Definition \ref{def:kappa weakly Morse}. The weakly $\kappa$-Morse condition is clearly invariant under quasi-isometries, hence, $g(h)$ is weakly $\kappa$-Morse. Using Theorem \ref{thm: CAT(0) all equivalent}, the quasi-geodesic ray $g(h)$ is $\kappa$-contracting. Importantly, the contraction constant of $g(h)$ depends only on $\mm_h$ and the quality of the quasi-isometry of $f$, and is independent of $x,y.$ Applying Remark \ref{rmk:technical_conclusion}, we get that $\hull(h)$ is $\kappa$-contracting. Combining Corollary \ref{cor:1-thin-corridor in LQC} and Lemma \ref{lem:invariance_of_neighborhood} gives us that $h$ is $\kappa$-contracting.
\end{proof}

The following is an immediate consequence of Corollary \ref{cor:medians_vs_morseness_CAT(0)} and Theorem \ref{thm:limiting model}.

\begin{corollary}\label{cor:medians_vs_morseness_LQC} Let $h_1,h_2$ be two $\kappa$-Morse median rays in a locally quasi-cubical space with $h_1(0)=h_2(0)=\go$. We have  

$$d_X(\go, m_X(h_1(i),h_2(j),\go)) \to \infty \iff h_1 \sim_{\kappa} h_2.$$

\end{corollary}



\section{Continuous injection into the boundary of the curve graph}

The main theorem of this section is Theorem \ref{thm:HHS inject intro} from the introduction:

\begin{theorem}\label{thm:map} Let $\calX$ be a proper HHS, let $S$ denote the maximal domain of $\calX$ with respect to the ABD construction, and let $G$ act on $\calX$ by HHS automorphisms. For any sublinear function $\kappa$, the projection $\pi_S:\calX \to \calC(S)$ induces a $G$-equivariant continuous injection
$$i_{\kappa}: \partial_\kappa \calX \rightarrow \partial \calC(S).$$

\end{theorem}

\begin{remark}\label{rem:HHS_auto}

Automorphisms of HHSes were studied in \cite{DHS17, DHS_corr}, and we refer the reader to those papers
for the full definition.  The essential point is that if $G$ is a group of HHS automorphisms of $\calX$, then one can arrange that every element of $G$ preserves the set of domains $\mathfrak S$ and the maximal domain $S = g \cdot S$, induces isometries $\calC(Y) \to \calC(g Y)$ which commute with the projections $\pi_Y$ in the expected fashion, etc.  In particular, all of the underlying machinery is preserved by $G$, including the cubulation machine from \cite{HHS_quasi}.  This precise setup cannot be guaranteed for the full group of HHS automorphisms, see \cite[Subsection 2.1]{DHS_corr} for a discussion.

\end{remark}

\subsection{Existence of the injection}

In this subsection, we prove that the map $\pi_S$ in Theorem \ref{thm:map} induces an injective map $i_\kappa$ between the respective boundaries. The next subsection proves the continuity of $i_\kappa$

The first lemma says that any weakly $\kappa$-Morse median ray determines a point in $\partial \calC(S)$.

\begin{lemma}\label{lem:kappa hier}
Let $h$ be a weakly $\kappa$-Morse median ray.  Then $\mathrm{diam}_S(h) = \infty$ and, in particular,   $\mathrm{diam}_Y(h) < \infty$ for all non-$\nest$-maximal $Y \in \mathfrak S$.
\end{lemma}

\begin{proof}
Suppose that there is some non-$\nest$-maximal $Y \in \mathfrak S$ so that $\diam_Y(h) = \infty$.  By Proposition \ref{prop: active path}, there exists a subsequence which we will also label $h_n$ and connected intervals $I^n_Y \subset h_n$ so that 
\begin{itemize}
\item $I_n^Y$ is contained in a bounded neighborhood of $E_Y$, the product region for $Y$ in $\calX$, and
\item $\diam(I^n_Y) \to \infty$ as $n \to \infty$,
\item The restriction $\pi_Y|_{h_n}:h_n \to \calC(Y)$ is coarsely constant off of $I^n_Y$.
\end{itemize}

It follows that there exists a connected segment $I^{\infty}_Y\subset h$ of infinite diameter, and hence $h$ eventually lives in a neighborhood of $E_Y$, the product region for $Y$ in $\calX$.

Using a standard argument, one can build uniform quasi-geodesic segments $q_n$ with endpoints $a_n,b_n \in I^{\infty}_Y \subset E_Y$ which contain points $x_n \in q_n$ so that $d_{\calX}(x_n, h) \asymp d_{\calX}(a_n,b_n)$ uniformly, while $d_{\calX} (a_n, b_n) \to \infty$.  This violates our assumption that $h$ is weakly $\kappa$-Morse, providing the contradiction.
\end{proof}

The next lemma characterizes when the projections of two median rays have distinct limits in $\partial \calC(S)$ via their medians:

\begin{lemma}\label{lem:well-defined}
Let $h, h'$ be two median rays such that $diam_{\calC (S)}(h)=diam_{\calC (S)}(h')=\infty$ and let $\lambda_h,\lambda_h'$ denote the corresponding points in $\partial \calC(S)$. We have

$$m_{\calX}(\go, h(s),h'(t) ) \rightarrow \infty \iff \lambda_h=\lambda_{h'}.$$

\end{lemma}

\begin{proof}
The backwards direction is immediate by the definition of the median (item (1) of Lemma \ref{lem:gate}) and the fact that $\pi_S$ is coarsely Lipschitz.

If $\lambda_h \neq \lambda_{h'},$ then there exist $t_0 \in [0, \infty)$ and a point $m \in \calC(S)$ such that $$m_S(\go, h(s), h'(t))\underset{3\delta}{\asymp}m,$$ for all $s,t \geq t_0$, where $\delta$ is the hyperbolicity constant for $\calC(U)$ for all $U \in \mathfrak S$.

Fix a geodesic $b$ connecting $\pi_S(\go),m$ in $\calC(S)$ and let $D$ be a constant large enough, depending only on $\delta$, so that $m_S(h(s), h'(t), \go) \in N(b,D)$ for all $s,t \in [0,\infty).$ We claim that there exists a constant $C$, depending on $h,h'$ and $\calX$ such that $\text{diam}(M_U) \leq C$ for all $U \sqsubsetneq S$, where 

$$M_U=\{m_U(\go,h(s), h'(t) )|s,t \in [0,\infty)\}.$$

Let $E$ be as in the bounded geodesic image axiom (Axiom \ref{BGIA}) and let $E'=E+D.$  There are two cases to consider:

\begin{enumerate}
    \item Case 1: If $U$ is a domain such that $\rho^U_S \cap N(b,E') = \emptyset$, then using the bounded geodesic image axiom (Axiom \ref{BGIA}), all $m_U(\go, h(s), h'(t))$ must be within $E$ of $\rho^S_U(m)$ for all $s,t \geq t_0.$ Further, since the geodesic $b$ connecting $\go$ and $m$ stays $E' \geq E$ from $\rho^U_S,$ we have 
    $$\rho^S_U(b)\underset{E}{\asymp}\pi_U(\go)\underset{E}{\asymp}\rho^S_U(m),$$
    for all $s,t \geq t_0.$ If at least one of $s,t$ is less than or equal to $t_0,$ then, using our choice of $D,E'$ above, we have $m_S(h(s), h'(t), \go) \in N(b,D),$ and hence
    $$d_S(m_S(h(s), h'(t), \go), \rho^U_S) \geq E.$$
    This implies that $m_U(h(s), h'(t), \go)$ must be within $E$ of $\rho^S_U(b)$ in $\calC(U)$. We conclude that $m_U(h(s), h'(t), \go)$ is within $2E$ of $\pi_U(\go)$ for all $s,t \geq 0.$


 \item Case 2: If $U$ is a domain with $\rho^U_S \cap N(b,E') \neq \emptyset$ that is, some points of $\rho^U_S$ are within $E'=E+D$ of $b$, then there exists $s_0$, independent of $U$, such that for all $s_1,s_2, t_1,t_2 \geq s_0,$ then $\pi_U(h(s_1))\underset{E}{\asymp}\pi_U(h(s_2))$ and $\pi_U(h'(t_1))\underset{E}{\asymp}\pi_U(h(t_2))$. Hence, there exists a constant $\delta',$ depending only on the hyperbolicity constant $\delta$, such that we have $\text{diam}(M_{U}) \leq \text{diam}_U(\{\go, h(s_0), h'(s_0)\})+\delta'$, for all $s,t \in [0,\infty).$ 

\end{enumerate}

Taking the constant $C=\text{max}(2E, \text{diam}_U(\{\go, h(s_0), h'(s_0)\})+\delta')$ gives the desired statement.

\end{proof}

In fact, the proof of the above lemma shows the following.

\begin{corollary}\label{cor:unbounded_medians}
Let $\{h_n\},h$ be median rays in an HHS $\calX$  with $\diam_S(h_n)=\diam_S(h)=\infty$ for each $n.$ If $p_n(s,t)=m_\calX(h_n(s),h(t), \go)$, then
$$\sup \{d_\calX(\go,p_n(s,t))| n \in \mathbb{N}, s,t \in \mathbb{R}^+\}=\infty \iff \sup \{d_S(\go,p_n(s,t))| n \in \mathbb{N}, s,t \in \mathbb{R}^+\}=\infty.$$

\end{corollary}

Lemmas \ref{lem: existence of median rays}, \ref{lem:kappa hier}, and \ref{lem:well-defined} as well as Corollary \ref{cor:medians_vs_morseness_LQC} combine to give us the following, which is part (1) of Theorem \ref{thm:map}:

\begin{proposition}\label{prop:the map} Let $X$ be an HHS with an unbounded products, let $S$ denote the maximal domain, and let $G< \mathrm{Aut}(\calX)$. The map $\pi_S: \calX \rightarrow \calC(S)$ induces a $G$-equivariant injective map
$$i_{\kappa}: \partial _\kappa \calX \rightarrow \partial \calC (S).$$

\end{proposition}

\begin{proof} We define the map $i_{\kappa}$ as follows. For a point $[\alpha] \in \partial_\kappa \calX$, we let $h_\alpha$ denote a $\kappa$-Morse median ray representing it as in Lemma \ref{lem: existence of median rays}. By Lemma \ref{lem:kappa hier} and Lemma \ref{lem:infinite_projection_median}, we have $\diam_S(h_\alpha) = \infty$, where $S \in \mathfrak S$ is $\nest$-maximal. Since $\pi_S(h) \subset \calC(S)$ is a quasi-geodesic, it determines a limit point $\lambda_{h_\alpha} \in \partial \calC (S)$. Define $$i([\alpha])=\lambda_{h_{\alpha}}.$$ To show that $i$ is well-defined, suppose that $[\alpha]=[\beta],$ that is, $\alpha$ and $\beta$ $\kappa$-fellow travel each other. Hence, the median rays $h_\alpha,h_\beta$ must also $\kappa$-fellow travel each other. Corollary \ref{cor:medians_vs_morseness_LQC} gives us that $m(h_\alpha,h_ \beta, \go) \rightarrow \infty.$ Therefore, Lemma \ref{lem:well-defined} shows that $\lambda_{h_\alpha}=\lambda_{h_\beta}$ which establishes that $i_{\kappa}$ is well-defined. Injectivity is also immediate by combining Corollary \ref{cor:medians_vs_morseness_LQC} and Lemma \ref{lem:well-defined}.

Since $\pi_S$ is $G$-equivariant, it is straight-forward to check $G$-equivariance of the map $i_{\kappa}$, so we leave it as an exercise for the interested reader.
\end{proof}

\subsection{Continuity} \label{subsec:continuity}

In this subsection, we prove item (2) of Theorem \ref{thm:map}. Before we do so, we briefly review the topology on the Gromov boundary of a non-proper quasi-geodesic hyperbolic metric space and the topology of the sublinearly Morse boundary.  We point the reader toward \cite{BenKap} and \cite{QRT20} for more details.

Given a metric space $X$, and points $x,y,p \in X$, the \emph{Gromov product} of $x,y$ with respect to $p$ is defined as 
$$(x,y)_p:=\frac{1}{2}(d(x,p)+d(y,p)-d(x,y)).$$

When $X$ is a (not necessarily proper) hyperbolic space, a sequence $(x_n) \in X$ is said to \emph{converge to infinity} if $(x_i,x_j)_p \rightarrow \infty$ for some (equivalently any) $p \in X.$ Two such sequences $(x_n), (y_n)$ are said to be equivalent if $(x_i,y_j)_p \rightarrow \infty $ as $i,j \rightarrow \infty.$ The $\emph{Gromov boundary}$ of $X$, denoted $\partial X$, is the collection of such equivalence classes. The Gromov boundary is given with respect to the following topology. For a point $\zeta \in \partial X$ and for $r>0$, we define
$$U(\zeta,r):=\{\eta \in \partial X| \text{for some }(x_n), (y_n) \text{ with } [x_n]=\zeta, [y_n]=\eta, \text{ we have } \underset{i,j \rightarrow \infty}{\lim}\text{inf}(x_i,y_j)_p \rightarrow \infty\}.$$

The Gromov boundary $\partial X$ is endowed with a topology by setting the basis of neighborhoods for any $\zeta \in \partial X$ to be the collection $\{U(\zeta,r)|r \geq 0\}.$ The resulting topology is independent of $p.$ 

It is immediate from the definitions that if $\alpha$ is a quasi-geodesic ray in $X$ and $x_n,y_n \in \alpha$ are two unbounded sequences, then $(x_n)$ and $(y_n)$ converge to infinity and $[(x_n)]=[(y_n)]$ in $\partial X.$ Hence, each quasi-geodesic ray $\alpha$ defines a unique point, denoted $[\alpha]$ in $\partial X.$

Moreover, we can extend the Gromov product to rays, via $(\alpha, \beta)_p = \lim_{s,t \to \infty} (\alpha(s), \beta(t))_p$.  We observe that $[\alpha] = [\beta]$ if and only if $(\alpha, \beta)_p = \infty$ for any $p \in X$.

We leave the proof of the following lemma to the reader: 

\begin{lemma}\label{lem:convergence}
Let $X$ be a $\delta$-hyperbolic space and let $h$ be a quasi-geodesic ray. If $(h_n)$ is a sequence of quasi-geodesic rays based at $h(0)$ such that $(h,h_n)_{h(0)} = \infty$, then $[h_n] \rightarrow [h]$ in $\partial X.$
\end{lemma}

The above lemma immediately provides the following median criterion for continuity for second countable topologies on $\partial_{\kappa} \calX$ which is easy to establish:

\begin{proposition}\label{prop:cont criterion}
Suppose that $\calX$ is an HHS, $h$ a median ray in $\calX$, and $(h_n)$ a sequence of median rays based at $h(0)$ with $\{\med_S(h(0), h_n, h)\}_n$ unbounded.  Then $[h_n] \to [h]$ in $\partial \calC(S)$.
\end{proposition}

Before we prove continuity, we briefly remind the reader of the topology on the sublinearly Morse boundary and the associated terminology from \cite{QRT20}. A quantity $D$ is said to be \emph{small compared to} a radius $r$ if $D \leq \frac{r}{2\kappa(r)}.$ Recall that for a quasi-geodesic ray $\beta$, and for $r \geq 0$, if $t_r$ is the first time with $\|\beta(t_r)\|=r$, the point $\beta(t_r)$ is denoted $\beta_r.$ Further, the subsegment $\beta([0,t_r])$ is denoted $\beta|_r.$ 

\begin{definition}(neighborhood basis) Let $\kappa$ be a sublinear function and let $\beta$ be a $\kappa$-Morse quasi-geodesic ray in $X$. For $r \geq 0$, the set $\calU(\beta, r) \subseteq \partial_\kappa X$ is defined as follows. An equivalence class $\eta \in \partial_\kappa X$ belongs to $\calU(\beta,r)$ if for any $(q,Q)$-quasi-geodesic $\alpha \in \eta$, with $\mm_\alpha(q,Q)$ small compared to $r$, we have

$$\alpha_r \subseteq \calN_\kappa(\beta,m_\beta(q,Q)).$$ The topology on $\partial_\kappa X$ is defined by declaring the basis of neighborhoods for any $\zeta \in \partial_\kappa X$ to be the collection $$\{\calU(\beta, r)| \beta \in \zeta \text{ and }r\geq 0\}.$$

\end{definition}


\begin{lemma}\label{lem:sliding_medians}
If $h_i,h$ are $\kappa$-Morse median rays with $[h_i] \rightarrow [h] \in \partial_\kappa \calX$, then the set $$A=\{d_\calX(\go,m_\calX(\go,h_i(j),h(j)))|i,j \in \mathbb{N} \}$$ is unbounded.
\end{lemma}

\begin{proof}
Let $h_i,h$ be as in the statement, $H_i=\hull_\calX(h_i \cup h)$ and let $f_i: Q_i \rightarrow H_i$ be the cubical models for $H_i,$ in particular, all such maps $f_i$ are $K$-median quasi-isometries. Since $h$ is $\kappa$-Morse, if $g_i$ is a coarse inverse for $f_i$, then the quasi-geodesic rays $g_i(h)$ are all $\kappa$-Morse in $Q_i$. Hence, there exists an excursion constant $c$, depending only on $K$, $\mm_h$ and $\kappa$ (and not on $i$) such that each $g_i(h)$ crosses an infinite sequence of excursion hyperplanes $\{(k_i^n, x_i^n)\}_{n\in \mathbb{N}}$ at excursion points $x_i^n \in k_i^n.$ We remark that since the excursion constant $c$ is independent of $i$, then for a fixed $n,$ the set $$B_n=\{d_{Q_i}(g_i(\go),x_i^n)| i \in \mathbb{N}\}$$ is (coarsely) bounded above by $\kappa(n)$. In fact, in light of Remark \ref{rmk: median paths}, since $g_i(h)$ are all uniform quality median paths, the set 
$$B_n'=\{d_{Q_i}(g_i(\go),y_i^n)| i \in \mathbb{N} \,\text{ and}\, y_i^n \in k_i^n \cap g_i(h)\}$$ is also bounded above.

\begin{claim}
For each $n$, there exists some $i$ such that $g_i(h_i)$ crosses $k_i^n$.
\end{claim}

Before we prove the claim, we explain how it gives the desired statement. Using the claim, for each $n,$ there exists $y^n_i \in g_i(h_i) \cap k_i^n.$ By definition of $x_i^n$, we have $x^n_i \in g_i(h) \cap k_i^n$. Since $k_i^n$ is a hyperplane and thus median convex, we have $m_{Q_i}(g_i(\go), x^n_i,y^n_i) \in k_i^n$. As $d_{Q_i}(\go, k_i^n) \rightarrow \infty$ as $n \rightarrow \infty,$ we have $m_{Q_i}(x^n_i,y^n_i,g(\go))$ is unbounded. More precisely, for each constant $C,$ there exists an $i$ such that $d_{Q_i}(g_i(\go), m_{Q_i}(g_i(\go),x^n_i,y^n_i))>C.$ Since the maps $f_i:Q_i \rightarrow H_i$ are all $K$-median quasi-isometries, the set $A$ must be unbounded. Hence, in order to conclude the proof, it only remains to prove the claim.

\begin{proof}[Proof of Claim] For the sake of contradiction, suppose that there exists some $n$ such that for each $k_i^n,$ the quasi-geodesic ray $g_i(h_i)$ does not cross $k_i^n.$ Since $n$ is fixed, the hyperplanes $k_i^n,k_i^{n+1}$ are all $L$-well-separated where $L$ is independent of $i$. Define 
$$t_0:=\text{sup}\{t| g_i(h(t)) \cap k_i^{n+1} \neq \emptyset \text{ for some }i\}.$$ Such a choice is possible as the set $B_{n+1}'$ defined above is bounded and since each $f_i$ is a $K$-median quasi-isometry. By Lemma \ref{lem: crossing well-separated hyperplanes implies linear divergence}, we have that $d_{Q_i}(g_i(h(t)), g_i(h_i(t))) \asymp t$ for all $t \geq t_0$. Hence, $d_\calX(h(t), h_i(t)) \asymp t$ for all $t \geq t_0$ which violates the fact that $[h_i] \rightarrow [h] \in \partial_\kappa \calX.$

\end{proof}

\end{proof}

With the above lemma in hand, continuity is straight-forward:

\begin{theorem} \label{thm: continuity} Let $\kappa$ be a sublinear function. The injective map $B_\kappa: \partial_\kappa \mathcal{X} \rightarrow \partial \calC(S)  $ is continuous.

\end{theorem}

\begin{proof}

Let $\aa_n \rightarrow \aa \in \partial_\kappa \mathcal{X}$ and let $h_n,h$ be median rays with $[h_n]=\aa_n$ and $[h]=\aa.$  As $h_n \rightarrow h \in \partial_\kappa \mathcal{X}$, using Lemma \ref{lem:sliding_medians}, we have $d_\calX(m_\calX(\go, h_n(t), h(s)), \go) \rightarrow \infty$ as $s,t \rightarrow \infty,$ and thus, using Corollary \ref{cor:unbounded_medians} $d_S(\go,m_S(h_n(t), h(s),\go)) \rightarrow \infty.$ Since $d_S(\go, m_S(h_n(t), h(s))$ coarsely agrees with $(h_n(s), h(t))_{\go}$ in $\calC(S),$ the conclusion follows by Proposition \ref{prop:cont criterion}.


\section{Hierarchical characterizations of sublinear Morseness} \label{sec:HHS char}

In this section, we characterize $\kappa$-Morseness for median rays in an HHS in terms of its hierarchical structure.  For the rest of this section, let $\calX$ be a proper HHS with unbounded projections.

We give two characterizations.  The first is in terms of a sublinear growth rate of subsurface projections along a median ray.

\begin{theorem}\label{thm:bounded projection characterization} Let $\calX$ be a proper HHS with unbounded products and let $p$ be the complexity of $\calX$.  The following hold:

\begin{enumerate}
    \item If $\gamma$ is a $\kappa$-Morse quasi-geodesic ray in  $\calX$, then $\gamma$ has $\kappa$-bounded projections.
    \item If $h$ is a median ray with $\kappa$-bounded projections and $\kappa^p$ is sublinear, then $h$ is $\kappa^p$-Morse.
    
    \end{enumerate}
\end{theorem}

The second characterization is in terms of a sublinear rate of progress of the shadow of the ray in the top level hyperbolic space $\calC (S)$.

\begin{theorem} \label{thm:characterization via speed} Let $\calX$ be an HHS of complexity $p$ and let $h$ be a median ray in $\calX$. We have the following.

\begin{enumerate}
    \item If $q$ is $\kappa$-Morse quasi-geodesic ray then $q$ has a $\kappa^p$-persistent shadow.


\item  If $h$ is a median ray and $h$ has a $\kappa$-persistent shadow, then $h$ is $\kappa$-Morse.

\end{enumerate}
\end{theorem}

We note both that having a $\kappa$-persistent shadow appears to be a stronger property, and that these proofs are intertwined via the ``passing up'' Lemma \ref{lem:passing-up}.  In particular, item (1) of Theorem \ref{thm:characterization via speed} is a direct consequence of item (1) of Theorem \ref{thm:bounded projection characterization} and Lemma \ref{lem:passing-up}, while item (2) of Theorem \ref{thm:bounded projection characterization} involves Lemma \ref{lem:passing-up} and an application of item (1) of Theorem \ref{thm:characterization via speed}.

\subsection{Sublinearly bounded projections and persistent shadows}

We begin by recalling the definitions of $\kappa$-bounded projections and $\kappa$-persistent shadows:

\begin{definition}[$\kappa$-bounded projections]\label{defn:kbp}
A ray $\gamma$ has \emph{$\kappa$-bounded projections} if there exists $c>0$ so that for each $Y \sqsubsetneq S$ and all $t$ we have
\begin{equation}\label{eq:kbp}
d_Y(\go, \gamma(t)) \leq c\kappa(t).   
\end{equation}
\end{definition}

While the above notion is regarding the diameter of the projection to each domain in the hierarchy, the following definition only involves the projection in the top-level curve graph.

\begin{definition}[$\kappa$-persistent shadow]\label{defn:kps}
A ray $\gamma$ has a \emph{$\kappa$-persistent shadow} if there exists a constant $c$ such that for all $s\leq t$, we have
\begin{equation}\label{eq:kps}
d_S(\gamma(s), \gamma(t)) \geq  c\cdot \frac{t-s}{\kappa(t)}-c.
\end{equation}
\end{definition}

The following statement is well-known to experts and follows from a standard hierarchical arguments using induction and the passing up Lemma \ref{lem:passing-up} (see e.g. \cite[Proposition 7.1]{QRT20}). Recall that the complexity of an HHS $\calX$ is defined to be the cardinality of a largest pairwise non-transverse domains in $\calX.$

\begin{lemma} \label{lem: passing up consequence} Let $\mathcal{X}$ be an HHS with complexity $p=\xi(\calX)$. For any $x,y \in \mathcal{X}$, if $d_U(x,y)<D$ for all proper domains $U$, and for some $D>1,$ then we have 

$$d_\calX(x,y) \prec d_S(\pi_S(x),\pi_S(y))D^{p}.$$

\end{lemma}

\begin{proof} Fix $D >1.$ The proof is by induction on $\xi(\calX)$.  When $\xi(\calX) = 1$, then $\calX = \calC(S)$ and the statement is trivial. Otherwise, suppose that the statement is true for any HHS of complexity less than $\xi(\calX)$. Let $\theta=\theta_0$ be as in the distance formula (Theorem \ref{thm:distance formula}). This gives us a constant $K=K(\theta)$ such that 
$$d_{\calX}(x,y) \asymp d_S(x,y) + \sum_{V \in \mathrm{Rel}_{\theta}(x,y)-S}d_V(x,y).$$ Let $\calU^{max}$ denote that set of domains $U \in \mathrm{Rel}_{\theta}(x,y)$ which are penultimate with respect to $\nest$. For each $U \in \calU^{max}$ let $x',y' \in F_U$ be such that $\pi_V(x) \asymp \pi_V(x')$ and $\pi_V(y) \asymp \pi_V(y')$ as in Remark \ref{rmk:product_region_properties}. Using the distance formula on $F_U,$ we have

\begin{align*} d_{\calX}(x,y) &\asymp d_S(x,y) + \sum_{V \in \mathrm{Rel}_{\theta}(x,y)-S}d_V(x,y)\\
&\prec d_S(x,y) + \sum_{U \in \calU^{max}}d_{F_U}(x',y').\\
\end{align*}

  Since $\xi(U) < \xi(\calX)$, we can apply the inductive hypothesis to get

$$ d_{F_U}(x',y') \prec d_U(x',y') \cdot D^{\xi(U)} \asymp d_U(x,y) \cdot D^{\xi(U)} \leq d_U(x,y) \cdot D^{\xi(\calX)}.$$ 

Hence we can conclude that
$$d_{\calX}(x,y) \prec d_S(x,y) + D^{\xi(\calX)}\sum_{U \in \calU^{max}} d_U(x,y).$$

Further, using the assumption of this lemma, we have $d_U(x,y) < D$ for every proper domain $U$, and hence

$$d_{\calX}(x,y) \prec d_S(x,y) + D^{\xi(\calX)+1}|\calU^{max}|.$$

To complete the proof, it suffices to bound the number of domains in $\calU^{max}$ in terms of $d_S(x,y)$ and $\calX$.  This is precisely the ``passing up'' Lemma \ref{lem:passing-up}, which implies $d_S(x,y)$ gives an upper bound on $|\calU^{max}|$.  This completes the proof.

\end{proof}

The following is an immediate corollary:

\begin{corollary} \label{cor: bounded projections imply large distance}
Let $\mathcal{X}$ be an HHS with complexity $p$ and let $\gamma$ be a quasi-geodesic ray. If $\gamma$ has $\kappa$-bounded projections, then $\gamma$ has a $\kappa^p$-persistent shadow.

\end{corollary}


The following lemma follows easily from the definitions.

\begin{lemma}\label{lem:relating_sublinear_projections}
Let $\gamma,\gamma'$ be two quasi-geodesic rays starting at $\go$. If $\gamma \sim_\kappa \gamma'$, then $\gamma$ has $\kappa$-bounded projections if and only if $\gamma'$ does.
\end{lemma}

\begin{proof}
The proof is immediate using the definitions along with Lemma \ref{lem: relating sublinearness}. Namely, if $\gamma$ has $\kappa$-bounded projections, and $\gamma' \in \calN_\kappa(\gamma,\nn)$ for some $\nn$, then for each $t$, there exists an $s$ such that $d_\calX(\gamma(s), \gamma'(t)) \leq \mm \kappa(t)$, where $\mm$ depends only on $\nn$ and the quasi-geodesic constants of $\gamma'.$  Now, since $\gamma$ has $\kappa$-bounded projections, there exists a constant $C$ such that $d_U(\go,\gamma(s)) \leq C \kappa(s)$ for all proper domains $U$ and all $s \in [0,\infty).$ However, since $d_\calX(\gamma(s), \gamma'(t)) \leq \mm \kappa (t),$ Lemma \ref{lem: relating sublinearness} gives us a constant $D$ such that $\kappa(s) \leq D \kappa(t).$ Thus, we get a constant $C'$ such that $d_U(\go, \gamma(s)) \leq C'\kappa(t),$ where $C'$ depends on $\kappa$, $C$ and the quasi-geodesic constants of $\gamma, \gamma'$. This gives us

\begin{align*}
d_\calX (\go, \gamma'(t)) &\leq d_\calX (\go, \gamma(s))+d_\calX(\gamma(s), \gamma'(t))\\
&\leq d_\calX (\go, \gamma(s))+\mm \kappa(t)\\
&\leq C'\kappa(t)+\mm\kappa(t)\\
&=(C'+\mm) \kappa(t).
\end{align*}
\end{proof}

We also get a similar statement for the $\kappa$-persistent shadow property:

\begin{lemma}\label{lem:relating shadows}
Let $\gamma, \gamma'$ be two quasi-geodesic rays starting at $\go$ and let $\kappa$ be a sublinear function so that $\kappa^2$ is also sublinear.  If $\gamma \sim_{\kappa} \gamma'$ and $\gamma$ has a $\kappa$-persistent shadow, then $\gamma'$ has a $\kappa^2$-persistent shadow.
\end{lemma}

\begin{proof} Let $s',t' \in [0,\infty)$ with $s' \leq t'.$ Since $\gamma \sim_{\kappa} \gamma'$ and since $\gamma,\gamma'$ are quasi-geodesics, there exist $s,t \in [0,\infty)$ such that with $d_\calX(\gamma(s), \gamma'(s')) \prec \kappa(s')$ and $d_\calX(\gamma(t),\gamma'(t')) \prec \kappa(t')$. Without loss of generality, we may assume that $s \leq t.$ Since $\gamma'$ is a quasi-geodesic, the triangle inequality gives us
\begin{align*}
    t'-s' &\asymp d_\calX(\gamma'(s'), \gamma'(t'))\\
    &\prec \kappa(s')+d_\calX(\gamma(s),\gamma(t))+\kappa(t') \\
    &\asymp \kappa(s')+(t-s)+\kappa(t').
\end{align*}

Using the assumption that $\gamma$ has $\kappa$ persistent shadow, we have
$$t-s \prec d_S(\gamma(s), \gamma(t))\kappa(t).$$ By the triangle inequality, we have 
\begin{align*}
    t-s &\prec d_S(\gamma(s), \gamma(t))\kappa(t)\\
    &\leq \left[(d_S(\gamma(s), \gamma'(s'))+d_S(\gamma'(s'), \gamma'(t'))+d_S(\gamma'(t'), \gamma(t))\right]\kappa(t).
\end{align*}

Since the map $\pi_S$ is coarsely Lipshitz, we have $d_S(\gamma(s), \gamma'(s')) \prec \kappa(s')$ and $d_S(\gamma(t), \gamma'(t')) \prec \kappa(t')$. Hence, combining the above, we get 
\begin{align*}
    t-s &\prec \left[\kappa(s')+d_S(\gamma'(s'), \gamma'(t'))+\kappa(t')\right] \kappa(t)\\
    &\leq \left[\kappa(t')+\kappa(t')d_S(\gamma'(s'), \gamma'(t'))+\kappa(t')\right] \kappa(t).
\end{align*}
where the last inequality holds as $s' \leq t'$ and as $\kappa \geq 1.$ Thus, we have
$$t-s \prec \left[2+d_S(\gamma'(s'), \gamma'(t'))\right] \kappa(t')\kappa(t).$$

Using the above, since $t'-s' \prec \kappa(s')+(t-s)+\kappa(t'),$ we have 
$$t'-s' \prec \kappa(s')+\left[2+d_S(\gamma'(s'), \gamma'(t'))\right] \kappa(t')\kappa(t)+\kappa(t').$$

Now, since $s' \leq t'$ and as $\kappa \geq 1,$ we have
\begin{align*}
t'-s' &\prec \kappa(t')\kappa(t)+\left[2+d_S(\gamma'(s'), \gamma'(t'))\right] \kappa(t')\kappa(t)+\kappa(t')\kappa(t)\\
&=(4+d_S(\gamma'(s'),\gamma'(t')))\kappa(t')\kappa(t).
\end{align*}

Further, since $|t-t'| \asymp d_\calX (\gamma(t), \gamma'(t')) \prec \kappa(t'),$ Lemma \ref{lem: relating sublinearness} gives us that $\kappa(t) \prec \kappa(t')$. Hence, we have $t'-s' \prec (4+d_S(\gamma'(s'),\gamma'(t')))\kappa(t')^2.$ Therefore, we get 
$$\frac{t'-s'}{\kappa^2(t')} \prec d_S(\gamma'(s'), \gamma'(t')).$$
\end{proof}

\subsection{Forward Direction of Theorem \ref{thm:bounded projection characterization}}

The argument for this direction is a sublinear generalization of the hierarchical argument for the Morse case, see e.g. the proof of \cite[Theorem 6.3]{DT15}.  The main idea is that if the subsurface projections grow too quickly along a sublinearly Morse median ray, then it must spend a long time near a standard product region, which violates the assumption that it's sublinearly Morse.

\begin{theorem}\label{thm: forward direction subsurface projections} If $\gamma$ is a $\kappa$-Morse quasi-geodesic ray, then $\gamma$ has $\kappa$-bounded projections and a $\kappa^p$-persistent shadow.
\end{theorem}

\begin{proof}

By Corollary \ref{cor: bounded projections imply large distance}, it suffices to prove that $\gamma$ has $\kappa$-bounded projections.

Let $h$ be a median representative of $\gamma$, so that $h \sim_{\kappa} \gamma$ and hence $h$ is $\kappa$-Morse. Using Proposition \ref{prop: Refining Morse}, for any $q, Q \in \mathbb{R}^+,$ there exists a constant $\cc=\cc(q, Q, \kappa)$ such that for any point $p$ living on some $(q, Q)$-quasi-geodesic starting at $h(t_1)$ and ending on $h(t_2),$ we have
\begin{equation}\label{1st equation characterization}
d_{\calX}(p,h|_{[t_1,t_2]}) \leq \cc \kappa (t_2). 
\end{equation}

For the sake of contradiction, suppose that for each $n$, there exists a proper domain $U_n$ and a time $t_n$ such that $d_{U_n}(h(0), h(t_n)) \geq n \kappa(t_n).$ Using the active interval Proposition \ref{prop: active path}(3), there exists a subinterval of $h([0, t_n])$, say $h|_{[t_{n_1},t_{n_2}]}$, which is coarsely contained in the product region $P_{U_n}$ and 

$$d_{\calX} (h(t_{n_1}), h(t_{n_2})) \underset{\calX}{\succ} d_{U_n}(h(t_{n_1}), h(t_{n_2})) \geq n \kappa(t_n) \geq n \kappa(t_{n_2}).$$

Since both $h(t_{n_1}), h(t_{n_2})$ coarsely live in $P_{U_n}$ and $d_{U_n}(h(t_{n_1}), h(t_{n_2})) \geq  n \kappa(t_{n_2})$, there exist uniform quality quasi-geodesics $\beta_n$ starting at $h(t_{n_1})$ and  ending at $h(t_{n_2})$, and points $p_n 
\in \beta_n$ such that $d_{\calX}(p_n,h|_{[t_{n_1},t_{n_2}]})$ is linear in $d_{U_n}(h(t_{n_1}), h(t_{n_2}))$ contradicting equation \eqref{1st equation characterization}.

Hence $h$ has $\kappa$-bounded projections, and so $\gamma$ does as well by Lemma \ref{lem:relating_sublinear_projections}.
\end{proof}

\subsection{Backwards Direction} \label{subsec:backwards direction HHS}

The other direction requires more work and new ideas not contained in the Morse case.

The idea is roughly as follows: Let $h$ be a median ray with $
\diam_S(\pi_S(h)) = \infty$.  We first show that the progress made by $\pi_S(h)$ through $\calC(S)$ is encoded by a sequence of $L$-well-separated hyperplanes crossed by $f(h)$ in the cubical model $f:H \to Q$ for $H = \hull_{\calX}(h)$ (Lemma \ref{lem:bounded projections imply excursion}).  When $h$ has $\kappa$-bounded projections, we can upgrade this sequence to consist of $\kappa$-excursion hyperplanes (Lemma \ref{lem:bounded projections imply excursion}).  From there, the statement follows from cubical facts in Section \ref{sec:CCC}.

The following lemma states that if two median convex subsets project far apart in $\calC(S)$, then their images to each other under their respective gate maps must have bounded diameter.

\begin{lemma}\label {lem: far in the top level implies small gates}  Let $\calX$ be an $HHS$. For each constant $K''$, there exist constants $K,K'$, depending only on $K''$ and $\calX$ such that for any $K''$-median convex sets $A, B \subseteq \calX$ with $d_{S}(\pi_S(A), \pi_S(B)) \geq K$, we have $diam_{\calX}(\gate_B(A)) \leq K'.$

\end{lemma}

\begin{proof}
Let $A,B$ be as in the statement and let $A_U=\pi_U(A), B_U=\pi_U(B)$ for $U \in \mathfrak S$. Notice that by the definition of the gate map (Lemma \ref{lem:gate}), it suffices to prove that the nearest point projection of $A_U$ to $B_U$ has uniformly bounded diameter. For $x,y \in A,$ denote $x_U,y_U$ some points in $\pi_U(x), \pi_U(y)$ respectively. We will show that if $x_U',y_U'$ are nearest point projections of $x_U,y_U$ to $B_U,$ then $d_U(x_U', y_U')$ is uniformly bounded.

Since $A_S$, $B_S$ are assumed to be far apart in $\calC(S)$, their closest point projections have bounded diameter by hyperbolicity.  Now suppose that $U \sqsubset S$.  Again, separation of $A_S$ from $B_S$ says that the set $\rho^U_S \subseteq \calC(S)$ must be far from at least one of $A_S$ or $B_S$. If $\rho^U_S$ is far from $B_S,$ then, by the bounded geodesic image axiom (Axiom \ref{BGIA}), that $\diam_U(B_U) < E$, and hence $d_U(x'_U,y'_U) <E$. Alternatively, if $\rho^U_S$ is far from $A$, then, again by the bounded geodesic image axiom, we have $d_U(x_U,y_U) \leq E$, and hence $d_U(x'_U,y'_U)\prec E$ because the closest point projection in $\calC(U)$ to $B_U$ is coarsely Lipschitz depending only on the hyperbolicity constant of $\calC(U)$, which is uniform.  This completes the proof.
\end{proof}

The following is an immediate consequence, which will be useful to record:

\begin{corollary}\label{cor: far in the top level implies small gates}
Let $\calX$ be an $HHS$. There exist constants $K,K'$, depending only on $\calX$ such that for any $A, B \subseteq \calX$ with $d_{S}(\pi_S(A), \pi_S(B)) \geq K$, we have $diam_{\calX}(\gate_{\hull_\calX(B)}(\hull_\calX(A))) \leq K'.$
\end{corollary}

The following is a consequence of Genevois' Proposition \ref{prop:Genevois}:

\begin{lemma}\label{lem: separation in HHS's} Let $Q$ be a finite dimensional CAT(0) cube complex, $\calX$ be an HHS, $f: Q \rightarrow \mathcal{X}$ be a $K$-median quasi-isometric embedding and let $Y_1,Y_2$ be two combinatorially convex sets in $Q$. The sets $Y_1,Y_2$ are $L$-well-separated if and only if $\gate_{f(Y_1)}(f(Y_2))$ and $\gate_{f(Y_2)}(f(Y_1))$ both have diameters at most $L',$ where $L'$ is determined by $L$ and $K.$ Similarly, $L$ is determined by $L'$ and $K.$ In particular, there exists a constant $C$, depending only on $\calX$, such that if $d_{\calX}(\pi_S(f(Y_1)), \pi_S(f(Y_1)) ) \geq C,$ then $Y_1,Y_2$ are $L$-well-separated where $L$ depends only on $K$ and $\calX.$
\end{lemma}

\begin{proof}
 The proof is immediate by combining Proposition \ref{prop:Genevois} and Corollary \ref{cor: comparing gates}. The last part of the statement follows by Lemma \ref{lem: far in the top level implies small gates}.
\end{proof}







Recall that for a quasi-geodesic $q,$ if $x=q(s), y=q(t)$ for $s<t$ and $\|x\|<\|y\|$, we write $x<y.$

\begin{lemma} \label{lem: hyperbolic hulls}
Let $X$ be an $E$-hyperbolic space, $q$ be a unparameterized $(A,A)$-quasi-geodesic ray in $X$ and let $F$ be a finite set of points in $X$. There exists a constant $R$, depending only on $E,A$ and $|F|$ such that the following holds. For any constant $C \geq 0$, there exist constants $D $, $K$, depending only on $C, E$, $A$ and $|F|$ such that such that for any $p,p' \in q$, if $d(p,p') \geq K$, then there exist 2 balls $\{B_i''=B_i(x_i,R)\}_{i=1}^2$ with $ x_i \in q$ and $x_1 < x_{2}$ such that: 

\begin{enumerate}
    \item  $d(B_1'',B_{2}'')\geq C$.
    
    \item Each $B_i''$ determines two connected sets $Y_{B_i''}, Z_{B_i''} \subseteq \hull_X(q \cup F)$ such that $p, q(0) \in Y_{B_i''}$, $p' \in Z_{B_i''}$ and $\hull_X(q \cup F)-B_i''=Y_{B_i''} \sqcup Z_{B_i''}$.
    
    \item  $diam( \hull_X(q \cup F) \cap Z_{B_1''} \cap Y_{B_{2}''}) \leq D $.

\end{enumerate}

\end{lemma}

\begin{figure}
    \centering
    \includegraphics[width=.5\textwidth]{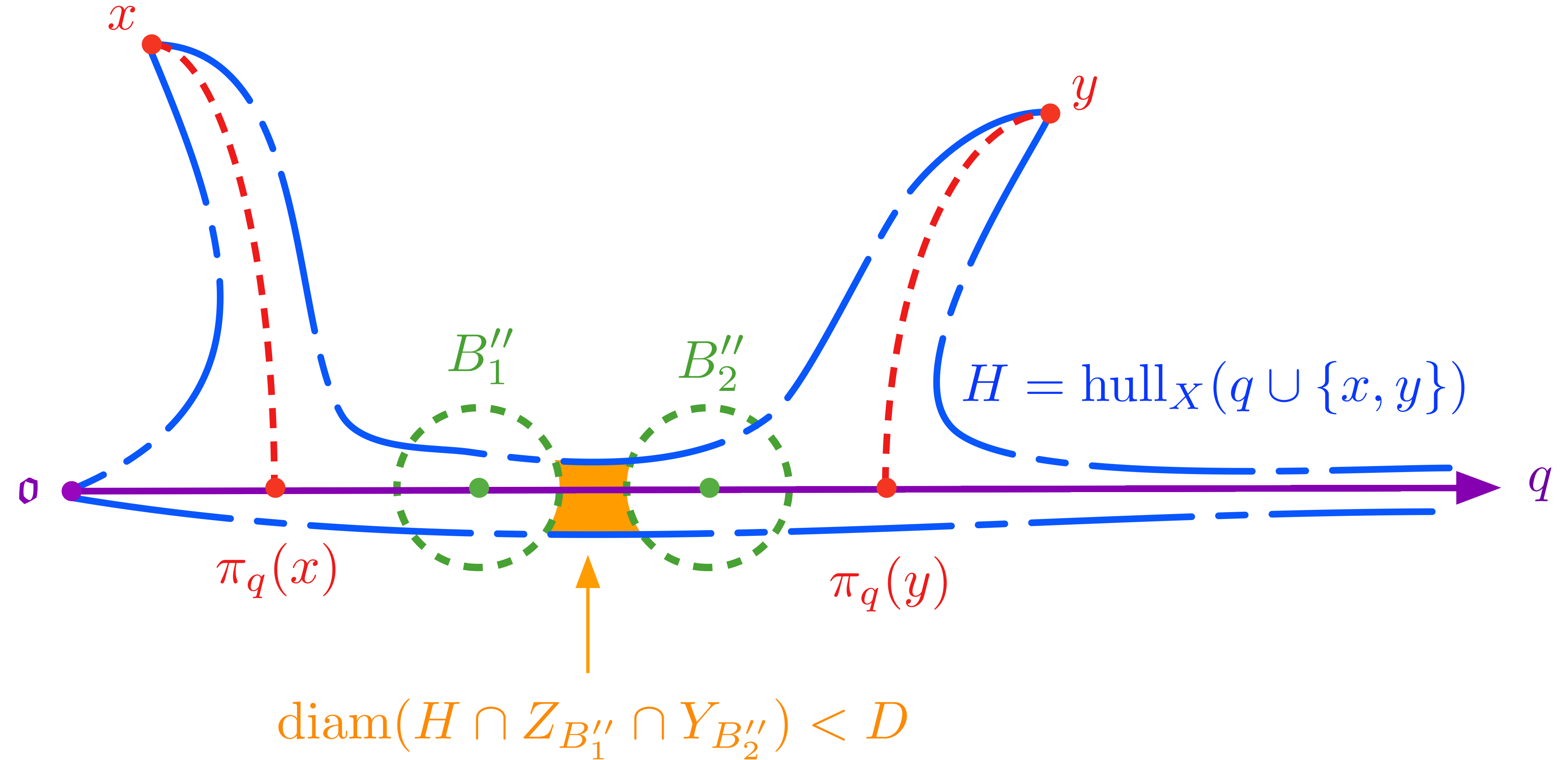}
    \caption{ The hull of finitely-many rays and points in a hyperbolic space is a quasi-tree.  In the special case where $F=\{x,y\}$, the two resulting balls of Lemma \ref{lem: hyperbolic hulls}, $B_1'', B_2''$, can be assured to both occur before, after, or between $\pi_q(x), \pi_q(y)$, as in the figure.}
    \label{fig:thin_hulls}
\end{figure}

\begin{proof} The statement follows using the fact that $hull(q \cup F)$ is $L$-quasi-isometric to a tree where $L$ depends only on $|F|+2.$ We leave out the details as an exercise for the reader.  See Figure \ref{fig:thin_hulls} for the case where $F=\{x,y\}$. This is the case that will be most significant for the following statements.
\end{proof}

\begin{remark}
The above Lemma \ref{lem: hyperbolic hulls} will be mainly used in the case where $|F|=2$. In this case, the constant $R$ of this lemma will only depend on $E$ and $A$.

\end{remark}

\begin{notation} \label{notation: relating maps}
For the rest of the section, we fix the following data. Let $h$ be a median ray with diam$(\pi_{S}(h))= \infty$, and let $F \subset \calX$ be a finite set of points. We fix a cubical model $ (g,f,Q: \rightarrow H)$ for $H=\hull_{\calX}(h \cup F)$, and we use $h',h''$ to denote the following:

\begin{enumerate}
    \item $h':=g \circ h:[0, \infty) \rightarrow Q.$
    
    \item $h'':=\pi_{S} \circ h: [0,\infty) \rightarrow \calC S.$

\end{enumerate}
Notice that in light of Remark \ref{rmk:median_ray_constants}, since $h$ is a median ray, the image $h''$ is an $(A,A)$-unparameterized quasi-geodesic where $A$ depends only on $\calX.$
\end{notation}

\begin{lemma}\label{lem:finding the hyperplane}

Let $F,h,h',h''$ be as in Notation \ref{notation: relating maps}. There exist constants $K,K'$, depending only on $\calX$ and $|F|$, such that if $d_S(p, p') \geq K$ for $p,p' \in h''$, then $h'$ crosses a hyperplane $k$ with $$\diam_{\mathcal{C}(S)}(\pi_S\circ f(k)) \leq K'.$$

\end{lemma}

\begin{proof} Let $E$ be the HHS constant and
let $R$ be the constant provided by Lemma \ref{lem: hyperbolic hulls}. Choose $C$ to be a constant large enough so that if $B_1'',B_2''$ are balls of radius $R$ centered at points in $h''$ with $d_S(B_1'',B_2'')\geq C,$ then $d_Q(B_1',B_2') \geq 1$ where $B_i'=f^{-1}(H\cap \pi_S^{-1}(B_i'')).$ Such a choice is possible since the map $\pi_S$ is coarsely Lipschitz and $f: Q \rightarrow H$ is a quasi-isometry. For such an $R$ and $C$, Lemma \ref{lem: hyperbolic hulls} provides constants $D,K$  such that if $d(p,p') \geq K$ for $p,p' \in h'',$ then there exists 2 balls $\{B_i''=B_i(x_i'', R)\}_{i=1}^{2}$ with $x_i'' \in h''$ which satisfy conclusion of Lemma \ref{lem: hyperbolic hulls}. In particular $d_S(B_1'',B_2'')\geq C.$ Using our choice of $C$, we have $d_Q(B_1',B_2') \geq 1$ where $B_i'=f^{-1}(H\cap \pi_S^{-1}(B_i''))$. Hence, there exists at least one hyperplane $k$ which separates the sets $B_1',B_2'$, in particular, $k$ intersects neither $B_1'$ nor $B_2'.$ Hence, the image $\pi_S \circ f(k)$ intersects neither $B_1''$ nor $B_2''$. Using item 3 of Lemma \ref{lem: hyperbolic hulls}, we get that diam$(\pi_S \circ f(k)) \leq D$ where $D$ depends only $C,E$, $A$ and $|F|$. Setting $K'=D$ concludes the proof.
\end{proof}

We again remark that in our main applications, we will only be concerned with the case where $|F|=2$.  In that case, the constants $K,K'$ above will depend only on $\calX.$

\begin{lemma} \label{lem:bounded projections imply excursion}

Let $h$ be a median ray which has $\kappa$-bounded projections for some sublinear function $\kappa$ with $\kappa^p$ sublinear. For a finite set $F \subset \calX$, if $(g,f,Q: \rightarrow H)$ is a cubical model for $H=\hull_{\calX}(h \cup F)$, then $g(h)$ crosses an infinite sequence of $\kappa^p$-excursion hyperplanes $\{k_i\}_{i \in \mathbb{N}}$ such that the excursion constant $c$ and $d_Q(g(\go), k_1)$ depend only on $\mathcal{X}$, $|F|$, $\kappa$ and the projection constant. In particular, when $F=\{x,y\}$ for $x,y \in \calX,$ the excursion constant $c$ and $d_Q(g(\go), k_1)$ are independent of the choices of $x,y.$
\end{lemma}

\begin{proof} First, notice that since $h$ has $\kappa$-bounded projections, Corollary \ref{cor: bounded projections imply large distance} gives us that $\pi_S(h)$ has infinite diameter. Define $h'':[0,\infty) \rightarrow \calX$ by $h''=\pi_S\circ h$ and let $q$ denote the image of the unparameterized quasi-geodesic $h''$ in $\calC (S)$. Similarly, let $h':[0,\infty) \rightarrow Q$ by the quasi-geodesic ray given by $h'=g\circ h$. Since we have a quasi-isometry $f:Q \rightarrow H=\hull_{\calX}(h \cup F)$, the map $f$ defines a quasi-isometric embedding $f:Q \rightarrow \calX,$ and hence, applying Lemma \ref{lem: separation in HHS's} gives us constants $C,L$, depending only on $\calX$ and $|F|$ so that whenever $k,k'$ are hyperplanes in $Q$ with $d_S(\pi_S(f(k)), \pi_S(f(k'))) \geq C,$ we have that $k,k'$ are $L$-well-separated. Let $N$ be the constant so that for any geodesic $\alpha$ starting and ending on the points $z,w \in q,$ we have $d_{Haus}([z,w]_q,\alpha) \leq N.$

Let $R=\text{max}\{K,K',C,N\}$, where $K,K'$ are as in Lemma \ref{lem:finding the hyperplane}. Using Corollary \ref{cor: bounded projections imply large distance}, we get a constant $D \geq 1$, depending only on $\calX$ and the projection constant, such that for any $s,t \in [0, \infty)$ with $s \leq t,$ we have
$$d_S(h''(s), h''(t)) \geq \frac{t-s}{D\kappa^p(t)}-D.$$

Observe that the function $\frac{t-s}{D\kappa^p(t)}-D$ is continuous. We choose a sequence of points $t_1, t_2, t_3,\dots t_n, \dots $ as follows. We fix $t_1=1$ and define $t_n$ by 
$$\frac{t_n-t_{n-1}}{D\kappa^p(t_n)}-D=7R.$$ This is possible since for a given $t_{n-1},$ the function $\frac{t_n-t_{n-1}}{D\kappa^p(t_n)}-D$ is increasing. Using the above equation, we have
\begin{align*}
d_S(h''(t_n), h''(t_{n-1})) &\geq \frac{t_n-t_{n-1}}{D\kappa^p(t_n)}-D\\
&\geq 7R.
\end{align*}

Let $p_n$ be the first point on $[h''(t_n), h''(t_{n+1})]_q$ with $d_S(h''(t_n), p_n) =F$ and let $\alpha_n$ be a geodesic connecting $h''(t_n), p_n.$ Now, since $d_S(h''(t_n), p_n)=R,$ and $R \geq K,$ using Lemma \ref{lem:finding the hyperplane}, the quasi-geodesic $h'$ must cross a hyperplane $k_n$ at a time $s_n \in [t_n, t_{n+1}]$ such that $\text{diam}(\pi_S (f(k_n))\leq K'$.  In particular, there exists a point $p_n' \in [h''(t_n),p_n]_q \cap \pi_S (f(k_n))$; see Figure \ref{fig:thin_hulls}.

For any point $a_n \in \pi_S (f(k_n))$, we have

\begin{align*}
    d_S(a_n,h''(t_n))&\leq d_S(a_n,p_n')+d_S(p_n', \alpha_n)+|\alpha_n|\\
    &\leq K'+N+K'\\
    &\leq R+N+R.
\end{align*}

Using our choice of $R$, we have $d(a_n, h''(t_n)) \leq 3R$ for any point $a_n \in \pi_S(f(k_n)).$ Using the triangle inequality, we get

\begin{align*}
    d_S(a_n, a_{n+1})&\geq d_S(h''(t_n), h''(t_{n+1}))-d_S(a_n, h''(t_n))-d_S(a_{n+1}, h''(t_{n+1}))\\
    &\geq 7R-3R-3R\\
    &=R,
\end{align*}
for any $a_n \in \pi_S(f(k_n)) $ and any $a_{n+1} \in \pi_S(f(k_{n+1}))$. Hence, $d_S(\pi_S(f(k_n)), \pi_S(f(k_{n+1}))) \geq R \geq C.$ Applying Lemma \ref{lem: separation in HHS's}, we get that $k_n, k_{n+1}$ are $L$-well-separated where $L$ depends only on $\calX$ and $|F|.$ Therefore, the quasi-geodesic ray $h':[0,\infty) \rightarrow Q$ crosses an infinite sequence of hyperplanes $\{k_n\}_{n \in \mathbb{N}}$ at times $s_n \in [t_n, t_{n+1}]$ such that $k_n,k_{n+1}$ are $L$-well-separated and

\begin{align*}
|s_n - s_{n+1}| &\leq |t_n -t_{n+1}|+|t_{n+1} -t_{n+2}|\\
&\leq (7R+D)D\kappa^p(t_{n+1})+(7R+D)D\kappa^p(t_{n+2}).
\end{align*}

Since $|t_{n+1} - t_{n+2}|=D(7R+D)\kappa^p(t_{n+2})$, using Lemma \ref{lem: relating sublinearness}, there exists a constant $A$, depending only on $\kappa$ and $D(7R+D)$ such that $d(t_{n+1}, t_{n+2}) \leq A \kappa^p(t_{n+1}).$ 

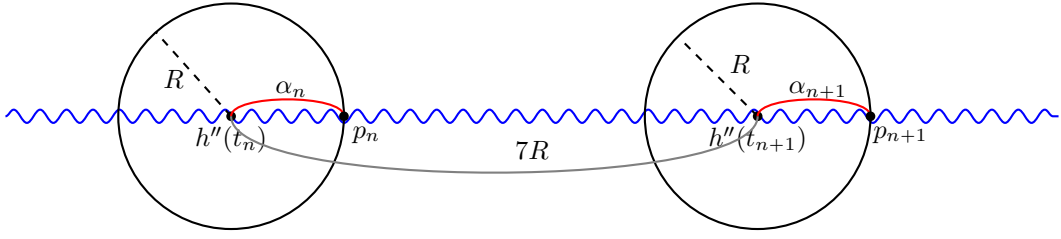
\begin{figure}
    \centering
   \begin{tikzpicture}[thick]
  \path [draw=blue,snake it]
    (-7,0) -- (7,0);
\draw[thick,black] (-4,0) circle (1.5cm);
\draw[dashed] (-4,0) -- (-5,1.12);
\draw[thick,fill=black] (-4,0) circle (0.05cm);
\node[below] at (-4,0) {$h''(t_{n})$};

\draw[thick,black] (3,0) circle (1.5cm);
\draw[dashed] (3,0) -- (2,1);
\node[right] at (2.5,.7) {$R$};

\draw[thick,fill=black] (3,0) circle (0.05cm);
\node[below] at (3,0) {$h''(t_{n+1})$};

\node[left] at (-4.5,.5) {$R$};
\draw[thick,red] (-4,0) .. controls ++(0,.3) and ++(0,.3) ..(-2.5,0);
\node[above] at (-3.2,.12) {$\alpha_n$};

\draw[thick,red] (3,0) .. controls ++(0,.3) and ++(0,.3) ..(4.5,0);
\node[above] at (3.8,.1) {$\alpha_{n+1}$};

\draw[thick,gray] (3,0) .. controls ++(0,-1) and ++(0,-1) ..(-4,0);

\node[below] at (0,-.2) {$7R$};

\draw[thick,fill=black] (-2.5,0) circle (0.05cm);
\node[below] at (-2.2,0) {$p_n$};

\draw[thick,fill=black] (4.5,0) circle (0.05cm);
\node[below] at (4.9,0) {$p_{n+1}$};

\end{tikzpicture}    \caption{The image of the hyperplane $k_n$ given by $\pi_S(f(k_n))$ lies between $h''(t_n)$ and $p_n$.}
    \label{fig:sequence_of_hyperplabes}
\end{figure}

This gives that 

\begin{align*} |s_n- s_{n+1}| &\leq |t_n -t_{n+1}|+|t_{n+1} - t_{n+2}|\\
&\leq (7R+D)D\kappa^p(t_{n+1})+A\kappa^p(t_{n+1})\\
&=((7R+D)D+A) \kappa^p(t_{n+1})\\ 
&\leq ((7R+D)D+A) \kappa^p(s_{n+1}).
\end{align*}

Using Lemma \ref{lem: relating sublinearness}, we have 
$$|s_n- s_{n+1}| \leq A' \kappa^p(s_{n}),$$

where $A'$ depends only on $\kappa,R,D,|F|$ and $A.$ This concludes the proof of the statement.
\end{proof}

In fact, the proof of the above lemma shows the following more general statement, which is Theorem \ref{thm:curves and cubes intro} from the introduction:

\begin{corollary}\label{cor:curves and cubes}  Let $\calX$ be an HHS, $h:[0,\infty) \rightarrow \calX$ be a median ray with $\diam_S(\pi_S(h))=\infty$ and let $F$ be a finite set of points in $\calX.$ If $(g,f,Q \rightarrow H)$ is a cubical model for $H=\hull(h \cup F)$, then we have the following. There exists a constant $C$, depending only on $\calX$ and $|F|$ such that for any $a,b \in h$ if $a''=\pi_S(a), b''=\pi_S(b)$ satisfy $d_S(a'',b'') \geq nC,$ then the points points $a'=g(a), b'=g(b)$ are separated by a collection of $n$ hyperplanes which are pairwise $L$-well-separated, where $L$ depends only on $\calX$ and $|F|.$

\end{corollary}

\begin{proof}
We let $h',h''$ be in Notation \ref{notation: relating maps} and we let $q$ be the image of the unparamaterized quasi-geodesic $h''$. Let $a,m,b \in h$, and $a''=\pi_S(a), m''=\pi_S(m), b''=\pi_S(b)$ be with $a''<m''<b''$ appearing along $h''$ in that order. Choosing $R$ exactly the same as in Lemma \ref{lem:bounded projections imply excursion} shows that if $d_S(a'',m'')>7R$, and $d_S(m'',b'')>7R,$ then $a'=g(a), b'=g(b)$ are separated by a pair of $L$-well-separated hyperplanes where $L$ depends only on $\calX$ and $|F|.$ This gives the desired statement.

\end{proof}

We can now prove item (2) of Theorem \ref{thm:bounded projection characterization} and item (1) of Theorem \ref{thm:characterization via speed}:

\begin{corollary}
Let $h$ be a median ray with $\kappa$-bounded projection for a $p$-stable sublinear function $\kappa.$ We have the following:

\begin{enumerate}
    \item $\hull(h)$ is $\kappa^p$-contracting.
    
    \item $h$ is $\kappa^p$-contracting.

\end{enumerate}
\end{corollary}

\begin{proof}
Let $h$ be as in the statement and let $x,y \in \calX.$ Since $h$ has $\kappa$-bounded projections, Corollary \ref{cor: bounded projections imply large distance} gives us that $\diam_{\calC(S)}(\pi_S(h))= \infty.$ Let $h', h,Q, H,f,g$ be a in Notation \ref{notation: relating maps} and let $Z=\hull_{\calX}(h)$. Choose a constant $D$ small enough so that for any $a \in H,$ if $b$ satisfies $d_{\calX}(a,b) \leq Dd_{\calX}(a,Z),$ then $d_{Q}(a',b') \leq d_Q(a', Z')$, where $Z'=\hull_Q(h')$, $a'=g(a)$, $b'=g(b)$. Such a choice is possible since $f:Q \rightarrow H$ is a quasi-isometry with constants depending only on $\calX$.

Using Lemma \ref{lem:bounded projections imply excursion}, we know that $h'$ crosses an infinite sequence of $\kappa^p$-excursion hyperplanes $\{k_i\}_{i \in \mathbb{N}}$ with an excursion constant $c$. As noted above, the constant $c$ and $d_{Q}(g(\go), k_1)$ are both independent of $x,y.$ Using Theorem \ref{thm:hyperplane characterization for quasi-geodesic rays}, the ray $h'$ is $\kappa^p$-Morse and hence the set $Z'=\hull(h')$ is $\kappa^p$-contracting in $Q$ with respect to the combinatorial projection by Proposition \ref{prop:contracting hulls, CAT(0)}. In other words, there exists  a constant $c',$ depending only on $\calX$ and $c,$ such that if $a,b \in Q$ satisfy $d_Q(a,b) \leq d_Q(a,P_{Z'}(a)),$ then $d_Q(P_{Z'}(a), P_{Z'}(b)) \leq c' \kappa^p(a)$. In particular, if $x'=g(x), y'=g(y)$ we have $d_Q(P_{Z'}(x'), P_{Z'}(y')) \leq c' \kappa^p(x')$ provided that $d_Q(x',y') \leq d_Q(x',P_{Z'}(x')).$ Corollary \ref{cor: comparing gates} gives us that $d_{\calX}(\gate_{Z}(x), \gate_Z(y)) \leq c'' \kappa^p(x)$ for a constant $c''$ depending only on $\calX$ and $c.$ This shows that $\hull(h)$ is $\kappa^p$-contracting. Remark \ref{rmk:technical_conclusion} gives us that $h$ is also $\kappa^p$-contracting.

\end{proof}

\end{proof}

\subsection{A comparison with combinatorial bounded projections}

In \cite{QRT20}, the authors introduced a combinatorial version of the $\kappa$-bounded projection. While they apparently were not trying to characterize $\kappa$-Morseness via this condition, it is reasonable to wonder if it does.  The goal of this subsection is to show that such a condition does not characterize $\kappa$-Morseness.

\begin{definition}[Definition 7.6 of {QRT20}] Let $\kappa$ be a sublinear function. A median ray $h$ is said to have \emph{$\kappa$-bounded combinatorial projections} if there exists a constant $C \geq 0$ such that for any proper $U \nest S,$ we have

$$\diam_U(h) \leq C \kappa(d_S(\go, \rho_S^U)).$$

\end{definition}

We note that if $h$ has $\kappa$-bounced combinatorial projections and $h'$ is another median ray with $\pi_S(h) \sim \pi_S(h')$ in $\partial \calC(S)$, then $h'$ has $\kappa$-bounded combinatorial projections by a standard argument using the bounded geodesic image axiom (Axiom \ref{BGIA}).

The authors in \cite{QRT20} prove that, in the mapping class group case, there exists an integer $p$ depending only on the surface such that if a median ray $h$ has $\kappa$-bounded combinatorial projections for a sublinear function $\kappa$ where $\kappa^p$ is sublinear, then $h$ is $\kappa^p$-Morse. Recently, this was extended to all hierarchically hyperbolic groups in \cite{QN22}.  Below, we give an example to demonstrate that this does not yield a characterization for sublinear Morseness. Namely, we will provide an example of a sublinearly Morse median ray $h$ which does not have $\kappa$-bounded combinatorial projections for any sublinear function $\kappa.$

\begin{proposition} \label{prop:counter Z}
There exists $\kappa$ sublinear and a $\kappa$-Morse geodesic ray $h$ in $\mathbb Z * \mathbb Z^2$ which does not have $\kappa'$-bounded combinatorial projections for any sublinear $\kappa'$.
\end{proposition}

\begin{proof}
Let $G=\langle a,b,c| [a,b] \rangle $ and consider the geodesic ray $h:[0,\infty) \rightarrow Cay(G,\{a,b,c\})$ defined as

$$caca^2ca^3\cdots ca^n \cdots,$$ and let $t_n$ be such that $h(t_n)=caca^2\cdots ca^n.$ The HHS given by $\calX=Cay(G,\{a,b,c\})$ has the contact graph $\calC(S)$ as it's maximal domain. Notice that $t_n=|h(t_n)|=n+(1+2+\cdots n) \sim n^2,$ therefore,  $h$ is $\sqrt{n}$-Morse. Further, since the distance made in the contact graph is coarsely the number of $c$'s appearing between $h(0)$ and $h(t_n),$ we have $d_S(h(0), h(t_n)) \geq A n$ for a constant $A$ depending only on $\calX.$ On the other hand, for each $t_n,$ if $U_n$ denotes the horizontal axis of the unique flat at the point $h(t_n)$, then we have 

\begin{itemize}
    \item $\diam_{U_n}(h)=d_U(h(0), h(t_n))=n,$ and
    \item $d_S(h(0), h(t_n)) \sim d_S(\go, \rho^{U_n}_S) \sim n.$
\end{itemize}

Using the two above two items, since $\diam_{U_n}(h)=n$ and $d_S(\go, \rho^{U_n}_S) \sim n,$ there exists no sublinear function $\kappa$ with 
$\diam_{U_n}(h) \leq C \kappa(d_S(\go, \rho_S^U))$.

\end{proof}

In order to produce a counterexample in $\Mod(S)$, it will be enough to produce a sufficiently nice embedding of $\mathbb Z * \mathbb Z^2$ into $\Mod(S)$.  This can be done using work of Runnels \cite{Runnels}.

Let $\phi \in \Mod(S)$ be pseudo-Anosov, and suppose that $\alpha, \beta$ are disjoint curves lying within $1$ of the quasi-axis for $\phi$ in $\calC(S)$.  By \cite[Theorem 4.1]{Runnels}, there exists $N>0$ so that $H = \langle \phi^N, T_{\alpha}^N, T_{\beta}^N\rangle < \Mod(S)$ is a quasi-isometrically embedded copy of $\mathbb Z * \mathbb Z^2$, where $T_{\alpha}, T_{\beta}$ are Dehn-twists around $\alpha, \beta$, respectively.

Moreover, it is not hard to show that the inclusion map $H \to MCG(S)$ induces a nice embedding between the HHS structure on $H$ \cite{HHS_I}, denoted $(H, \mathfrak S_H)$ and the standard HHS structure on $\Mod(S)$.  In particular, it is a \emph{hieromorphism} in the sense of \cite{HHS_II, DHS17}, with the following useful features:
\begin{itemize}
    \item There is an $H$-equivariant injection $i_{\mathfrak S}:\mathfrak S_H \to \mathfrak S$ at the level of domains, and
    \item For each $U \in \mathfrak S_H$, there is a $\mathrm{Stab}_H(U)$-equivariant uniform quasi-isometric embedding $\calC(U) \to \calC(i_{\mathfrak S}(U))$.
\end{itemize}

In particular, the contact graph of $H$ equivariantly quasi-isometrically embeds into the curve graph $\calC(S)$, and the product regions of $H$ quasi-isometrically embed into product regions of $\Mod(S)$.

The details of these statements and the following proposition are straight-forward but somewhat tedious to check, so we leave them to the interested reader.

\begin{proposition}\label{prop:counter MCG}
The geodesic ray $h$ in $H$ from Proposition \ref{prop:counter Z} determines a $\kappa$-Morse quasi-geodesic ray in $\Mod(S)$ which does not have $\kappa'$-bounded combinatorial projections for any sublinear $\kappa$'.
\end{proposition}

\section{Sublinearly Morse Teichm\"uller geodesics}

In this section, we apply some of the results above to make some conclusions about the behavior of $\kappa$-Morse Teichm\"uller geodesics.  Specifically, we will prove Theorem \ref{thm:Teich geo intro} from the introduction.

\subsection{Background and supporting statements}

Fix a finite-type surface $S$ which admits a hyperbolic metric, and let $\xi(S) = 3g-3+n <0$, where $g$ is the genus and $n$ is the number of punctures, or equivalently the number of simple closed curves in any pants decomposition of $S$.  Let $\Teich(S)$ be its Teichm\"uller space equipped with the Teichm\"uller metric, and $\Mod(S)$ its mapping class group.

Recall that for any point $\sigma \in \Teich(S)$, there is a \emph{Bers pants decomposition} of $\sigma$, which is a pants decomposition with the property that the hyperbolic lengths of the curves are uniformly bounded in terms of $S$ (so independent of $\sigma$).  This defines a map $\calB_S:\Teich(S) \to \calC(S)$, and composing further with Masur-Minsky-style subsurface projections \cite{MM00} gives maps $\calB_Y = \pi_Y \circ \calB_S:\Teich(S) \to \calC(S)$.  Note that if $\mathrm{sys}(\sigma)$ is a shortest curve on $\sigma$ (in hyperbolic length), then $d_S(\calB_S(\sigma), \mathrm{sys}(\sigma))$ is uniformly bounded.  Moreover, note that we could have taken the extremal length systole, since extremal and hyperbolic lengths are comparable for short curves.

The following theorem combines work of Masur-Minsky \cite{MM99} and Rafi \cite{Raf14}:

\begin{theorem}\label{thm:quasi shadows}
For any Teichm\"uller geodesic $\gamma$, the projection $\calB_Y(\gamma)$ of $\gamma$ to $\calC(Y)$ for any subsurface $Y \subset S$ is a unparametrized uniform quasi-geodesic.
\end{theorem}

We note that in the HHS structure on $\Teich(S)$ with the Teichm\"uller metric \cite{Dur16}, the set of hyperbolic spaces consists of the standard curve graphs $\calC(Y)$ for nonannular $Y \subset S$, but for annuli $\alpha$, we replace $\calC(\alpha)$ a combinatorial horoball $\mathcal H_{\alpha}$ over $\calC(\alpha)$; see \cite{Dur16, GM08}.  These horoballs are quasi-isometric to a horoball in $\mathbb H^2$ and are equipped with a projection $\pi_{\calH_{\alpha}}:\calX \to \calH_{\alpha}$ where
\begin{itemize}
    \item The vertical coordinate of $\pi_{\calH_{\alpha}}(\sigma)$ is $\frac{1}{e^{\mathrm{Ext}_{\sigma}(\alpha)}}$, where $\mathrm{Ext}_{\sigma}(\alpha)$ is the \emph{extremal length} of $\alpha$ in the metric $\sigma$, and
    \item The horizontal coordinate of $\pi_{\calH_{\alpha}}(\sigma)$ is $\calB_{\alpha}(\sigma)$, thereby encoding relative twisting around $\alpha$.
\end{itemize}
Notably, Theorem \ref{thm:quasi shadows} does not hold for the projections of Teichm\"uller geodesics to the $\calH(\alpha)$, as curves can become short along a Teichm\"uller geodesic even when their relative twisting is bounded, thereby causing unbounded backtracking in the vertical direction.  See \cite{Raf05,Modami_short} for an analysis of short curves along Teichm\"uller geodesics.

In particular, Teichm\"uller rays are \emph{not} median rays.  However, Theorem \ref{thm:quasi shadows} will be enough for our purposes.

\subsection{Characterizing sublinear Morseness for Teichm\"uller geodesics}

We first prove the following characterization theorem:

\begin{theorem}\label{thm:Teich statements}
There exists $p = p(S)>0$ so that for any sublinear function $\kappa$ the following holds:

\begin{enumerate}
    \item If $\gamma$ is a $\kappa$-Morse Teichm\"uller ray, then $\gamma$ has $\kappa$-bounded projections and a $\kappa^p$-persistent shadow.
    \item If $\gamma$ is a Teichm\"uller ray with $\kappa$-bounded projections and $\kappa^{2p}$ is sublinear, then $\gamma$ is $\kappa^{2p}$-Morse.
    \item If $\gamma$ is a Teichm\"uller ray with a $\kappa$-persistent shadow and $\kappa^{p+1}$ is sublinear, then $\gamma$ is $\kappa^{p+1}$-Morse.
\end{enumerate}
\end{theorem}

Note that item (1) follows immediately from Theorem \ref{thm:bounded projection characterization}.  Hence it suffices to prove items (2) and (3), which we accomplish by proving the following slightly more general theorem.

Let $\calU^{\text{bot}}$ denote the collection of all domains in $\mathfrak S$ such that if $U \in \mathfrak S$ and $U \sqsubseteq V$ for some $V \in \calU^{\text{bot}}$, then $U=V.$

\begin{theorem}\label{thm:teich backwards}
Let $\calX$ be a proper HHS with unbounded products.  Let $\gamma$ be a quasi-geodesic ray in $\calX$ such that $\pi_U(\gamma)$ is an $(A,A)$-unparameterized quasi-geodesic for all $U \notin \calU^{\text{bot}}$ and for some $A \geq 1$.
\begin{enumerate}
    \item If $\gamma$ has $\kappa$-bounded projections and $\kappa^{2p}$ is sublinear, then $\gamma$ is $\kappa^{2p}$-Morse.
    \item If $\gamma$ has a $\kappa$-persistent shadow and $\kappa^p$ is sublinear, then $\gamma$ is $\kappa^p$-Morse.
\end{enumerate}

\end{theorem}

\begin{proof}
Let $E$ be the HHS constant and let $b_n$ be a sequence of median paths starting at $\gamma(0)=\go$ and ending on $\gamma(n)$. Up to passing to a subsequence, the sequence $b_n$ converges to a median ray $b$. Let $H=\hull(b)$ and $\gate_H:\calX \to H$ be the gate map. Notice that by the definition of $b$ and hyperbolicity of $\calC(S)$ (where $S$ is the $\nest$-maximal domain),  the set $\pi_S(b)$ coarsely agrees with $\pi_S(\gamma)$, and hence the set $H'':=\pi_S(H \cup \gamma)$ is coarsely a line, with constants depending only on $\calX$.

The proof consists of the following steps:

\begin{enumerate}
    \item We show that there exists a constant $\dd$ such that for each $t \geq0,$ we have $d_\calX(\gamma(t), \gate_H(\gamma(t)) \leq \dd \kappa(t)$. In particular, $\gamma$ is in some $\kappa$-neighborhood of $H$. This is Claim \ref{clm:kappa-nbhd}.
    
    \item Using item (1) and the cubical model for $H$, we construct a median ray $h \in H$ such that $h \sim_{\kappa^p} \gamma.$ This is Claim \ref{clm:median_ray}.
    
    \item If $\gamma$ has $\kappa$-bounded projections, and $\kappa^p \geq \kappa,$ the quasi-geodesic ray $h$ must have $\kappa^p$-bounded projections.  Alternatively, if $\gamma$ has a $\kappa$-persistent shadow, then $h$ has a $\kappa^p$-persistent shadow.

    \item As $h \sim_{\kappa^p} \gamma$ and $\gamma$ has $\kappa^p$-bounded projections, Lemma \ref{lem:relating_sublinear_projections} gives us that $h$ has $\kappa^p$-bounded projections and hence $h$ must be $\kappa^{2p}$-Morse by Theorem \ref{thm:bounded projection characterization}. Since $h \sim_{\kappa^p} \gamma$, and $h$ is $\kappa^{2p}$-Morse, $\gamma$ must also be $\kappa^{2p}$-Morse.

\end{enumerate}

Thus, in order to finish the proof, we need to prove Claim \ref{clm:kappa-nbhd} and Claim \ref{clm:median_ray}.

\begin{claim}\label{clm:kappa-nbhd} There exists constant $\dd$ such that for each $t \geq0,$ we have $d_\calX(\gamma(t), \gate_H(\gamma(t)) \leq \dd \kappa(t)$. In particular, $\gamma$ is in some $\kappa$-neighborhood of $H$

\end{claim}
\begin{proof}[Proof of Claim]

This proof is essentially a ``passing up'' argument, using Lemma \ref{lem:passing-up}.  In particular, the existence of a linear-number of large bottom-level domains will force the existence of a linear-sized higher-level domain, contradicting the $\kappa$-bounded assumption.

To see the claim, notice that there exists a constant $C \geq E$, depending only on the HHS constant and on $A$, such that $d_W(\gamma(t),\gate_H(\gamma(t))<C,$ for all $W \notin \calU^{\text{bot}}$ and all $t \geq 0.$ On the other hand, since $\gamma$ has $\kappa$-bounded projections, there exists a constant $c$ such that for $U \in \calU^{\text{bot}}$, we have 

\begin{align*}
d_U(\gamma(t), \gate_H(\gamma(t))) &\leq d_U(\gamma(t), \go)\\
&\leq c \kappa(t),
\end{align*}

for all $t \geq 0.$ Applying the passing up Lemma \ref{lem:passing-up} with $V=S$, yields a constant $m=m(C)$ such that if $d_{U_i}(\gamma(t), \gate_H(\gamma(t))>E$ whenever $\{U_i\}_{i=1}^m$ are in $\calU^{bot},$ then there exists a domain $W$ with $U_i \sqsubsetneq W$ for some $i$ such that $d_{U_i}(\gamma(t), \gate_H(\gamma(t))>C.$ Hence, at most $m$ domains $U_i \in \calU ^{bot}$ can satisfy $d_{U_i}(\gamma(t), \gate_H(\gamma(t)))>E.$ Now, applying the distance formula (Theorem \ref{thm:distance formula}) with threshold $\theta=\text{max}\{\theta_0,C\}$ gives a constant $K$ such that $$d_\calX(\gamma(t),\gate_H(\gamma(t))) \underset{K}{\asymp}\underset{U \in \mathrm{Rel}_{\theta}(\gamma(t),\gate_H(\gamma(t)))}{\sum} d_U(\gamma(t),\gate_H(\gamma(t))).$$ Using the above, the set $\mathrm{Rel}_{\theta}(\gamma(t),\gate_H(\gamma(t)))$ contains at most $m$ elements, where $m$ depends only $C$ which depends only on HHS constant and $A$. That is, there exists $n \leq m$ such that

$$d_\calX(\gamma(t),\gate_H(\gamma(t))) \underset{K}{\asymp} \sum_{i=1}^n d_{U_i}(\gamma(t),\gate_H(\gamma(t))) \leq nc \kappa(t) \leq mc \kappa(t).$$

Hence, there exists a constant $\dd$, depending only on the HHS constant and $A$ such that $d_\calX(\gamma(t), \gate_H(\gamma(t))) \leq \dd \kappa(t)$ which concludes the proof of the first claim.

\end{proof}
\begin{claim}\label{clm:median_ray}
There exists a median ray $h \in H$ such that $h \sim_{\kappa^p} \gamma$.
\end{claim}

\begin{proof}[Proof of Claim]

The proof of this claim uses a slight variation on the ideas in Subsection \ref{subsec:backwards direction HHS}.  In particular, the assumptions force the median representative of $\gamma$ to cross a sequence of hyperplane-like subspaces in the hull, which transport to hyperplane-like convex subsets in the cubical model of its hull.  The assumptions will give that these convex subsets are $\kappa^p$-well-separated, and moreover that $\gamma$ passes within $\kappa^p$ of each of them.  The claim follows quickly once those facts are established.

To see the claim, observe that if $\gamma$ has $\kappa$-bounded projections, Corollary \ref{cor: bounded projections imply large distance} gives us that

$$d_S(\gamma(s), \gamma(t)) \underset{D}{\succ} \frac{t-s}{\kappa^p(t)},$$ for a constant $D$ depending only on the projection constant, the quasi-geodesic constants of $\gamma$ and $E.$

If alternatively $\gamma$ has a $\kappa$-persistent shadow, then we can take $\kappa$ instead of $\kappa^p$ in the above inequality.  We will proceed with the $\kappa$-bounded projections assumption and comment on the difference at the end.

Choose a sequence of balls $B_i''$ along $\pi_S(\gamma)$ as follows:

\begin{enumerate}
    \item Each $B_i''$ is of a large enough radius $R$, depending only on $E$, so that each $B_i''$ disconnects $H''=\pi_S(H \cup \gamma)$. This is possible since $H''$ is coarsely a line. In particular, we insist that $R \geq E.$
    
    \item Each $B_i''$ is centered at $\pi_S(\gamma(t_i))$ and satisfies $d_S(B_i'',B_{i-1}'')=K$ where $K \geq R$ is as in Corollary \ref{cor: far in the top level implies small gates}. In particular, if we define $B_i:=H \cap \hull_\calX(\pi^{-1}_S(B_i''))$ we get that $\diam_\calX(\gate_{B_i}(B_{i-1})) \leq K'$ where $K'$ depends only only $E.$ Similarly, $\diam_\calX(\gate_{B_{i-1}}(B_{i})) \leq K'$.
\end{enumerate}

See Figure \ref{fig:Teich_char_2}.  Note that the $B_i \subset \calX$ are $M$-median convex for $M$ depending only on $E$.  Our choice of $B_i''$ gives us that 

$$
    2K \geq d_S(\gamma(t_{i}), \gamma(t_{i-1})) \underset{D}{\succ} \frac{t_{i}-t_{i-1}}{\kappa^p(t_{i})}.$$
Thus, we have $t_i-t_{i-1} \leq D'' \kappa^p(t_i)$ for a constant $D''$ depending only on the projection constant, $E$ and the quasi-geodesic constants of $\gamma.$ Since $\gamma$ is a quasi-geodesic, we get that $d_{\calX}(\gamma(t_i), \gamma(t_{i-1})) \leq C \kappa^p(t_i)$ where $C$ depends only on $D''$ and the quasi-geodesic constants of $\gamma$. Combining this statement with Claim \ref{clm:kappa-nbhd}, we get that 

\begin{align*}
    d_{\calX}(\gate_H(\gamma(t_i)), \gate_H(\gamma(t_{i-1})) )&\leq   d_{\calX}(\gate_H(\gamma(t_i)), \gamma(t_i))+d_{\calX}(\gamma(t_i), \gamma(t_{i-1}))+d(\gamma(t_{i-1}), \gate_H(\gamma(t_{i-1})))\\
    &\leq \dd \kappa^p(t_i)+C \kappa^p(t_i)+\dd \kappa^p(t_{i-1})\\
    & \leq (2\dd+C)\kappa^p(t_i),
\end{align*} where $\dd$ is as in Claim \ref{clm:kappa-nbhd}.

Set $B_i':=\hull_{Q}(g(B_i))$ in $Q,$ where $(g,f,Q \rightarrow H)$ is the cubical model for $H$. Notice that since $B_i$ is $M$-median convex, with $M$ depending only on $E,$ the set $B_i'=g(B_i)$ is  $M'$-median convex where $M'$ depends only on $E$ and the quasi-constants of $g$, which only depend on $\calX$ and $2$ via Theorem \ref{thm:limiting model}.

Therefore, $B_i'=\hull_Q(f^{-1}(B_i)) \subset Q$ is a combinatorially convex set within Hausdorff distance $E'$ of $f^{-1}(B_i)$ for $E'$ depending only on $E$. Combining this fact with Corollary \ref{cor: comparing gates} gives us that $\diam_Q(P_{B_i'}(B'_{i-1})) \leq K''$ and $\diam_Q(P_{B_{i-1}'}(B_{i}')) \leq K''$ where $K''$ depends only on $E$. Applying Proposition \ref{prop:Genevois} gives us that the combinatorially convex sets $B_{i-1}',B_i'$ are $L$-well-separated for each $i$ where the constant $L$ depends only on $E.$  See Figure \ref{fig:Teich_char_2}.

\begin{figure}
    \centering
    \includegraphics[width=1\textwidth]{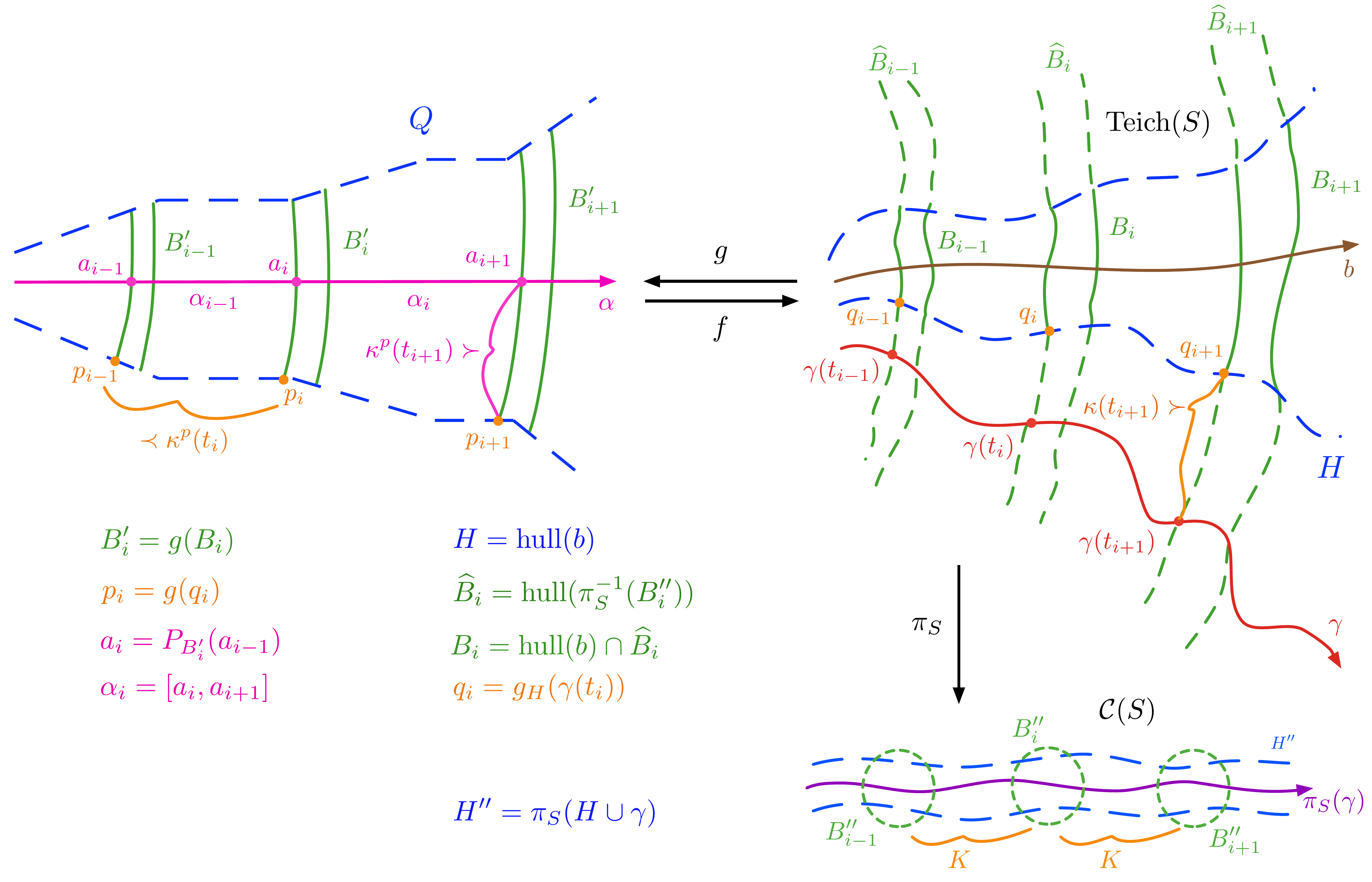}
 
    \caption{Geometry of the cubical model for $H= \hull(b)$ via the curve graph. The set $H''=\pi_S(H \cup \gamma) \subset \calC(S)$ is coarsely a line by definition of $H.$ The balls $B_i''$ are centered on $\pi_S(\gamma)$ and each has a large enough radius to disconnect $H''.$  The median hulls of their preimages $\widehat{B}_i \subset \Teich(S)$ restrict to median convex subsets $B_i$ of the hull.  These determine convex subsets $B_i'$ in the cubical model $Q$ of $H$, which may be thought of as ``thick hyperplanes" as they are combinatorially convex and separate $Q$.\\
    To build the geodesic ray $\alpha \subset Q$, we concatenate geodesic segments $\alpha_i$ between the relative gates $a_i$ of the $B'_i$.  The points $\gamma(t_i)$ are $\dd\kappa(t_i)$-close to their median gate images $q_i = g_H(\gamma(t_i))$ on $H$, while the images of the latter $p_i = g(q_i)$ are $\kappa^p$-close to the $a_i$.} \label{fig:Teich_char_2}
\end{figure}

Since $g$ is a quasi-isometry and since  $d_{\calX}(\gate_H(\gamma(t_i)), \gate_H(\gamma(t_{i-1})) ) \leq  (2\dd+C)\kappa^p(t_i)$, if we define $p_i:=g(\gate_H(\gamma(t_i))$, then we have $d_Q(p_i,p_{i-1}) \leq F \kappa^p(t_i)$ where $F$ depends only on $E,A$ and the quasi-geodesic constants of $\gamma.$

We construct a combinatorial geodesic ray in $Q$ inductively as follows. Let $a_0:=g(\gamma(0)).$   We define $a_i=P_{B_i'}(a_{i-1})$ and we let $\alpha_i$ be a combinatorial geodesic connecting $a_{i-1}, a_i$.  We claim that $\displaystyle \alpha= \underset{i \in \mathbb{N}}{ \bigcup \alpha_i}$ is a geodesic.  To show this, we need only to show that no hyperplane can cross $\alpha$ twice.

Let $J$ be a hyperplane that crosses $\alpha,$ say at $\alpha_j$ for some $j$, hence, $J$ separates $a_{j-1}, a_j$. Consequently, using the characterization of the combinatorial projection in Definition \ref{def:comb_proj}, 
the hyperplane $J$ cannot cross $B_j$ and hence it cannot meet $\displaystyle \alpha= \underset{i> j}{ \bigcup \alpha_i}$. The hyperplane $J$ also cannot meet any $\alpha_i$ for any $i  \leq j-1$, for if it does, then $J$ crosses $B_i'$ while also separating $a_{i-1}$ from $a_i=P_{B_i'}(a_{i-1})$, violating violating the characterization of the gate map $P_{B_i}$ in Definition \ref{def:comb_proj}.

Hence $\alpha$ is a combinatorial geodesic. Notice that by the respective definitions of $p_i,a_i$, both are contained in $B_i'$ for all $i \geq 1.$ We will now show that $d_Q(p_i, a_i)$ is bounded above by a multiple of $\kappa^p(t_i).$ 

For a fixed $i \geq 1,$ we consider the collection of hyperplanes $\calH_i$ which separate $p_i,a_i \in \alpha.$   Such a collection is one with no facing triples as it separates $p_i,a_i$. Also, every hyperplane $J \in \mathcal{H}_i$ meets $B_i'$ as $p_i,a_i \in B_i'$. Since $B_n,B_{n-1}$ are $L$-well-separated for all $n \in \mathbb{N}$, the number of hyperplanes $J \in \calH_i$ which meet $B_{i-1}$ or $B_{i+1}$ is bounded above by $2L.$ Further, each hyperplane $J \in \calH_i$ which crosses neither $B_{i-1}$ nor $B_{i+1}$ must separate $p_{i-1},p_i$ or $p_i,p_{i+1}.$ Since $d_Q(p_n,p_{n-1}) \leq F \kappa^p(t_n)$ for all $n \in \mathbb{N},$ we have 

\begin{align*}
d_Q(p_i,a_i)&=|\calH_i| \\
&\leq 2L+F \kappa^p(t_i)+F\kappa^p(t_{i+1})\\
&\leq 2L\kappa^p(t_{i+1})+F \kappa^p(t_{i+1})+F\kappa^p(t_{i+1})\\
&\leq (2L+2F) \kappa^p(t_{i+1}),
\end{align*} where the last two inequalities hold since $\kappa^p \geq 1$ is a non-decreasing function. Using Lemma \ref{lem: relating sublinearness}, we get that $d_Q(p_i,a_i)=|\calH_i| \leq F' \kappa^p(t_i)$ where $F'$ depends only on $F,L$ and $\kappa^p.$

Since $f:Q \rightarrow H$ is a quasimedian quasi-isometry and $\alpha$ is a combinatorial geodesic ray, we get that $h:=f(\alpha)$ is a median ray and $d_Q(f(p_i),f(a_i)) \leq F'' \kappa^p(t_i).$ Hence, using the definition of $p_i,$ we get $d_{\calX}(\gate_H(\gamma(t_i)),f(a_i)) \leq G \kappa^p(t_i)$, for $G$ depending only on $F'$ and the quasi-constants of $f$. Combining this with Claim \ref{clm:kappa-nbhd}, we get 

\begin{align*}
d_\calX (\gamma(t_i), f(a_i)) & \leq d_\calX (\gamma(t_i), \gate_H(\gamma(t_i)))+d_ \calX( \gate_H(\gamma(t_i)), f(a_i))\\
&\leq \dd \kappa(t_i)^p+G \kappa^p(t_i)\\
&=(\dd+G)\kappa^p(t_i).
\end{align*}

This shows that the sequence $\{\gamma(t_i)\}_i$ is in some $\kappa^p$-neighborhood of $h=f(\alpha)$. However, since $d_\calX( \gamma(t_{i-1}),\gamma(t_i)) \leq C\kappa^p(t_i)$, the triangle inequality gives us that $\gamma$ is in some $\kappa^p$-neighborhood of $h=f(\alpha).$ Thus, using Lemma 3.1 of \cite{QRT20}, the ray $h$ is also in some $\kappa$-neighborhood of $\gamma$. Thus, $\gamma \sim_{\kappa^p} h$ which concludes the proof of the claim.

\end{proof}

In the case that $\gamma$ has $\kappa$-bounded projections, we can use Lemma \ref{lem:relating_sublinear_projections} to conclude that $h$ has $\kappa^p$-bounded projections and hence $h$ must be $\kappa^{2p}$-Morse by Theorem \ref{thm:bounded projection characterization}. Since $h \sim_ {\kappa^p} \gamma,$ the quasi-geodesic ray $\gamma$ must also be $\kappa^{2p}$-Morse.

Alternatively, if $\gamma$ has a $\kappa$-persistent shadow, then $\gamma \sim_{\kappa^p} h$ and so $h$ has a $\kappa^{p+1}$-persistent shadow by Lemma \ref{lem:relating shadows}.  Hence $h$ is $\kappa^{p+1}$-Morse by Theorem \ref{thm:bounded projection characterization}.  This completes the proof.

\end{proof}

\bibliography{sublinearMCG}{}
\bibliographystyle{alpha}

\end{document}